\documentclass[a4paper,11pt]{amsart}
\usepackage{amsmath,amsthm,amssymb,amsfonts,enumerate,color,esint,bbm}
\usepackage[pdftex]{graphicx}
\usepackage{float}
\usepackage{tikz}
\usepackage[shortlabels]{enumitem}
\usepackage{subfig}
\usepackage[colorlinks=true, allcolors=blue]{hyperref}
\usepackage{amsrefs}
\usepackage{pgfplots}
\pgfplotsset{
compat=newest,
colormap={blackwhite}{gray(0cm)=(.25); gray(1cm)=(1)}
}
\usepackage{subfig}
\usetikzlibrary{hobby}

\newcommand{\N}{\mathbb{N}}
\newcommand{\R}{\mathbb{R}}
\newcommand{\Z}{\mathbb{Z}}

\newcommand{\abs}[1]{\left\vert#1\right\vert}
\def\({\left(}
\def\){\right)}
\newcommand{\trace}{\operatorname{tr}}
\newcommand{\one}{\mathbbm{1}}

\newcommand{\ep}{\varepsilon}

\newtheorem{thm}{Theorem}[section]
\newtheorem{prop}[thm]{Proposition}

\newtheorem{lem}[thm]{Lemma}

\theoremstyle{definition}
\newtheorem{defn}[thm]{Definition}
\newtheorem{rem}[thm]{Remark}

\numberwithin{equation}{section}

\allowdisplaybreaks

\author[S.~Patrizi]{\href{http://stepatrizi.altervista.org/}{Stefania Patrizi} }

\address{Department of Mathematics\\
The University of Texas at Austin\\
2515 Speedway, Austin\\
TX 78712, United States of America}
\email{spatrizi@math.utexas.edu}

\author[M. Vaughan]{\href{https://maryvaughan.github.io/}{Mary Vaughan}}

\address{Department of Mathematics and Statistics\\
The University of Western Australia\\
35 Stirling HWY\\
Crawley WA 6009, Australia}
\email{mary.vaughan@uwa.edu.au}



\begin{document}

\title[The discrete dislocation dynamics of multiple dislocation loops]{The discrete dislocation dynamics of multiple dislocation loops}

\begin{abstract}
We consider a nonlocal reaction-diffusion equation that physically arises from the classical Peierls--Nabarro model for dislocations in crystalline structures. Our initial configuration corresponds to multiple slip loop dislocations in $\R^n$, $n \geq 2$. After suitably rescaling the equation with a small phase parameter $\ep>0$, the rescaled solution solves a fractional Allen--Cahn equation. We show that, as $\ep \to 0$, the limiting solution exhibits multiple interfaces evolving independently and according to their mean curvature. 
\end{abstract}

\maketitle

\section{Introduction}

We study the fractional Allen--Cahn equation
\begin{equation} \label{eq:pde}
\ep \partial_t u^{\ep} = \frac{1}{\ep \abs{\ln \ep} }  (\ep\mathcal{I}_n [u^{\ep}]  - W'(u^\ep)) \quad \hbox{in}~(0,\infty)\times \R^n, ~n \geq 2,
\end{equation}
where $\ep>0$ is a small parameter,
$\mathcal{I}_n=-c_n(-\Delta )^\frac12$ denotes up to a constant  the square root of the Laplacian in $\R^n$,  
and $W$ is a multi-well potential, see \eqref{eq:operator} and \eqref{eq:W} respectively.  

Equation \eqref{eq:pde} is a rescaled version of the evolutionary Peierls--Nabarro model for atomic dislocations in crystalline structures. 
In this article, we initiate the mathematical analysis of the evolution of multiple slip dislocations according to \eqref{eq:pde}. 
Our initial configuration corresponds to multiple slip loops, all contained in the same slip plane, which we view as the boundaries of a nested sequence of  smooth and bounded sets in $\R^n$. 
We show that, as $\ep \to 0$, the curves 
move independently and according to their mean curvature. 
To the best of our knowledge, we are the first to consider the movement of multiple fronts in $\R^n$, $n \geq 2$, (the physical dimension being $n=2$) 
in this setting. 
Our results are new even for the case of a single curve and completes the work of Imbert and Souganidis in \cite{Imbert}. 
In dimension $n=2$, our main theorem is a formal passage from the microscopic Peierls--Nabarro model to the discrete dislocation dynamics at the mesoscopic scale and is the first time this has been done for multiple curved dislocations. 
Our work establishes the foundation for future research on the dynamics of an ensemble  of a large number of curved dislocations all  contained in the same slip plane. 

Before further detailing the significance of our main result and discussing the Peierls--Nabarro model for slip loop dislocations, let us formalize our problem mathematically.

\subsection{Setting of the problem and main result}
The operator $\mathcal{I}_n$ is a nonlocal integro-differential operator of order 1 and is 
defined on functions $u \in C^{1,1}(\R^n)$ by
\begin{equation}\label{eq:operator}
\mathcal{I}_n u(x) 
	= \operatorname{P.V.}~ \int_{\R^n} \( u(x+y) - u(x)\) \,\frac{dy}{\abs{y}^{n+1}}, \quad x \in \R^n, \quad n \in \N,
\end{equation}
where $\operatorname{P.V.}$ indicates that the integral is taken in the principal value sense. 
For further background on fractional Laplacians, see for example \cites{Hitchhikers,Stinga}. 
The potential $W$ satisfies
\begin{equation}\label{eq:W}
\begin{cases}
W \in C^{4, \beta} (\R) & \hbox{for some}~0 < \beta <1 \\
W(u+1) = W(u) & \hbox{for any}~ u \in \R\\
W= 0 & \hbox{on}~\Z\\
W>0 & \hbox{on}~\R \setminus \Z\\
W''(0) >0.
\end{cases}
\end{equation}
We let $u^{\ep}$ be the solution to \eqref{eq:pde} when the initial condition $u_0^\ep$ is a superposition of layer solutions. 
The layer solution (also called the phase transition) $\phi:\R \to (0,1)$ is the unique solution to  
\begin{equation} \label{eq:standing wave}
\begin{cases}
C_n \mathcal{I}_1[\phi] = W'(\phi) & \hbox{in}~\R\\
 \dot{\phi}>0 & \hbox{in}~\R\\
\phi(-\infty) = 0, \quad \phi(+\infty)=1,\quad \phi(0) = \frac{1}{2},
\end{cases}
\end{equation}
where $\mathcal{I}_1$ is the fractional operator in \eqref{eq:operator} on $\R$ and the constant $C_n>0$ (given explicitly in \eqref{eq:Cn}) depends only on $n \geq 2$. 
Further discussion on $\phi$ will be presented in Section \ref{sec:functions}. 

For a fixed $N \in \N$, let $(\Omega_0^i)_{i=1}^N$ be a finite sequence of open subsets of $\R^n$ that are  bounded, satisfy $\Omega_{0}^{i+1} \subset \subset \Omega_0^{i}$  and have smooth boundaries $\Gamma_0^i = \partial \Omega_0^i$. When $n=2$, the curves   
$\Gamma_0^i $ can be understood as the initial dislocation loops in the crystal.
Let $d_i^0(x)$ be the signed distance function associated to  $\Omega_0^i$, $i=1,\dots, N$, given by
\begin{equation}\label{eq:initial d_i}
d_i^0(x) = \begin{cases}
d(x, \Gamma_0^i) & \hbox{if}~x \in \Omega_0^i \\
-d(x,\Gamma_0^i) & \hbox{otherwise}.
\end{cases}
\end{equation}
For our initial condition to be well-prepared, we let $u_0^\ep$ be the $N$-fold sum of the layer solutions $\phi(d_i^0(x)/\ep)$, see \eqref{eq:initial cond} and Figure \ref{fig:initial}.

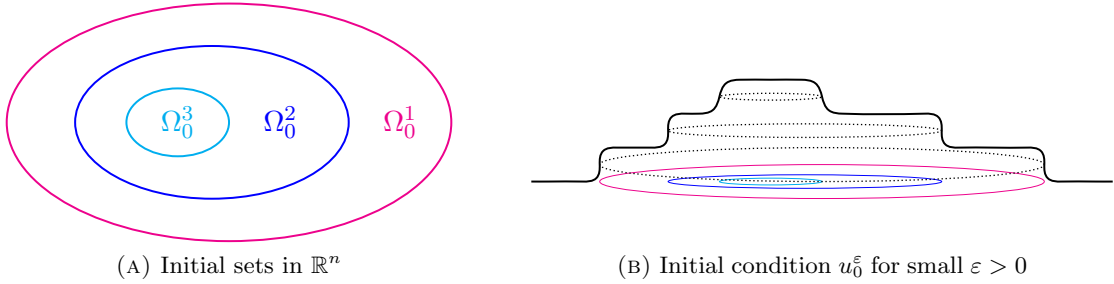
\begin{figure}[h]
\centering
\subfloat[Initial sets in $\R^n$]{
 \begin{tikzpicture}[scale=0.45, use Hobby shortcut, closed=true]
\draw[magenta,line width=.75pt] (1.5,0) ellipse (6.5cm and 3.5cm);  
\node[magenta] at (6.5,0) {$\Omega_0^1$} ;
\draw[blue,line width=.75pt] (1,0) ellipse (4cm and 2.25cm);  
\node[blue] at (3,0) {$\Omega_0^2$} ;
\draw[cyan,line width=.75pt] (0,0) ellipse (1.5cm and 1cm);   
\node[cyan] at (0,0) {$\Omega_0^3$} ;
\end{tikzpicture}
}
\qquad
\subfloat[Initial condition $u_0^{\ep}$ for small $\ep>0$]{
 \begin{tikzpicture}[scale=0.45, use Hobby shortcut, closed=false]
\draw[magenta,line width=.25pt] (1.5,0) ellipse (6.5cm and .5cm);  
\draw[blue,line width=.25pt] (1,0) ellipse (4cm and .2cm);  
\draw[cyan,line width=.25pt] (0,0) ellipse (1.5cm and .1cm);   
\draw[densely dotted,line width=.5pt] (1.5,.5) ellipse (6.5cm and .5cm);  
\draw[densely dotted,line width=.5pt] (1,1.5) ellipse (4cm and .2cm);  
\draw[densely dotted,line width=.5pt] (0,2.5) ellipse (1.5cm and .1cm);   
\draw[line width=.75pt]
	(-7,0)..(-6.5,0)..(-5.5,0)..(-5.1,.1)..(-5,.5)..(-4.9,.9)..(-4.5,1)..
	(-3.5,1)..(-3.1,1.1)..(-3,1.5)..(-2.9,1.9)..(-2.5,2)..
	(-2.1,2)..(-1.7,2.1)..(-1.5,2.5)..(-1.3,2.9)..(-.9,3)..(0,3)..
	(.9,3)..(1.3,2.9)..(1.5,2.5)..(1.7,2.1)..(2.1,2)..
	(4.5,2)..(4.9,1.9)..(5,1.5)..(5.1,1.1)..(5.5,1)..
	(7.5,1)..(7.9,.9)..(8,.5)..(8.1,.1)..(8.5,0)..(9.5,0)..(10,0);
\node[opacity=0] at (1,-1.5) {};
\end{tikzpicture}
}
\caption{Initial configuration for $N=3$ in dimension $n=2$}
\label{fig:initial}
\end{figure}

Consider a continuous viscosity  solution $u^i(t,x)$ to the mean curvature equation (see  \eqref{eq:meancurvature} with $\mu$ as in \eqref{eq:velocity-intro}) whose
 positive,  zero and negative sets  at time $t=0$ respectively are  $\Omega_0^i$, $\Gamma_0^i$ and  $(\overline{\Omega_0^i})^c$.  
If $ {^+}\Omega_t^i$, $\Gamma_t^i$, and ${^-}\Omega_t^i$ respectively are the positive, zero, and negative sets of $u^i(t,\cdot)$ at time $t>0$, then we say that the triplet 
$( {^+}\Omega_t^i, \Gamma_t^i, {^-}\Omega_t^i)_{t \geq 0}$ denotes the 
level set evolution of $( \Omega_0^i, \Gamma_0^i, (\overline{\Omega_0^i})^c)$.  
Since $\Omega_{0}^{i+1} \subset \subset \Omega_0^{i}$, one has that 
$ {^+}\Omega_t^{i+1}\subset \subset {^+}\Omega_t^i$ and $ {^-}\Omega_t^{i}\subset \subset {^-}\Omega_t^{i+1}$, for all times $t>0$ such that $ {^+}\Omega_t^i\neq\emptyset$. 
See Section \ref{sec:MC} for definitions and details on the level set approach to motion by mean curvature. 

We now present the main result of our paper.  
See Section \ref{sec:frac} for more on solutions to \eqref{eq:pde}.

\begin{thm}\label{thm:main}
Let $u^{\ep} = u^{\ep}(t,x)$ be the unique  solution of the reaction-diffusion equation \eqref{eq:pde} with the initial datum $u^{\ep}_0:\R^n \to (0,N)$ defined by
\begin{equation}\label{eq:initial cond}
u^{\ep}_0(x) = \sum_{i=1}^N \phi\(\frac{d_i^0(x)}{\ep}\),
\end{equation}
where $\phi$ solves \eqref{eq:standing wave}, $d_i^0$ are given in \eqref{eq:initial d_i}, and $N\geq1$.
Then, as $\ep \to 0$, the solutions $u^{\ep}$ satisfy
\[
u^\ep \to \begin{cases}
N & {^+}\Omega_t^N,\\
 i  \quad \,\, \hbox{locally uniformly in}~&{^+}\Omega_t^i \cap {^-}\Omega_t^{i+1}, \quad i=1,\dots, N-1,\\
 0 & {^-}\Omega_t^1,
\end{cases}
\]
where $( {^+}\Omega_t^i, \Gamma_t^i, {^-}\Omega_t^i)$ denotes the level set evolution of $( \Omega_0^i, \Gamma_0^i, (\overline{\Omega_0^i})^c)$ with velocity \eqref{eq:velocity-intro}.
\end{thm}

\begin{figure}[h]
\centering
 \begin{tikzpicture}[scale=0.5, use Hobby shortcut, closed=true]
\draw[magenta,line width=.75pt]
	(0,4)..(3,4.5)..(4.5,2)..(7,0)..(6.5,-3)..(5,-4)..(2,-3)..(0,-3)..
	(-2.5,-4)..(-3,-2.5)..(-3,-2)..(-4.5,-1)..(-4,1)..(-4.2,1.5)..(-6,3)..(-4,5.25)..(-3.5,2.5)..(-2,2);
\draw[blue,line width=.75pt]
	(1,2.5)..(2,2.25)..(3,1.2)..(4,1)..(5,-.5).. 
	(3.5,-1)..(3.25,-1.5)..(2,-2)..(-1.5,-1.5)..
	(-2.5,-1)..(-2.75,-.75)..(-3,.25)..(-1.2,1.5)..(-.5,1.75)..(.2,2.25)..(1,2.5);
\draw[cyan,line width=.75pt] (1,0) ellipse (2cm and 1cm);    
\node[] at (1,0) { $u^{\ep} \to 3$};
\node[] at (1,1.5) { $u^{\ep} \to 2$};
\node[] at (1,3.25) { $u^{\ep} \to 1$};
\node[] at (1,5.5) { $u^{\ep} \to 0$};
\node[cyan] at (3.5,0) {$\Gamma_t^3$};
\node[blue] at (5.65,0) {$\Gamma_t^2$};
\node[magenta] at (7.6,0) {$\Gamma_t^1$};
\end{tikzpicture}
\caption{Convergence result for $N=3$  in dimension $n=2$}
\label{fig:main thm}
\end{figure}
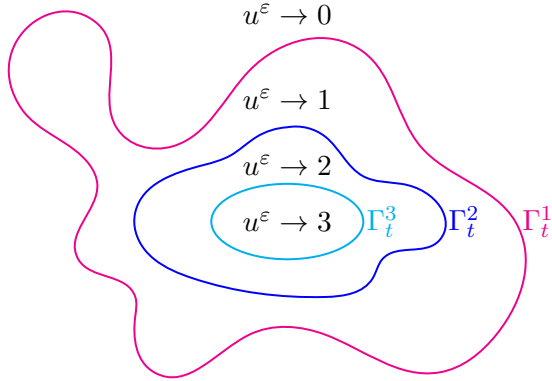

As illustrated in Figure \ref{fig:main thm}, Theorem \ref{thm:main} says that the solutions $u^{\ep}$ converge to integers ``between'' the interfaces $\Gamma_t^i$.
In particular, if $u^i$ is a continuous viscosity  solution to the mean curvature equation whose positive, zero, and negative sets at time $t=0$ are $\Omega_0^i$, $\Gamma_0^i$ and  $(\overline{\Omega_0^i})^c$, respectively, then 
\[
\lim_{\ep \to 0} u^\ep(t,x) = \begin{cases}
N & \{u^N>0\},\\
 i \,\, \quad \hbox{locally uniformly in}~&\{u^i >0\} \cap \{u^{i+1}<0\}, \quad i=1,\dots, N-1,\\
 0 & \{u^{i}<0\}.
\end{cases}
\]
Moreover, the sets $(\Gamma_t^i)_{t \geq 0}$ move independently by mean curvature. 
Specifically, $(\Gamma_t^i)_{t \geq 0}$ move in the direction of the interior normal vector with scalar velocity
\begin{equation}\label{eq:velocity-intro}
v_i
=\mu \sum_{j=1}^{n-1}\kappa^i_j, \quad \mu:= \frac{c_0}{2}\frac{|S^{n-2}|}{n-1}>0,
\end{equation}
where $\kappa^i_j$ are the principal curvatures of $\Gamma_t^i$, $c_0>0$ is explicit (see \eqref{eq:c0-gamma}), and $S^{n-2}$ denotes the unit sphere in $\R^{n-1}$. 
We use the level set approach to handle possible singularities for large times $t>0$.
In dimension 2, $\Gamma_t^i$ can be interpreted as  the dislocation loop at time $t>0$ evolved from $\Gamma_0^i$. 

For the case in which 
the sets $\Gamma_t^i$ do not develop interior, i.e.~$\Gamma_t^i = \partial({^+}\Omega_t^i) = \partial ({^-}\Omega_t^i)$, 
the limiting function in Theorem \ref{thm:main} makes integer jumps on the surfaces $\Gamma_t^i$ and satisfies
\[
\lim_{\ep \to 0} u^{\ep} = \frac{N}{2} + \frac{1}{2} \sum_{i=1}^N\( \one_{{^+}\Omega_t^i} - \one_{(\overline{{^+}\Omega_t^i})^c}\)  \quad \hbox{in}~\big((0,\infty)\times  \R^n\big)
\setminus  \bigcup_{i=1}^N \bigg(\bigcup_{t>0}\{t\}\times \Gamma_t^i\bigg)
\]
where $\one_\Omega$ denotes the characteristic function of the set $\Omega \subset \R^n$. 
However,  
it is well known that  $\Gamma_t^i$ may develop interior in finite time,  even if $\Gamma^i_0$ has none, see e.g.~\cite{Grayson}. In this situation,
the discontinuity set at time $t$ of the limiting function is contained in the sets $\Gamma_t^i$, but we cannot say exactly where the jumps occur within these sets.

\subsection{Strategies and significance}\label{sec:significance}

We now discuss the key aspects of Theorem \ref{thm:main}, its proof, and some of the relevant literature from a mathematical point of view. We will address physical motivations further in the next subsection.

Theorem \ref{thm:main} for $N=1$ has been addressed in the literature in the local case when $W$ is instead a double-well potential. 
For instance, the classical Allen--Cahn equation for which \eqref{eq:pde} is instead driven by the usual Laplacian $\Delta$ was studied famously by Modica--Mortola \cite{MM}  for the stationary case. 
Chen studied the corresponding evolutionary Allen--Cahn problem and proved that the solution exhibits an interface moving by mean curvature \cite{Chen}. 
See Section \ref{sec:MC} for more on the phase field theory. 
In the fractional setting and again for a double-well potential, the stationary case was studied by Ambrosio--De Philippis--Martinazzi \cite{AmbrosioDePhilippisMartinazzi} and Savin--Valdinoci  \cite{SavinValdinoci}
and the evolution problem was considered by Imbert--Souganidis in the preprint \cite{Imbert}. 
We will describe below how our result is not included in and does not follow from \cite{Imbert, SavinValdinoci}. 
Indeed, our problem deals with $N\geq1$ interfaces and a multi-well potential. 

The proof of Theorem \ref{thm:main} relies on 
the abstract method introduced in  \cite{ Barles-DaLio, Barles-Souganidis} for the study of front propagation. 
For this, we construct new barriers that are strict sub/super solutions to \eqref{eq:pde}, see Section \ref{sec:barriers}. 
The authors recently proved in \cite{PatriziVaughan} that the mean curvature of a smooth, evolving front arises as the $\ep$-limit of an auxiliary function that is intrinsically related to the difference between an $n$- and $1$-dimensional fractional Laplacian  (see Theorem \ref{lem:4} and also Lemma \ref{lem:ae near front}). This plays a key role in building the appropriate barriers.

We found that a superposition of the corresponding barriers in \cite{Imbert} cannot be applied for several reasons.
Their idea is to use the ansatz 
$\phi_c(d(t,x)/\ep)$ to construct a subsolution close to the front $\Gamma^1_t$ 
where $d$ is the signed distance to  ${^+}\Omega_t^1$  and $\phi_c$  solves a traveling wave equation with speed $c$ (with $c=0$ corresponding to \eqref{eq:standing wave}). 
Since the equation is nonlocal and nonlinear, a difficulty arises in dealing with $d(t,x)$ when $(t,x)$ is far from the front since $d$ may not be smooth at such points. 
In \cite{Imbert}, they truncate/extend the subsolution away from the front, taking particular care when truncating from below in order to remain a subsolution.
We cannot use a superposition of the corresponding subsolutions from \cite{Imbert} because  those functions go asymptotically to integers far from the front  like $\ep|\ln\ep|$ and this introduces an error associated with the nonlinear potential that cannot be controlled. 
Instead, we truncate/extend the signed distance function away from the front which allows us to directly construct a global subsolution using a superposition of layer solutions $\phi(\tilde{d}_i(t,x)/\ep)$ where $\tilde{d}_i$ is a smooth bounded extension of the signed distance function $d_i$ to ${^+}\Omega_t^i$ and $\phi$ solves the stationary equation \eqref{eq:standing wave}.
 Each of these functions  approaches integers far from the front $\Gamma_t^i$  like $\ep$, and we can control the error.

Moreover, as previously mentioned, in \cite{Imbert}, they use traveling waves to control the error. However, existence and asymptotic behavior of traveling waves is not proven but rather \emph{assumed}. 
In fact, the asymptotic behavior at infinity is not true as shown in the stationary case. 
We use the stationary solution $\phi$ to \eqref{eq:standing wave}  whose existence, uniqueness, and asymptotic behavior are known  (see Lemma \ref{lem:asymptotics}).

In the actual construction of the barriers, we introduce lower-order correctors to control the error as $\ep \to0$. 
We use a superposition of solutions  $\psi= \psi_\ep^i$ to the linearized equation
\begin{equation}\label{eq:corrector-intro}
 -C_n\mathcal{I}_1[\psi] + {W}''(\phi) \psi =g
 \end{equation}
 for some right-hand side $g = g_\ep^i$ depending on $\ep>0$ and on the signed distance function $d_i$. 
The explicit form of $g$ is somewhat technical (given in Section \ref{sec:functions}) and is different than in \cite{Imbert}.
We add an extra term involving $\sigma\in\R$ (with $\sigma = 0$ corresponding to \cite{Imbert}) which provides us with a new way for controlling the error.
Moreover, in \cite{Imbert}, they \emph{assume} the existence, uniqueness, and asymptotic behavior of  correctors. 
It was left as an open question whether such correctors exist. 
Then, the assumed asymptotic behavior of the traveling waves and corrector are 
used to build a global subsolution. 
However, the assumption that $\psi$ and its derivatives are bounded independently of $\ep$ and that the behavior at infinity of $\psi$ is uniform in $\ep$ both seem to be too strong.
In this paper, we prove  both existence, uniqueness, and asymptotic behavior of the solution $\psi$ to \eqref{eq:corrector-intro}. For example, we found, via delicate analysis, that $\psi = \psi^i_\ep$ can only be bounded independently of $\ep>0$ near the front $\Gamma_t^i$. In general, it holds that $\psi =O(\ep^{-\frac{1}{2}}|\ln \ep|^{-1})$. 
See Theorem \ref{lem:psi-reg} for detailed properties of $\psi$ which are sufficient to prove Theorem \ref{thm:main}.

A further feature 
of our work is the addition of two heuristical derivations in Section \ref{sec:heuristics}. 
First, we provide heuristics for the evolution of the fronts $\Gamma_t^i$ by mean curvature in Theorem \ref{thm:main} in order to give an intuitive explanation of the behavior of the fronts. 
Second, we write a heuristical derivation for the equation of the corrector \eqref{eq:corrector-intro} which clarifies the choice of $g$. 

The results and proofs of this paper can be applied to situations where the number of fronts,  $N$, is large, as discussed below. 
Our formal computations  show that the dynamics of the surface $\Gamma_t^i$ is also affected by an interaction potential that depends on the distance of
$\Gamma_t^i$  from the other surfaces $\Gamma_t^j$, $j\neq i$. This potential takes the form  
\begin{equation}\label{interaction-potential:intro}
\frac{c_0 C_n }{\abs{\ln\ep}}\sum_{j\neq i}\frac{1}{d_j},
\end{equation}
with $c_0>0$ and $C_n>0$  as defined  in  \eqref{eq:c0-gamma} and  \eqref{eq:Cn}, where $d_j$ is the signed distance function from ${^+}\Omega_t^j$, see Remark \ref{Potential-remark}. However, since the fronts remain at a positive distance from one another, this potential
is negligible to  first-order approximation. This information is completely lost if we cut the barriers far from the front (as in \cite{Imbert}). 
In dimension $n=2$, 
 when dislocation lines are {\em straight} and parallel (thus with zero mean curvature), 
 after a logarithmic rescaling in time and a reduction of spatial dimension, 
equation \eqref{eq:pde} becomes
\begin{equation}\label{uepeq-n=1}
\ep \partial_t u^{\ep} =C_2\mathcal{I}_1 [u^{\ep}]  -  \frac{1}{\ep }  W'(u^\ep) \quad \hbox{in}~(0,\infty)\times \R.
\end{equation}
The  asymptotics of the solution to \eqref{uepeq-n=1} with initial condition as in 
\eqref{eq:initial cond}, where $d_i^0(x)=x-x_i^0$ and $x=x_i^0$ represents the location of the curve  $\Gamma_0^i$, were studied by  Gonzalez--Monneau \cite{GonzalezMonneau}.
They   proved 
that  in this time scale, dislocation lines  $(\Gamma_t^i)_{t \geq 0}$ move with velocities 
$$v_i=c_0 C_2 \sum_{j\neq i}\frac{1}{d_j},$$
where   $d_j(t,x)=x-x_j(t)$ is  the signed distance function from the dislocation line $x=x_j(t)$. Moreover, they remain at a positive distance from each other for all times. 

Equation \eqref{uepeq-n=1}, where the operator $ \mathcal{I}_1$ is replaced by 
$-c_{1,s}(-\Delta)^{s}$,   $s\in(0,1)$,  and with the appropriate space scaling,  is studied in \cite{DipierroFigalliValdinoci,DipierroPalatucciValdinoci}.
The case where the number of dislocation lines tends to infinity is explored in \cite{patsan,patsan2}, where 
the authors proved that, after a suitable rescaling, the solution to \eqref{uepeq-n=1} with an appropriate initial condition, converges as both $\ep\to0$ and $N\to\infty$ to the solution of a well-known equation in dislocation theory, representing a plastic flow rule for the density of dislocations. In dimensions $n\geq 1$, the case where $N\to\infty$ is studied in \cite{MonneauPatrizi}, but only when $\ep=1$.
 In this case, the limit equation is implicitly defined by a cell problem.
In extending the analysis to the scenario where both  $\ep\to0$ and $N\to\infty$ in dimension $n\geq 2$,  we believe that, after  an appropriate rescaling in space and time, the potential \eqref{interaction-potential:intro} will no longer be negligible and will play an important role, as it does in dimension 1.
 For a physical interpretation of the parameter $\ep$, see Section \ref{sec:PN}.

We also mention that equation \eqref{uepeq-n=1}, with more general initial data, such as~
opposite orientations of dislocations, is considered in \cite{patval1, patval2, patval3, MeursPatrizi}.
In our setting, fronts with different orientations can be studied by choosing in \eqref{eq:initial cond} the signed distance function $d_i^0$ to  either ${^+}\Omega_t^i$ or 
${^-}\Omega_t^i$. In this case, we expect  fronts will  move with velocity $\pm v_i$, with $v_i$ given in   \eqref{eq:velocity-intro}, depending on their orientation, and collisions  between fronts may occur.
We believe that our analysis still holds for every time smaller than the first collision time. The analysis after the first collision time will be the subject of future work.

Lastly, we mention that one could also include obstacles and external forces 
by adding a suitable  stress $\sigma(t,x)$ to \eqref{eq:pde}. 
This was handled in several of the aforementioned papers, such as \cites{DipierroFigalliValdinoci, DipierroPalatucciValdinoci, GonzalezMonneau, MonneauPatrizi,patval1,patval2, patval3}, but we do not consider it here for simplicity.

\subsection{The Peierls--Nabarro model for slip loop dislocations}\label{sec:PN}

We now describe how \eqref{eq:pde} with initial data \eqref{eq:initial cond}
arises from the Peierls--Nabarro model for slip loop dislocations. 

Dislocations are line defects in crystalline materials whose motion is responsible for the plastic behavior of metals (see \cites{HirthLothe,HullBacon} for more details). 
The dynamics of an {\em ensemble of dislocation lines all contained in the same plane} (the slip plane) can be described at different scales by different models:
\begin{itemize}
	\item atomic scale (Frenkel--Kontorova model),
	\item microscopic scale (Peierls--Nabarro model),
	\item mesoscopic scale (Discrete dislocation dynamics),
	\item macroscopic scale (Elasto-visco-plasticity with density of dislocation). 
\end{itemize}
 Several changes of scales exist in the literature.
The passage from the Frenkel--Kontorova model to the  Peierls--Nabarro model is performed  in \cite{FinoIbMon}. 
In the evolutive models, further changes of scales  are only  known  when the dislocations are  {\em straight}, parallel lines (i.e.~one-dimensional models):  
as already mentioned in Subsection \ref{sec:significance}, Gonzalez--Monneau \cite{GonzalezMonneau} passed from the  Peierls--Nabarro model to the discrete dislocation dynamics
(see also \cite{patval1,patval3,patval4,MeursPatrizi} for the case of dislocations with  different orientation); 
the passage from  the discrete dislocation dynamics to the elasto-visco-plasticity with density of dislocation is done  in 
\cite{forimmon, Meupelpoz}; finally, a direct passage from the   Peierls--Nabarro model to the  elasto-visco-plasticity with density of dislocation is performed in \cite{patsan, patsan2}. 
We reference the reader to \cite{HIM, DipierroPatriziValdinoci} and the references therein for more on the dynamics of edge dislocations in one-dimension. 

For different models for straight dislocation lines, we refer to \cite{ alidelgarpon, alidelgarpon2, garrllepon,  garmepescardia} and the references therein. 

In reality, dislocation lines are rarely straight and often form loops, see \cite{sed}. 
The dynamics of an ensemble of {\em curved} dislocation lines lying in the same slip plane is described by the above models in dimension $n=2$. 
Our main result, Theorem \ref{thm:main}, is a passage from the Peierls--Nabarro model to the discrete dislocation dynamics in any dimension $n\geq2$ and is the first time it has been done for multiple curved dislocations. 
 Differently from the one-dimensional case where lines move according to an interaction potential (see \cite{GonzalezMonneau}), 
here  the line tension effect is stronger than the interaction between the curves, and we have a movement by mean curvature as established in Theorem \ref{thm:main}. 
See also the work of Garroni--M\"uller  \cite{GarroniMuller} and  Garroni--Conti--Muller  \cite{Contigarmul} for  variational Peierls--Nabarro  models  in dimension two.  
As already discussed  in Subsection \ref{sec:significance}, we also  mention that Monneau and the first author  in \cite{MonneauPatrizi} studied homogenization of the Peierls--Nabarro model in any dimension that can be seen as a direct passage from the microscopic scale to the macroscopic scale.
 See   \cite{imr, patval2} for related results. 

\medskip 

We next briefly review the Peierls--Nabarro model for a dislocation loop. 
For the original model, we refer to \cites{PN1,PN2,PN3}. 
A slip dislocation occurs when a portion of atoms  slides over another along a particular plane, called the slip plane, and is the curve created by the boundary between the shifted
and unshifted regions. 
In Cartesian coordinates $x_1x_2x_3$, we assume that the slip plane is the $x_1x_2$-plane. 
The movement of the dislocation is determined by the so-called Burgers' vector $b$.
For slip dislocations, the Burgers' vector is contained within the slip plane, so  we assume it to be in the direction of the $x_1$-axis, say $b = e_1$.

There are two types of slip dislocations that correspond to straight lines in $\R^2$: edge and screw. 
An edge dislocation is formed when an extra half-plane of atoms is included in the crystal. 
In this case, the Burgers' vector is perpendicular to the dislocation line
 and the motion of the dislocation line is in the direction of $b$. 
 A screw dislocation instead creates a spiral or helical path around the core.
Here, the Burgers' vector is parallel to the dislocation curve
which means that the dislocation line moves perpendicular to the direction of $b$. 
It is probable that most slip dislocations are not straight lines and thus exhibit components of both edge and screw dislocations. These are called mixed dislocations.
\emph{Slip loops} (or loop dislocations) are slip dislocations in the form of simple closed curves. 
Since the Burgers' vector is fixed for a given slip loop, the type of dislocation changes from point to point along the dislocation curve; for an idealized sketch, see Figure \ref{fig:loop description}. Depending on the orientation of the crystal and any external body forces acting upon the dislocation,
the slip loop will either shrink or expand within the slip plane. 

\begin{figure}[ht]
\centering
\subfloat[In the slip plane $x_1x_2$]{
 \begin{tikzpicture}[scale=0.29, use Hobby shortcut, closed=true]
 \label{fig:loopa}
\draw[line width=.75pt, magenta] (0,0) ellipse (6cm and 4cm);
\draw[line width=.5pt, -stealth,magenta](0,4)--(2.25,4);
\draw[line width=.5pt, -stealth,magenta](0,-4)--(-2.25,-4);
\draw[line width=.5pt, -stealth,magenta](-6,0)--(-6,2.25);
\draw[line width=.5pt, -stealth,magenta](6,0)--(6,-2.25);
\draw[line width=.5pt, -stealth,magenta](4,3)--(5.6,2);
\draw[line width=.5pt, -stealth,magenta](-4,-3)--(-5.6,-2);
\draw[line width=.5pt, -stealth,magenta](4,-3)--(2.4,-4);
\draw[line width=.5pt, -stealth,magenta](-4,3)--(-2.4,4);
\node at (0,5) {Screw};
\node at (0,-5) {Screw};
\node at (-8,0) {Edge};
\node at (8,0) {Edge};
\node at (5.5,4) {Mixed};
\node at (-5.5,-4) {Mixed};
\node at (5.5,-4) {Mixed};
\node at (-5.5,4) {Mixed};
\draw[<->,line width=.75pt] (-11,-8)--(-11,6);
\draw[<->,line width=.75pt] (-12,-7)--(10,-7);
\node at (-11,6.5) {$x_2$};
\node at (10.75,-7) {$x_1$};
\draw[line width=1.25pt, -stealth](-10,-6.2)--(-8,-6.2);
\node at (-9,-5.45) {{\small$b=e_1$}};
\node at (0,0) {$\Phi \simeq 1$};
\node at (9,6) {$\Phi \simeq 0$};
\end{tikzpicture}
}
\hskip7pt
\subfloat[In $\R^3$ (inspired by Fig.~3.11 in \cite{PassTrouw}) ]{
\begin{tikzpicture}[scale=0.8, x=0.75pt,y=0.75pt,yscale=-1,xscale=1]
\draw  [draw opacity=0][line width=0.5]  (384.5,349.47) .. controls (384.5,349.47) and (384.5,349.47) .. (384.5,349.47) .. controls (384.5,349.47) and (384.5,349.47) .. (384.5,349.47) .. controls (384.5,363.41) and (314.88,374.72) .. (229,374.72) .. controls (143.12,374.72) and (73.5,363.41) .. (73.5,349.47) .. controls (73.5,335.52) and (143.12,324.22) .. (229,324.22) -- (229,349.47) -- cycle ; \draw  [color=magenta  ,draw opacity=1 ][line width=0.75]  (384.5,349.47) .. controls (384.5,349.47) and (384.5,349.47) .. (384.5,349.47) .. controls (384.5,349.47) and (384.5,349.47) .. (384.5,349.47) .. controls (384.5,363.41) and (314.88,374.72) .. (229,374.72) .. controls (143.12,374.72) and (73.5,363.41) .. (73.5,349.47) .. controls (73.5,335.52) and (143.12,324.22) .. (226.8,324.2) ;  
%
\draw (210.5,259.97) -- (211,319.97) ;
\draw    (210.5,259.97) -- (310.5,289.47) ;
\draw    (210,278.97) -- (310.5,309.47) ;
\draw    (210.5,299.97) -- (310.5,329.97) ;
\draw    (229.5,265.97) -- (230,325.97) ;
\draw    (251,271.47) -- (251,332.47) ;
\draw    (270,277.97) -- (270.5,337.97) ;
\draw    (290.5,283.47) -- (290,342.47) ;
\draw    (310.5,289.47) -- (311,349.47) ;
\draw    (435.5,290.97) -- (310.5,289.47) ;
\draw    (335,288.97) -- (335.5,349.47) ;
\draw    (359.75,289.97) -- (360,349.72) ;
\draw    (384.5,289.97) -- (384.5,349.47) ;
\draw    (409,290.47) -- (410,351.47) ;
\draw    (435,309.47) -- (310.5,309.47) ;
\draw    (435,329.97) -- (310.5,329.97) ;
\draw [blue  ,draw opacity=1 ][line width=.75]    (435.75,350.72) -- (286,349.47) ;
\draw [blue  ,draw opacity=1 ][line width=.75]    (211,319.97) -- (311,349.47) ;
\draw    (286,349.47) -- (286.5,409.47) ;
\draw    (436.5,369.97) -- (286,368.47) ;
\draw    (436,390.47) -- (286.5,389.97) ;
\draw    (436,410.47) -- (286.5,409.47) ;
\draw    (211,319.97) -- (211.5,379.97) ;
\draw    (230,325.97) -- (230.5,385.97) ;
\draw    (242.5,331.47) -- (243,391.47) ;
\draw    (211.5,379.97) -- (230.5,385.97) ;
\draw    (211,338.97) -- (230.5,344.97) ;
\draw    (211.5,359.97) -- (230,365.97) ;
\draw [blue  ,draw opacity=1 ,line width=.75]    (230,325.97) -- (286,349.47) ;
\draw    (255.5,337.22) -- (256,397.22) ;
\draw    (270,342.97) -- (270.5,402.97) ;
\draw [blue ,draw opacity=1, line width=.75 ]   (290,342.47) -- (270,342.97) ;
\draw    (311,349.47) -- (311.5,409.97) ;
\draw    (360,349.72) -- (364,369.22) ;
\draw    (364,369.22) -- (367,390.72) ;
\draw    (367,390.72) -- (367.5,410.72) ;
\draw    (403.5,391.97) -- (403,409.97) ;
\draw    (410,351.47) -- (406.5,371.47) ;
\draw    (406.5,371.47) -- (403.5,391.97) ;
\draw    (335.5,261.47) -- (210.5,259.97) ;
\draw    (236.5,259.97) -- (335,288.97) ;
\draw    (259.75,260.47) -- (359.75,289.97) ;
\draw    (284.5,260.47) -- (384.5,289.97) ;
\draw    (309,260.97) -- (409,290.47) ;
\draw    (230.5,385.97) -- (286.5,409.47) ;
\draw    (230.5,366.47) -- (286.5,389.97) ;
\draw    (230,344.97) -- (286,368.47) ;
\draw [blue  ,draw opacity=1, line width=.75 ]   (270.5,337.97) -- (255.5,337.22) ;
\draw [blue ,draw opacity=1, line width=.75]   (251,331.47) -- (242.5,331.47) ;
\draw    (435.5,290.97) -- (436,410.47) ;
\draw    (335.5,261.47) -- (435.5,290.97) ;
\draw    (415.5,284.97) -- (290.5,283.47) ;
\draw    (395,279.47) -- (270,277.97) ;
\draw    (376,272.97) -- (251,271.47) ;
\draw    (354.5,267.47) -- (229.5,265.97) ;
\draw    (335.5,349.47) -- (337.5,368.47) ;
\draw    (337.5,368.47) -- (339.5,389.47) ;
\draw    (339.5,389.47) -- (340,409.97) ;
\draw[line width=1.25pt, -stealth](107,358)--(134,358);
\node at (120,348) {{\small$b=e_1$}};
\node at (405,340) {Edge};
\node at (209,312) {Screw};
\node[opacity=0] at (210,425) {};
\end{tikzpicture}
}
\caption{Dislocation types in a slip loop dislocation with fixed Burgers' vector}
\label{fig:loop description}
\end{figure}
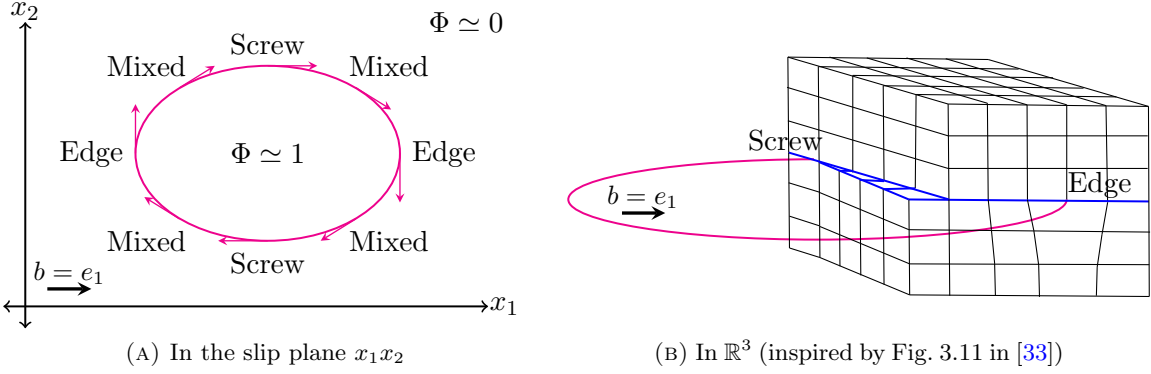

The Peierls--Nabarro model is a phase field model for dislocation dynamics which incorporates the atomic features of a crystalline structure into continuum framework. 
In the phase field approach, dislocations are interfaces represented by a transition of a continuous field: 
a phase parameter 
$\Phi(x_1,x_2)$ between $0$ and $1$ is used to capture the disregistry of the upper half crystal $\{x_3>0\}$ from the lower half crystal $\{x_3<0\}$. 
In particular, the dislocation loop is the set $\Gamma_0 := \{ \Phi = 1/2\} \subset \R^2$, and the phase parameter satisfies $\Phi \simeq 1$ inside the loop and $\Phi \simeq 0$ outside the loop (as indicated in Figure \ref{fig:loopa}). 
(To change the orientation, one would instead assume that $\Phi \simeq 0$ inside the loop and $\Phi \simeq 1$ outside the loop.)

For details on the following derivation, we refer to \cite{MonneauPatrizi}*{Section 2.1} and also \cite{Alvarez}.

Let $U = U(x_1,x_2,x_3): \R^3 \to \R^3$ be the displacement between an atom at location $(x_1,x_2,x_3)$ and its rest position. 
In the Peierls--Nabarro model, the total energy is the elastic energy for bonds between atoms plus the so-called misfit energy due to nonlinear atomic interactions across the slip plane:
\[
\mathcal{E}(\Phi,U) = \mathcal{E}_{\text{elastic}}(\Phi,U) + \mathcal{E}_{\text{misfit}}(\Phi). 
\]
To quantify the mismatch of atoms across the slip plane, we use an interplanar potential $W$ 
so that
\[
\mathcal{E}_{\text{misfit}}(\Phi)
	= \int_{\R^2} W(\Phi(x_1,x_2)) \, dx_1 \, dx_2.
\]
The elastic energy is given by
\[
\mathcal{E}_{\text{elastic}}(\Phi,U) = \frac{1}{2} \int_{\R^3} e : \Lambda : e \, dx_1\, dx_2 \, dx_3
\]
where
\[
e = \frac12 \bigg[ (\nabla U + (\nabla U)^T)
	 - \Phi(x_1,x_2) \delta_0(x_3) (e_1 \otimes e_3 + e_3 \otimes e_1) \bigg], 
\]
$\delta_0$ is the Dirac measure at 0, $\otimes$ is the tensor product of vectors, 
 $\Lambda = \{\Lambda _{ijlk}\}$ are the elastic coefficients, and $ e : \Lambda : e=\sum_{ijkl}\Lambda _{ijlk}e_{ij}e_{kl}$. We assume isotropic elasticity, so that
\[
\Lambda _{ijlk} = \lambda \delta_{ij} \delta_{kl} + \mu(\delta_{ik} \delta_{jl} + \delta_{il} \delta_{jk})
\]
where $\lambda, \mu$ are the Lam\'e coefficients. 

Given the phase parameter $\Phi$, we minimize $\mathcal{E}(\Phi,U)$ over all $U$ to formally find  an energy on the slip plane:
\[
\mathcal{E}(\Phi) := \inf_U \mathcal{E}(\Phi,U)
	= - \frac12 \int_{\R^2} (k_0\star \Phi)(x) \Phi(x) \, dx + \int_{\R^2} W(\Phi(x))\, dx,
\]
with $x=(x_1,x_2)$, where, in the special case that $\lambda = 0$ and $\mu= 4\pi$,  the Fourier transform of the kernel $k_0(x)$ is
\[
\widehat{k_0}(y) = - 2\pi |y|, \quad y \in \R^2,
\]
and $k_0\star \Phi =\mathcal{I}_2(\Phi)$. 

 In the case of a straight dislocation line, the equilibrium configuration is obtained by minimizing the energy $\mathcal{E}(\Phi)$ 
 with respect to $\Phi$, under the constraint that far from the dislocation core,
the function $\Phi$ tends to 0 in one half plane and to 1 in the other half plane. In particular,
the phase transition $\Phi$  is then solution of the following equation
\[
C_2\mathcal{I}_1(\Phi) = W'(\Phi) \quad \hbox{in}~\R,
\] 
where $\mathcal{I}_1 = -\pi (-\Delta)^\frac12$ is the fractional Laplace operator of order 1 in $\R$ and $C_2$ is defined in \eqref{eq:Cn}.
In the original model, $W(u)=1-\cos(2\pi u)$ and the solutions $\Phi$ is explicit, see \cite{CabreSola,GonzalezMonneau}. 
In a more general model, once can consider $W$ satisfying \eqref{eq:W}. Here, the periodicity of $W$ reflects the periodicity of the crystal while 
the minimum property on $\Z$  captures how the energy penalizes atoms which are not aligned with the lattice structure. 

In face cubic structured (FCC) lattices, which are commonly found in metals and alloys, 
dislocation curves move within the slip plane at low temperatures. 
We are interested in the evolution of multiple loop dislocations in the same slip plane and with the same Burgers' vector. 
For this, we use a single parameter $u (t,x)$ defined for $x$ in the slip plane $\R^n$, $n \geq 2$, the physical dimension being $n=2$. 
The dislocation dynamics are then captured by the evolutionary Peierls--Nabarro model:
\[
\partial_tu = \mathcal{I}_n[u] - W'(u) \quad \hbox{in}~(0,\infty) \times \R^n.
\]
At the microscopic scale, we assume that the dislocations curves are at a distance of order $1/\ep$ from each other. 
This can be represented by the initial condition
\[
u(0,x) = \sum_{i=1}^N \phi \(\frac{d_i^0(\ep x)}{\ep}\) \quad \hbox{for}~x \in \R^n,~\ep>0,
\]
where $\phi$ solves \eqref{eq:standing wave} in $\R$ and $d_i^0$ is the signed distance function associated to the loop $\Gamma^i_0$.

In order to understand the movement of the dislocation curves at a larger (mesoscopic) scale, we consider the rescaling
\begin{equation*}\label{eq:scaling}
u^{\ep}(t,x) =  u\(\frac{t}{\ep^2 \abs{\ln \ep}}, \frac{x}{\ep}\), \quad (t,x) \in [0,\infty) \times \R^n,~\ep>0. 
\end{equation*} 
Here, $\ep>0$ represents the scaling between the microscopic scale and the mesoscopic scale. 
The presence of the factor $\abs{\ln \ep}$ is well-known in physics, see \cites{Brown,DaLio,HirthLothe}. 
Roughly speaking, it arises from an integrability condition for the kernel of $\mathcal{I}_n$  in \eqref{eq:operator}, see \cite{PatriziVaughan, SavinValdinoci}.
One can easily check that $u^{\ep}$ solves \eqref{eq:pde} with the initial condition \eqref{eq:initial cond}.

Therefore, Theorem \ref{thm:main} indeed establishes the passage from the Peierls--Nabarro model to the discrete dislocation dynamics for an ensemble slip loop dislocations.  
As discussed in Section \ref{sec:significance}, our results lay the groundwork for future study on the dynamics of an ensemble of curved dislocations, such as collisions of dislocations with different orientations, external body forces, and understanding the passage from the Peierls--Nabarro model to the elasto-visco-plasticity with density of dislocation at the macroscopic scale.

\subsection{Organization of the paper}

The rest of the paper is organized as follows. 
First, in Section \ref{sec:MC}, we provide the necessary background pertaining to motion by mean curvature and the level set approach. 
Section \ref{sec:frac} contains preliminary results on fractional Laplacians and solutions $u^\ep$.
Then, we establish preliminary results pertaining to the phase transition $\phi$, the corrector $\psi$, and other auxiliary functions needed for the rest of the paper in Section \ref{sec:functions}. 
Section \ref{sec:heuristics} contains heuristics for the evolution of fronts in Theorem \ref{thm:main} and heuristics for the equation for the corrector. 
The construction of barriers is presented in Section \ref{sec:barriers}. 
Section \ref{sec:proof of thm} contains the proof of Theorem \ref{thm:main}.
Lastly, since the proofs of some of auxiliary results in Section \ref{sec:functions} are rather technical, we write them in Sections \ref{sec:estimates}, \ref{sec:ae-updates}, \ref{sec:estimatesbis}, \ref{lem:ae psi estimatesec}. 

Some of the estimates in Section \ref{sec:functions} are used to justify the heuristics in Section \ref{sec:heuristics}. Nevertheless, we encourage the interested reader to study Section \ref{sec:heuristics} in detail before Section \ref{sec:functions}, while referring back for notation and results as necessary.

\subsection{Notations}

In the paper, we will denote by $C>0$ any universal constant depending only on the dimension $n$ and $W$. 

We let $B(x_0,r)$ and $\overline B(x_0,r)$ denote respectively the ball of radius $r>0$ centered at $x_0\in \R^n$ and its closure,  and let $S^n$ denote the unit sphere in $\R^{n+1}$. 

For $\beta \in (0,1]$,  $k \in \N \cup \{0\}$ and $m\in\N$, we denote by $C^{k,\beta}(\R^m)$ the usual class of functions with bounded $C^{k,\beta}$ norm over $\R^m$. 
For $\beta=0$ we simply write $C^{k}(\R^m)$.
The class $H^{\frac{1}{2}}(\R)$ 
is the set of functions $g \in L^2(\R)$ such that
\[
\int_{\R} \int_{\R} \frac{|g(x) - g(y)|^2}{|x-y|^{3}} \, dy \, dx < \infty.
\]
For multi-variable functions $v(\xi;t,x)$, we write $v \in C_\xi^{k,\beta}(\R)$ and $v \in H^{\frac{1}{2}}_{\xi}(\R)$ if, respectively,  $v(\cdot;t,x) \in C^{k,\beta}(\R)$ and $v(\cdot;t,x) \in H^{\frac{1}{2}}(\R)$ for all $t,x$ in the domain of $v$. 
Moreover, we use the dot notation for derivatives with respect to the variable $\xi$, namely $\dot{v}(\xi;t,x) = v_\xi(\xi;t,x)$.

Given a function $\eta = \eta(t,x)$, defined on a set $A$,  we write $\eta = O(\ep)$ if there is $C>0$ such that
$|\eta(t,x)| \leq C \ep$ for all $(t,x)\in A$, and we write $\eta = o_\ep(1)$ if  $\lim_{\ep \to 0} \eta(t,x) = 0$, uniformly in $(t,x)\in A$. 
Given a sequence of functions $u^\ep(t,x)$, we write
\[
\liminf_{\ep \to 0}{_*} u^\ep(t,x) = \inf \left\{ \liminf_{\ep \to 0} u^\ep(t_\ep,x_\ep): ( t_\ep,x_\ep) \to (t,x)\right\}
\]
and
\[
\limsup_{\ep \to 0}{^*} u^\ep(t,x) =\sup \left\{ \limsup_{\ep \to 0} u^\ep(t_\ep,x_\ep): ( t_\ep,x_\ep) \to (t,x)\right\}.
\]
Given vectors  $p,\,q$, we denote by $p\otimes q$ the usual tensor product of $p$ and $q$. 
For a vector $p \in \R^n$, we write $\hat{p} = p/\abs{p}$. 

For a set $A$, we denote by $\one_A$ the characteristic function of the set $A$.

\section{Motion by mean curvature}\label{sec:MC}

In this section, we present the necessary background and preliminaries regarding the mean curvature of an evolving front.

\subsection{The level set approach}\label{sec:levelset}

Here, we review the level set approach for the geometric motions of the fronts. 
For a smooth function $u = u(t,x)$, consider a level set $\Gamma_t = \{x \in \R^n : u(t,x) = \ell\}$ of $u(t,\cdot)$ at level $\ell \in \R$ and assume that $\Gamma_t$ is bounded and $\nabla u$ does not vanish on $\Gamma_t$. 
Let $d(t,x)$ denote the signed distance function to   $\{u>\ell\}$:
\[
d(t,x) = \begin{cases}
	d(x,\Gamma_t) & \hbox{for}~u(t,x) \geq \ell\\
	-d(x,\Gamma_t) & \hbox{for}~u(t,x) < \ell. 
\end{cases}
\]
Note that $n(t,x) := \nabla d(t,x)$ is normal to $\Gamma_t$ and satisfies $|\nabla d| = 1$ in a time-space neighborhood $\mathcal{N}$ of $\Gamma_t$. 
Then, as theorized by Osher--Sethian \cite{OsherSethian} 
and justified by Evans--Spruck in \cite{EvansSpruck}, 
the  level set $\Gamma_t$ moves in the normal direction with velocity $v(t,x)$ proportional to its mean curvature, i.e.,
\begin{equation}\label{eq:velocity}
\begin{aligned}
v(t,x) 
	&= - \mu \operatorname{div}(n(t,x)) n(t,x) \\
	&= -\mu \Delta d(t,x) \nabla d(t,x) \quad \hbox{in}~\mathcal{N}, 
\end{aligned}
\end{equation}
where $\mu>0$ is a constant,
if and only if $u$ is a solution in $\mathcal{N}$ to the \emph{mean curvature equation}
\begin{equation}\label{eq:meancurvature}
\partial_t u = \mu \trace\( (I - \widehat{\nabla u}\otimes\widehat{\nabla u}) D^2u\),
\end{equation} 
where we recall that $\widehat{p} = p/|p|$ for $p \in \R^n$.
Consequently,   if $u$ is smooth and $\nabla u\neq0$ in an open region $O\subset (0,\infty) \times \R^n$, then
 $u$ is a solution to \eqref{eq:meancurvature} in $O$  if and only if 
 all its level sets  move according to their mean curvature \eqref{eq:velocity} in $O$.

 However,  level sets that move by mean curvature may develop singularities in finite time. Indeed, observe that  equation \eqref{eq:meancurvature}  is degenerate,  and it is not well-defined when  $\nabla u=0$. 
To  overcome these difficulties and develop a weak notion of evolving fronts, they use in \cite{EvansSpruck} the theory of viscosity solutions (see \cite{UsersGuide} for more on viscosity solutions). 
More precisely, for a bounded, open set $\Omega_0 \subset \R^n$, set $\Gamma_0 = \partial \Omega_0$ and consider the initial triplet $(\Omega_0, \Gamma_0, (\overline{\Omega_0})^c)$.  
Focusing on the zero level set of the solution, 
let $u_0(x)$ be a continuous function such that
\[
\Omega_0 = \{x : u_0(x)>0\}, \quad
\Gamma_0= \{x : u_0(x)=0\}, \quad \hbox{and} \quad
(\overline{\Omega_0})^c = \{x : u_0(x)<0\}, 
\]
and $u_0$ is a  constant outside a ball. 
Then, there exists a  unique continuous viscosity  solution $u$  to \eqref{eq:meancurvature} in $(0,\infty)\times\R^n$ with initial datum $u(0,x) = u_0(x)$, 
and such that $u$ is constant outside a ball of  $(0,\infty)\times\R^n$. 
 By our previous discussion, the zero level sets of $u$ move according to their mean curvature, at least in the region where $u$ is smooth and $\nabla u\neq 0$. 
Set
\begin{equation}\label{eq:triplets}
{^+}\Omega_t = \{ x : u(t,x)>0\}, \quad
\Gamma_t = \{x : u(t,x)=0\}, \quad \hbox{and} \quad
{^-}\Omega_t = \{ x : u(t,x)<0\}.
\end{equation}
The triplet $({^+}\Omega_t, \Gamma_t, {^-}\Omega_t)_{t \geq 0}$ is called  the \emph{level set evolution} of $(\Omega_0, \Gamma_0, (\overline{\Omega_0})^c)$. 
 As shown in \cite{EvansSpruck}, $\Gamma_t$ does not depend on the particular choice of $u_0$ but only on $\Gamma_0$. 

If $({^+}\tilde \Omega_t, \tilde \Gamma_t, {^-}\tilde \Omega_t)_{t \geq 0}$ denotes the  level set evolution of $(\tilde \Omega_0, \tilde \Gamma_0, (\overline{\tilde\Omega_0})^c)$  with $\tilde \Gamma_0=\partial \tilde \Omega_0$, 
and $\Omega_0 \subset\subset \tilde\Omega_0$,  then ${^+}\Omega_t \subset\subset   {^+}\tilde \Omega_t$ and $ {^-}\tilde \Omega_t  \subset\subset {^-}\Omega_t $ for all times $t>0$ such that $ {^+}\tilde \Omega_t \neq\emptyset$, see \cite[Theorem 7.3]{EvansSpruck}. 
 
In the special setting in which $\Gamma_t$ is smooth for $t\in  [t_0,t_0+h]$, the signed distance function $d$ is a solution to \eqref{eq:meancurvature} in $\{(t,x): t \in [t_0,t_0+h],~x \in \Gamma_t\}$.
More generally, if $u$ solves 
\begin{equation*}
\partial_t u = \mu \trace\( (I - \widehat{\nabla u}\otimes\widehat{\nabla u}) D^2u\)+\sigma
\end{equation*}
in a neighborhood of $\Gamma_t$ for some $\sigma=\sigma(t,x)$, then, since 
$$|\nabla d|=1,\quad   \partial_t d=\frac{\partial_t u}{|\nabla u|},\quad D^2d= \frac{1}{|\nabla u|}\trace\( (I - \widehat{\nabla u}\otimes\widehat{\nabla u}) D^2u\)\quad\text{in }
\bigcup_{t \in [t_0,t_0+h]} \{t\} \times \Gamma_t,$$
the function $d$ solves
\begin{equation}\label{meancurvature_intro_sigma}
\partial_t d = \mu \trace\( (I - \widehat{\nabla d}\otimes\widehat{\nabla d}) D^2d\)+\frac{\sigma}{|\nabla u|}=\mu\Delta d+\frac{\sigma}{|\nabla u|} \quad\text{in }
\bigcup_{t \in [t_0,t_0+h]} \{t\} \times \Gamma_t.
\end{equation}

We remark that around the same time as \cite{EvansSpruck}, Chen--Giga--Goto  independently established a level set approach for degenerate parabolic PDEs \cite{CGG}. 
Shortly after, Barles--Soner--Souganidis  rigorously connected the level set approach and the phase field theory for reaction-diffusion equations \cite{BSS}.

\subsection{Generalized flows}\label{sec:flows}

A different approach for defining the weak evolutions of fronts was considered   in \cite{Barles-Souganidis} and \cite{Barles-DaLio}. 
We present a slight variation of the definition in \cite{Barles-DaLio} (see also \cite{Imbert09}).

Let $F(p,X)$ be given by
\begin{equation}\label{Fdefintro-flow}
F(p,X) = - \mu \trace\(( I - \widehat{p}\otimes\widehat{p}) X\)
\end{equation}
and the lower and upper semi-continuous envelopes of $F$ be denoted by $F_*$ and $F^*$ respectively.

\begin{defn}
\label{defn:flow}
A family $ (D_t)_{t >0}$ (resp., $(E_t)_{t >0}$) of open (resp., closed)  subsets of $\R^n$ is a \emph{generalized super-flow (resp., sub-flow)} of the mean curvature equation \eqref{eq:meancurvature} 
if for all $(t_0,x_0) \in (0,\infty) \times \R^n$,  $h,\, r>0$, and for all smooth functions $\varphi:(0,\infty) \times \R^n \to \R$ such that
\begin{enumerate}[start=1,label={(\roman*)}]
\item \label{item:i}
(Boundedness) for all $t\in[t_0,t_0+h]$, the set
 $$
\{ x \in \R^n : \varphi(t,x) >0\}\quad (\text{resp. }\{ x \in \R^n : \varphi(t,x) < 0\})
$$ is bounded and 
$$
\{ x \in \overline{B}(x_0,r)  : \varphi(t,x) >0\}\quad (\text{resp. }\{ x\in \overline{B}(x_0,r)  : \varphi(t,x) < 0\})
$$
is non-empty,
\item \label{item:ii}
 (Speed) 
there exists $\delta =\delta(\varphi)>0$ such that
\begin{align*}
\partial_t\varphi + F^*(\nabla\varphi,D^2\varphi) &\leq -\delta \quad \hbox{in}~[t_0, t_0 + h] \times \overline{B}(x_0,r),\\
(\text{resp., }\partial_t\varphi + F_*(\nabla\varphi,D^2\varphi) &\geq \delta)
\end{align*}
\item  \label{item:iii}
(Non-degeneracy)
\[
\nabla\varphi \not= 0 \quad \hbox{on}~\{ (t,x) \in [t_0, t_0 + h] \times  \overline{B}(x_0,r) : \varphi(t,x) = 0\},
\] 
\item \label{item:iv}
(Initial condition)
\begin{align*}
\{x \in \R^n: \varphi(t_0,x) \geq 0\} &\subset D_{t_0}\\
(\text{resp., }\{x \in \R^n: \varphi(t_0,x) \leq 0\} &\subset \R^n \setminus E_{t_0}),
\end{align*}
\item \label{item:v}
(Boundary condition)  for all $t\in[t_0,t_0+h]$,
\begin{align*} 
\{x \in \R^n\setminus B(x_0,r) : \varphi(t,x) \geq 0\} &\subset  
D_t\\
(\text{resp., }\{x \in  \R^n\setminus B(x_0,r) : \varphi(t,x) \leq 0\} &\subset \R^n \setminus E_{t}),
\end{align*}
\end{enumerate}
it holds that
\begin{align*}
\{x \in B(x_0,r): \varphi(t_0+h,x)>0\}& \subset D_{t_0+h}\\
(\text{resp., }\{x \in B(x_0,r): \varphi(t_0+h,x)<0\}& \subset \R^n \setminus E_{t_0+h}).
\end{align*}
\end{defn}

In this paper, we will implement the abstract method developed in  \cite{Barles-DaLio}-\cite{Barles-Souganidis}. Precisely, for $i=1,\ldots,N$, let $\Omega_0^i$ be the open sets defined in \eqref{eq:initial d_i}. We will show that there exist families of  open sets $(D_t^i)_{t\geq 0}$ and  $(E_t^i)_{t\geq 0}$,  $i=1,\ldots, N$, such that the following hold:
$(D_t^i)_{t\geq 0}$ and $((E_t^i)^c)_{t\geq 0}$ are respectively a generalized super and  sub-flow  of the mean curvature equation 
 with $\mu$ as  in \eqref{eq:velocity-intro},   $\Omega_0^i \subset D_0^i$, $(\overline{\Omega}_0^i)^c \subset E_0^i$,  and 
$$u^\ep(t,x)\to i\quad\text{if }x\in  D_t^i\cap   E_t^{i+1}.$$
 Therefore, if  $({^+}\Omega^i_t, \Gamma^i_t, {^-}\Omega^i_t)_{t \geq 0}$ denotes the level set evolution of $(\Omega^i_0, \Gamma^i_0, (\overline{\Omega^i_0})^c)$,   by 
 \cite[Theorem 1.1, Theorem 1.2]{Barles-DaLio} (see also   \cite[Proposition 2.1, Theorem 2.4]{Barles-Souganidis}), 
 \[
{^+}\Omega_t^i \subset D_t^i \subset {^+}\Omega_t^i \cup \Gamma_t^i
\quad \hbox{and} \quad
{^-}\Omega_t^i \subset E_t^i \subset {^-}\Omega_t^i \cup \Gamma_t^i.
\]
In particular,  if the sets $(\Gamma^i_t)_{t \geq0}$ do not develop interior, 
 then 
 $$ {^+}\Omega^i_t=D^i_t \quad \hbox{and} \quad {^-}\Omega^i_t = E_t^i.$$
 Our main result, Theorem \ref {thm:main}, immediately follows. 

\subsection{Extension of the signed distance function}

Recall that the signed distance function $d = d(t,x)$ associated to the front $\Gamma_t$ in \eqref{eq:triplets} is smooth in some neighborhood $Q_\rho = \{|d|<\rho\}$ of the front, provided $\Gamma_t$ is smooth. 
However, in general, $d$ is not smooth away from the front. 
Throughout the paper, we will use the following smooth extension of the distance function away from $\Gamma_t$. 

\begin{defn}[Extension of the signed distance function]\label{defn:extension}
For $t\in[t_0,t_0+h]$, let $\tilde{d}$ be the signed distance function from a bounded domain  $\Omega_t$ with smooth boundary  $\Gamma_t$  and let   $\rho>0$ be such that $\tilde{d}$  is smooth in 
\[
Q_{2\rho} = \{(t,x)\in [t_0,t_0+h] \times \R^n: |\tilde{d}(t,x)| \leq 2\rho\}.
\]
Let $\eta(t, x)$ be a smooth, bounded function such that 
\[
\eta = 1~\hbox{in}~\{|\tilde{d}|\leq \rho\}, \quad \eta = 0~\hbox{in}~\{|\tilde{d}|\geq 2\rho\},\quad 0 \leq \eta \leq 1.
\]
We extend $\tilde{d}(t,x)$ in the set $\{(t,x)\in [t_0,t_0+h] \times \R^n: |\tilde{d}(t,x)|>\rho\}$ with the smooth bounded function $d(t,x)$ given by
\[
d(t,x) 
= \begin{cases}
\tilde{d}(t,x) & \hbox{in}~Q_\rho=\{|\tilde{d}(t,x)|\leq \rho\}\\
\tilde{d}(t,x) \eta(t,x) + 2\rho(1-\eta(t,x)) & \hbox{in}~\{\rho < \tilde{d}(t,x) < 2 \rho\}\\
\tilde{d}(t,x) \eta(t,x) -2\rho(1-\eta(t,x)) & \hbox{in}~\{-2\rho < \tilde{d}(t,x) < - \rho\}\\
2\rho &\hbox{in}~\{\tilde{d}(t,x)\geq 2\rho\}\\
-2\rho &\hbox{in}~\{\tilde{d}(t,x)\leq -2\rho\}.
\end{cases}
\]
Notice that, in $\{\rho < \tilde{d} < 2\rho\}$, the function $d$ satisfies
\[
d = 2\rho + (\tilde{d}-2\rho)\eta > 2 \rho - \rho \eta \geq \rho,
\]
and, in $\{-2\rho < \tilde{d} <-\rho\}$, the function $d$ satisfies
\[
d = -2\rho + (\tilde{d}+2\rho)\eta <- 2 \rho + \rho \eta \leq -\rho.
\]
\end{defn}

\section{Preliminary results}\label{sec:frac}

In this section, we recall a few basic properties of the operator $\mathcal{I}_n $ and solutions to \eqref{eq:pde} that will be used later on in the paper.

\subsection{Fractional Laplacians}

Let $u \in C^{1,1}(\R^n)$.
First of all, by using that 
$$\operatorname{P.V.} \int_{\{\abs{y}<R\}}\nabla u(x) \cdot y \,\frac{dy}{\abs{y}^{n+1}}=0,$$
for any $R>0$, we can write
\begin{align*}
\mathcal{I}_n u(x) 
	&= \int_{\{\abs{y}<R\}} \( u(x+y) - u(x) - \nabla u(x) \cdot y\) \,\frac{dy}{\abs{y}^{n+1}}
		+ \int_{\{\abs{y}>R\}} \( u(x+y) - u(x)\) \,\frac{dy}{\abs{y}^{n+1}}.
\end{align*}
In particular if $u$ is smooth and bounded, both integrals above are finite and we can bound $\mathcal{I}_n u$ as follows,
\begin{equation}\label{Iubound-C11}|\mathcal{I}_n u(x)|\leq c_n\left(\|D^2u\|_\infty R+\frac{\|u\|_\infty}{R}\right), \end{equation}
where we used the following lemma whose proof  is a  direct computation in polar coordinates. 

\begin{lem}\label{kernellemma3}
There exists $C_1,\, C_2>0$, only depending on $n$, such that for any $R>0$,
$$\int_{\{|z|<R\}}\frac{dz}{|z|^{n-1}}= C_1 R
\quad \hbox{and} \quad \int_{\{|z|>R\}}\frac{dz}{|z|^{n+1}}= \frac{C_2}{R}.$$
\end{lem}

We will frequently use Lemma \ref{kernellemma3} throughout the paper without reference.

The next is an auxiliary lemma that allows us to view one-dimensional fractional Laplacians of functions defined over $\R$   equivalently as $n$-dimensional fractional Laplacians. 

\begin{lem} \cite[Lemma 3.2]{PatriziVaughan}\label{lem:one to n}
For a vector $e \in \R^n$ and a function $v\in C^{1,1}(\R)$, let $v_e(x) = v(e\cdot x): \R^n \to \R$. Then,
\[
\mathcal{I}_n[v_e](x) =|e| C_n \mathcal{I}_1[v](e\cdot x)
\]
where
\begin{equation}\label{eq:Cn}
C_n =  \int_{\R^{n-1}} \frac{1}{(\abs{y}^2 + 1)^{\frac{n+1}{2}} } \, dy.
\end{equation}
 Consequently,
\begin{equation*}\label{eq:one to n}
|e| C_n \mathcal{I}_1[v](\xi) =\operatorname{P.V.} \int_{\R^n} \( v(\xi + e \cdot z) - v(\xi)\) \frac{dz}{\abs{z}^{n+1}}, \quad \xi \in \R.
\end{equation*}
\end{lem}

\subsection{Properties of solutions to \eqref{eq:pde}}

Here, we state existence, uniqueness, and comparison principles for viscosity solutions to \eqref{eq:pde} for a fixed $\ep>0$. 

First, the following comparison principles can be found in \cite{JakobsenKarlsen} and will be used throughout the paper without reference. 
For the definition of viscosity subsolutions, supersolution, and solutions, see also \cite{UsersGuide}.
For ease, we denote by $USC_b([t_0,t_0+h] \times \R^n)$ (resp.~$LSC_b([t_0,t_0+h] \times \R^n)$) the set of upper (resp.~lower) semicontinuous functions on $[t_0,t_0+h] \times \R^n$ which are bounded on $[t_0,t_0+h] \times \R^n$. 

\begin{prop}[Comparison principle in $\R^n$]
Fix $\ep>0$. 
If $u \in USC_b([t_0,t_0+h] \times \R^n)$ is a viscosity subsolution and $v \in LSC_b([t_0,t_0+h] \times \R^n)$ is a viscosity supersolution of \eqref{eq:pde} such that $u(t_0,\cdot) \leq v(t_0,\cdot)$
on $\R^n$, then $u \leq v$ on $[t_0,t_0+h] \times \R^n$. 
\end{prop}

\begin{prop}[Comparison principle in bounded domains]
Fix $\ep>0$ and let $\Omega \subset \R^n$ be a bounded domain. 
If $u \in USC_b([t_0,t_0+h] \times \R^n)$ is a viscosity subsolution and $v \in LSC_b([t_0,t_0+h] \times \R^n)$ is a viscosity supersolution of \eqref{eq:pde} such that $u(t_0,\cdot ) \leq v(t_0,\cdot )$ on $\R^n$ and $u \leq v$ on $[t_0, t_0+h] \times (\R^n \setminus \Omega)$, then $u \leq v$ on $[t_0,t_0+h] \times \R^n$. 
\end{prop}

Next, we prove existence and uniqueness of viscosity solutions. 

\begin{prop}[Existence and uniqueness]
Fix $\ep>0$ and let $u_0 \in C^{1,1}(\R^n)$. 
There exists a unique viscosity solution $u^\ep \in C([0,\infty) \times \R^n) \cap L^{\infty}([0,\infty)\times \R^n)$  to \eqref{eq:pde} with initial data $u^\ep(0,x) = u_0(x)$. 
\end{prop}

\begin{proof} 
Since $u_0 \in C^{1,1}(\R^n)$,  by \eqref{Iubound-C11} with $R=1$, the functions
\[
u^{\pm}(t,x) := u_0(x) \pm \frac{Ct}{\ep^2 |\ln \ep|}
\]
are a supersolution and subsolution of \eqref{eq:pde}, respectively, if $C \geq \ep c_n\|u_0\|_{C^{1,1}(\R^n)} + \|W'\|_{L^\infty(\R)}$. 
Noting that $u^\pm(0,x) = u_0(x)$, existence of a unique continuous viscosity  solution $u^\ep$ follows by Perron's method and the above comparison principle in $\R^n$. 

Let $N\in \N$ be such that $\|u_0\|_{L^\infty(\R^n)}\leq N$. Since $W'=0$ on $\N$, the constant functions $\pm N$ are solutions to   \eqref{eq:pde}. By the comparison principle in $\R^n$, $-N\leq u^\ep\leq N$
on $[0,\infty) \times \R^n$, thus $u^\ep\in L^{\infty}([0,\infty)\times\R^n)$. 
\end{proof}

\section{The phase transition, the corrector, and the auxiliary functions}\label{sec:functions}

In this section, we will introduce the phase transition $\phi$ and the corrector $\psi$.
Along the way, we will also define the auxiliary functions $a_{\ep}$ and $\bar{a}_{\ep}$ and exhibit their relationship with fractional Laplacians and the mean curvature equation, respectively. 

\subsection{The phase transition $\phi$}

Let $\phi$ be the solution to \eqref{eq:standing wave}. 
For convenience in the notation, let $c_0$ and $\alpha$ be given respectively by
\begin{equation}\label{eq:c0-gamma}
c_0^{-1} = \int_\R [\dot{\phi}(\xi)]^2 \, d \xi \quad \hbox{and} \quad \alpha = \frac{W''(0)}{C_n}
\end{equation}
where $C_n$ is defined in \eqref{eq:Cn}, and let $H(\xi)$ be the Heaviside function. 

\begin{lem}\label{lem:asymptotics}  
There is a unique solution $\phi \in C^{4, \beta}(\R)$ of \eqref{eq:standing wave}, with $\beta$ as in \eqref{eq:W}. Moreover, there exists a constant $C>0$  such that
\begin{equation}\label{eq:asymptotics for phi}
\abs{\phi(\xi) - H(\xi) + \frac{1}{\alpha \xi}} \leq \frac{C}{\abs{\xi}^{2}}, \quad \text{for }\abs{\xi} \geq 1
\end{equation}
and
\begin{equation}\label{eq:asymptotics for phi dot}
\frac{1}{C\abs{\xi}^{2}}\leq \dot{\phi}(\xi) \leq \frac{C}{\abs{\xi}^{2}}, \quad 
|\ddot{\phi}(\xi)| \leq \frac{C}{\abs{\xi}^{2}},\quad |\dddot{\phi}(\xi)| \leq \frac{C}{\abs{\xi}^{2}}, \quad \text{for }\abs{\xi} \geq 1.
\end{equation}
\end{lem}

\begin{proof}
Existence of a unique solution $\phi\in C^{2,\beta}(\R)$ of \eqref{eq:standing wave} is proven in \cite{CabreSola}, see Theorem 1.2 and Lemma 2.3, and also in \cite{PSV}*{Theorem 2}.
The regularity of $W$ in \eqref{eq:W} implies that  $\phi$ is actually in $C^{4,\beta}(\R)$, see \cite{CabreSola}*{Lemma 2.3}. 
Estimate \eqref{eq:asymptotics for phi} and the estimate on $\dot\phi$ in \eqref{eq:asymptotics for phi dot} are proven in \cite{GonzalezMonneau}*{Theorem 3.1}. Finally,
the estimates on $\ddot\phi$ and $\dddot\phi$  in \eqref{eq:asymptotics for phi dot} are  proven in \cite{MonneauPatrizi2}*{Lemma 3.1}.
\end{proof}

\subsection{The auxiliary functions $a_{\ep}$ and $\bar{a}_{\ep}$}

Now, we will introduce two auxiliary functions that are necessary for our analysis.
Let $\Omega_t$ be a bounded domain with smooth boundary $\Gamma_t$, for $t\in[t_0,t_0+h]$.
Throughout this section, let $d = d(t,x)$ and $\rho$ be such that  $d$ is the smooth extension of the signed distance function $\tilde d$  to  $\Gamma_t$ outside of $Q_\rho$ as given in Definition \ref{defn:extension}. 
We also introduce a new parameter $0<\gamma<1$.
 
Define the function $a_{\ep}=a_{\ep}(\xi; t,x)$ by
\begin{equation}\label{aepsilondef}
a_{\ep}= \int_{\{|z|<\frac{\gamma}{\ep}\}} \( \phi\(\xi + \frac{d(t,x+\ep z)- d(t,x)}{\ep}\) - \phi\( \xi + \nabla d(t,x) \cdot z\) \) \frac{dz}{\abs{z}^{n+1}},
\end{equation}
where $(\xi,t,x) \in \R \times [t_0,t_0+h] \times \R^n$. By Lemma \ref{kernellemma3} and by the regularity of $\phi$ and $d$, the integral in \eqref{aepsilondef} is well defined. 
We note that the integral is taken over $\{|z|<\frac{\gamma}{\ep}\}$ instead of $\R^n$  so that $a_\ep$ is sufficiently small for $(t,x)$ close to $\Gamma_t$ and large $\xi$ (see \eqref{aeclosetofront-xi}) and its derivatives in $t$ and $x$ remain bounded (see the proof of \eqref{aderivativesLinfinityestimate}).

The corresponding function $\bar{a}_{\ep} = \bar{a}_{\ep}(t,x)$ is given by 
\begin{equation} \label{eq:a-epsilon bar}
\bar{a}_{\ep}(t,x) = \frac{1}{\ep \abs{\ln \ep}} \int_{\R}a_{\ep}\(\xi;t,x\) \dot{\phi}(\xi) \, d \xi.
\end{equation}
One of the main observations in \cite{Imbert} is that $\bar{a}_{\ep}(t,x)$ converges,  up to constants, to $\Delta d(t,x)$ in the neighborhood $Q_\rho$, and therefore to the mean curvature of $\Gamma_t$ at points $x\in\Gamma_t$, see \cite{Imbert}*{Lemma 4}. 
In our setting, we use the recent result in \cite{PatriziVaughan}.

\begin{thm}\label{lem:4} 
For $t\in[t_0,t_0+h]$, let $\Omega_t$ be a bounded domain with smooth boundary $\Gamma_t$. 
Let $d$ be as in Definition \ref{defn:extension}. Then,
\begin{align*}
\lim_{\ep \to 0} \bar{a}_{\ep}(t,x)
	&= \frac{1}{2} \frac{|S^{n-2}|}{n-1} \Delta d(t,x)
	=\mu c_0^{-1} \trace\( (I - \widehat{\nabla d} \otimes \widehat{\nabla d}) D^2d\)
\end{align*}
uniformly in $(t,x)  \in Q_{\rho}$ where $\mu>0$ is in \eqref{eq:velocity-intro}.  
\end{thm}

\begin{proof}
This is \cite[Theorem 1.3]{PatriziVaughan} (see also Remark 1.1 there) when $\gamma = +\infty$. 
For $0 < \gamma<1$, the proof is the same after making a minor adaptation to \cite[Lemma 4.1]{PatriziVaughan}. 
\end{proof}

We will also need several estimates on $a_\ep, \bar{a}_\ep$ and their derivatives, none of which are proven in \cite{Imbert}. 
These will be used in the next subsection to prove asymptotic properties of the corrector (which are assumed in \cite{Imbert}). 
First, we have the following general estimate. 
The proof is delayed until Section \ref{sec:estimates}.

\begin{lem}\label{lem:ae bound}
There exists  $C>0$ such that, for all $(\xi,t,x) \in \R \times [t_0,t_0+h]\times \R^n$, 
\begin{equation}\label{aLinfinityestimate}
\|a_\ep\|_{C_\xi^{3}(\R)} \leq C \ep^\frac{1}{2},
 \end{equation} 
 \begin{equation}\label{aderivativesLinfinityestimate}
\|\partial_t a_{\ep}\|_{C_\xi^{3}(\R)} ,\, \|\nabla_x a_{\ep}\|_{C_\xi^{3}(\R)} \leq  C |\ln \ep|, \quad \|D_x^2 a_{\ep}\|_{C_\xi^{3}(\R)} \leq \frac{C}{\ep},
 \end{equation} 
 \begin{equation}\label{aL2estimatefar}
 \abs{a_{\ep}\(\xi;t,x\)},\,\  \abs{\dot a_{\ep}\(\xi;t,x\)}\leq \frac{C}{1+|\xi|},
\end{equation}
 \begin{equation}\label{aL2estimatefarxtder}\begin{split}
&\abs{\partial_t a_{\ep}\(\xi;t,x\)},\, \abs{\nabla_x a_{\ep}\(\xi;t,x\)} \leq \frac{C}{\ep(1+|\xi|)},\quad  \abs{D_x^2 a_{\ep}\(\xi;t,x\)} \leq \frac{C}{\ep^2(1+|\xi|)},\\&
\abs{\partial_t \dot a_{\ep}\(\xi;t,x\)},\, \abs{\nabla_x \dot a_{\ep}\(\xi;t,x\)}\leq \frac{C}{\ep(1+|\xi|)},\quad \abs{D_x^2 \dot a_{\ep}\(\xi;t,x\)}\leq \frac{C}{\ep^2(1+|\xi|)}.
\end{split}\end{equation}
 Consequently, for all $(t,x) \in [t_0,t_0+h] \times \R^n$,
\begin{equation}
\label{abarestimate}\abs{\bar{a}_{\ep}(t,x)} \leq \frac{C}{\ep^\frac{1}{2} \abs{\ln \ep}},
\end{equation}
and 
\begin{equation}
\label{abarestimatextder}\abs{\partial_t \bar{a}_{\ep}(t,x)},\, \abs{\nabla_x  \bar{a}_{\ep}(t,x)} \leq \frac{C}{\ep}, \quad \abs{D_x^2 \bar{a}_{\ep}(t,x)} \leq \frac{C}{\ep^2|\ln \ep|}.
\end{equation}
\end{lem}

We also have the following refined estimates near the front $\Gamma_t$. 
The proof is in Section \ref{sec:ae-updates}.

\begin{lem}\label{lem:ae-updates}
There exists $C>0$ such that,  for $(\xi,t,x) \in \R \times [t_0,t_0+h]\times \R^n$, if $|\nabla d(t,x)| \not=0$, then
\begin{equation}\label{eq:aeestnonzerograd}
\begin{split}
&|a_\ep(\xi;t,x)|, \, |\dot{a}_\ep(\xi;t,x)|, \, |\ddot{a}_\ep(\xi;t,x)| \leq \frac{C \ep |\ln \ep|}{|\nabla d(t,x)|},\\
&|\bar{a}_\ep(t,x)| \leq \frac{C}{|\nabla d(t,x)|}.
\end{split}
\end{equation}
Moreover, if $(t,x)\in Q_\rho$,  then
\begin{equation}\label{aderivativesclosetothefront}
\begin{split}
&|\partial_ta_{\ep}(\xi;t,x)|,\, |\nabla_x  a_{\ep}(\xi;t,x)|,\, |\nabla_x \dot a_{\ep}(\xi;t,x)| \leq C ,\\
&|\partial_t\bar{a}_{\ep}(t,x)|,\, |\nabla_x  \bar{a}_{\ep}(t,x)| \leq \frac{C}{\ep |\ln \ep|},
\end{split}
\end{equation}
and, there exists $\ep_0=\ep_0(\gamma)$, with $\gamma$ as in \eqref{aepsilondef}, such that for $0 < \ep<\ep_0$, 
\begin{equation}\label{aeclosetofront-xi}
|a_\ep(\xi;t,x)| \leq \frac{C\gamma}{1+|\xi|}. 
\end{equation}
\end{lem}

It is also important to notice that when $\xi=d(t,x)/\ep$, morally, $a_{\ep}(d/\ep)$ is the difference between an $n$-dimensional and a $1$-dimensional fractional Laplacian of $\phi(d/\ep)$. 
This is made precise in the following lemma whose proof is delayed until Section \ref{sec:estimatesbis}.

\begin{lem}\label{lem:ae near front}
There is  $C = C(\gamma,\rho)>0$ such that 
\[
\left|a_{\ep}\(\frac{d(t,x)}{\ep} ; t,x\)- \left(\ep \mathcal{I}_n\left[\phi\(\frac{d(t,\cdot)}{\ep}\) \right](x) - C_n \mathcal{I}_1[\phi]\(\frac{d(t,x)}{\ep}\) \right)\right| 
	\leq C\ep
\]
for any $(t,x) \in [t_0,t_0+h] \times \R^n$. 
\end{lem}
 
 \subsection{The corrector $\psi$}\label{sec:corrector}

Next, we introduce the corrector that is required to solve the equation when constructing barriers for the proof of Theorem \ref{thm:main}. For this, we use
the linearized operator $\mathcal{L}$ associated to \eqref{eq:standing wave} around $\phi$ that is 
given by
\begin{equation}\label{eq:linearized operator}
\mathcal{L}[\psi] = -C_n\mathcal{I}_1[\psi] + {W}''(\phi) \psi.
\end{equation}
The correction $\psi = \psi(\xi;t,x)$ that we will need solves
\begin{equation}\label{eq:linearized wave}
\begin{cases}
\displaystyle{
\mathcal{L}[\psi(\cdot;t,x)](\xi)} 
	=  \frac{a_{\ep}\(\xi;t,x\)}{\ep \abs{\ln \ep}}
		+ c_0 \dot{\phi}\( \xi\) (\sigma -\bar{a}_{\ep}(t,x)) 
		+ \tilde{\sigma} \(W''\( \phi (\xi)\) - W''(0)\),
		 &\xi \in \R\\
\psi(\pm\infty;t,x) = 0,
\end{cases}
\end{equation}
for all $x \in \R^n$ and $t \in [t_0,t_0+h]$,
where  $\sigma\in\R$ is a small constant and $\tilde{\sigma}$ is such that $\sigma = W''(0)\tilde{\sigma}$.
See Section \ref{sec:heuristics} for a formal derivation of \eqref{eq:linearized wave}. 

The rest of this section is devoted to establishing existence, uniqueness, and properties of $\psi$. First, to prove existence and uniqueness of the solution of
\eqref{eq:linearized wave}, we will use the following result from  
\cite{GonzalezMonneau} regarding the linearized equation with a general right-hand side. 

\begin{lem}\label{psiexistencemonneau}
Let $g\in  H^{\frac{1}{2}}(\R)$ be such that 
\begin{equation*}\label{gcond}\int_{\R}g(\xi)\dot\phi(\xi)\,d\xi=0.
		\end{equation*}
		Then, there exists a unique solution $\psi \in H^{\frac{1}{2}}(\R)$ such that  $\int_{\R}\psi(\xi)\dot\phi(\xi)\,d\xi=0$ of 
\begin{equation*}
\mathcal{L}[\psi] 
	=  g \quad  \text{in }\R.
\end{equation*}
Moreover, if $g\in C^{2,\beta}(\R)$ with $\beta$ as in \eqref{eq:W}, then $\psi\in C^{3,\beta}(\R)$, and 
\begin{equation}\label{h12normpsiLinfty}
\|\psi\|_{ C^{3,\beta}(\R) }\leq C( \|g\|_ { C^{2,\beta}(\R)}+\|\psi\|_{L^\infty(\R)}). 
\end{equation}
\end{lem}
\begin{proof}
The proof of existence of a unique solution $\psi\in H^{\frac{1}{2}}(\R)$ satisfying $\int_{\R}\psi(\xi)\dot\phi(\xi)\,d\xi=0$  is contained in the proof of  
\cite{GonzalezMonneau}*{Theorem 3.2}. If  $g\in C^{2,\beta}(\R)\cap H^{\frac{1}{2}}(\R)$, then for all $p\geq 2$, 
$$\|g\|_{L^p(\R)}\leq C$$ and the boundedness of $\psi$ follows from \cite{GonzalezMonneau}*{Corollary 5.16}. The $C^{3,\beta}$ regularity of $\psi$ and \eqref{h12normpsiLinfty} are  a consequence of  the regularity of $\phi$ given by Lemma \ref{lem:asymptotics}  and  \cite{CabreSola}*{Lemma 2.3}. 
\end{proof}

We now show that the right-hand side in \eqref{eq:linearized wave} satisfies the hypothesis of Lemma \ref{psiexistencemonneau}.

\begin{lem}\label{ghortogonaltophidot}
Let 
\begin{equation}\label{gdef}g(\xi;t,x):= \frac{a_{\ep}\(\xi;t,x\)}{\ep \abs{\ln \ep}}
		+ c_0 \dot{\phi}\( \xi\) (\sigma -\bar{a}_{\ep}(t,x)) 
		+ \tilde{\sigma} \(W''\( \phi (\xi)\) - W''(0)\).
\end{equation}		
		Then, 
		\begin{equation}\label{gcond}\int_{\R}g(\xi;t,x)\dot\phi(\xi)\,d\xi=0.
		\end{equation}
		Moreover, 
		$g\in H_\xi^\frac12(\R)\cap C_\xi^{2,\beta}(\R)$  with $\beta$ as in \eqref{eq:W}, uniformly  in $(t,x)\in  [t_0,t_0+h] \times\R^n.$
\end{lem}
\begin{proof}
Recalling that $\sigma = W''(0)\tilde{\sigma}$, we compute
\begin{align*}
\int_{\R} g(\xi;t,x)\dot\phi(\xi)\,d\xi
	&=  \int_{\R} \(\frac{1}{\ep \abs{\ln \ep}} a_{\ep}\(\xi;t,x\)- \dot{\phi}\(\xi\)c_0 \bar{a}_{\ep}(t,x) \) \dot{\phi}(\xi)\, d \xi\\
	&\quad + \int_{\R}\(c_0 \sigma  \dot{\phi}\(\xi\)+ \tilde{\sigma} W''\( \phi \(\xi\)\) - \sigma\)\dot{\phi}(\xi) \, d \xi.
\end{align*}
Using the definitions of $c_0$  and $\bar{a}_{\ep}$ in \eqref{eq:c0-gamma}  and \eqref{eq:a-epsilon bar}, respectively,  we get
\begin{align*}
 \int_{\R}\bigg(\frac{1}{\ep \abs{\ln \ep}} a_{\ep}\(\xi;t,x\)
		- \dot{\phi}\( \xi\)c_0 \bar{a}_{\ep}(t,x)\bigg) \dot{\phi}(\xi)\, d \xi
		&= \frac{1}{\ep \abs{\ln \ep}}  \int_{\R}a_{\ep}\(\xi;t,x\) \dot{\phi}(\xi) \, d \xi 
		-\bar{a}_{\ep}(t,x)
		= 0.
\end{align*}
Then, we use that $W'$ is periodic, $\phi(\infty)=1$ and  $\phi(-\infty)=0$ to find that
\begin{align*}
 \int_{\R}  \tilde{\sigma}W''\(\phi(\xi)\) \dot{\phi}(\xi) \, d \xi
 	=  \tilde{\sigma} \int_{\R} \frac{d}{d\xi}[W'\(\phi(\xi)\)] \, d \xi 
	&= \tilde{\sigma} [W'(1) - W'(0)] = 0
\end{align*}
and again the definition of $c_0$ to see that
\begin{align*}
\int_{\R} \(c_0 \sigma \dot{\phi}(\xi) - \sigma \) \dot{\phi}(\xi)\, d \xi
	&= c_0 \sigma\int_{\R} [ \dot{\phi}(\xi)]^2 \, d \xi - \sigma
	=0,
\end{align*}
as desired. This proves \eqref{gcond}. 

Next, from  \eqref{eq:asymptotics for phi}, \eqref{eq:asymptotics for phi dot} and  \eqref{aL2estimatefar}  we have that
$$ W''\( \phi (\xi)\) - W''(0)=O(\phi (\xi)),\,  \dot{\phi}\( \xi\),\,a_{\ep}\(\xi;t,x\)
\in  H^{1}_{\xi}(\R),$$
which implies that $g\in  H^{\frac12}_{\xi}(\R)$,  with $\|g\|_{ H^{\frac12}_{\xi}(\R)}\leq C_\ep$ for all $(t,x)\in [t_0,t_0+h]\times\R^n$. 
Moreover, from the regularity of $\phi$ and $W$ and \eqref{aLinfinityestimate}, it follows that $g\in C_\xi^{2,\beta}(\R)$ with $\|g\|_{C_\xi^{2,\beta}(\R)}\leq C_\ep$,  
for all $(t,x)\in  [t_0,t_0+h]\times\R^n$.
\end{proof}

By the previous two lemmas, we have existence and uniqueness of the solution $\psi$ to \eqref{eq:linearized wave}. 
We mention again that existence and uniqueness of the corrector are assumed in \cite{Imbert}. 

They also assume in \cite{Imbert} that the bounds on the corrector and its derivatives and the asymptotic behavior at infinity are all uniform in $\ep>0$. 
However, 
this appears to be too strong of an assumption.
Using Lemmas \ref{lem:ae bound} and \ref{lem:ae-updates}, 
we prove asymptotic estimates on $\psi$ and its derivatives as $\ep \to 0$ that are needed for the proof of Theorem \ref{thm:main}.  In particular, we require global estimates and also more refined estimates for $(t,x) \in Q_\rho$.

\begin{thm}\label{lem:psi-reg}
There is a unique solution $\psi = \psi(\xi;t,x)\in C_\xi ^{3,\beta}(\R)$ to \eqref{eq:linearized wave} with $\beta$ as in \eqref{eq:W}, and $C>0$ such that, for all 
 $(\xi, t,x)\in \R\times [t_0,t_0+h] \times\R^n$,
 \begin{equation}\label{psiLinfinityasymptoticsfar}
 \begin{split}
& |\psi(\xi;t,x)|,\,|\dot{\psi}(\xi;t,x)|,\,|\ddot{\psi}(\xi;t,x)|\leq \frac{C}{\ep^{\frac12}|\ln\ep|},\\&
|\partial_t\psi(\xi;t,x)|,\,|\nabla_x\psi(\xi;t,x)|,\,|\nabla_x\dot\psi(\xi;t,x)|\leq \frac{C}{\ep},\\
&|D^2_x\psi(\xi;t,x)|\leq \frac{C}{\ep^2|\ln\ep|},
\end{split}
\end{equation}
and 
\begin{equation}\label{psiasymptoticsfar}| \psi(\xi;t,x)|\leq\frac{C}{\ep|\ln\ep|(1+|\xi|)}. 
\end{equation}  
Moreover, if $(t,x)\in Q_\rho$, then
 \begin{equation}\label{psiLinfinityasymptoticnewclosefront}\begin{split}
 & |\psi(\xi;t,x)|,\,|\dot{\psi}(\xi;t,x)|,\,|\ddot{\psi}(\xi;t,x)|\leq C,\\&
 |\partial_t\psi(\xi;t,x)|,\,|\nabla_x\psi(\xi;t,x)|,\,|\nabla_x\dot\psi(\xi;t,x)|\leq \frac{C}{\ep|\ln\ep|}
\end{split}
 \end{equation} 
 and, there exists $\ep_0=\ep_0(\gamma)$, with $\gamma$ as in \eqref{aepsilondef}, such that for $\ep<\ep_0$, 
 \begin{equation}\label{psiasymptoticsveryfarnew}| \psi(\xi;t,x)|\leq\frac{C\gamma}{\ep|\ln\ep|(1+|\xi|)},
\end{equation}  
with $\gamma$ as in \eqref{aepsilondef}. 
\end{thm}

\begin{proof}
First, the existence of a unique solution $\psi\in H^\frac12_\xi(\R)\cap C_\xi^{3,\beta}(\R)$ to \eqref{eq:linearized wave} satisfying $\int_{\R}\psi(\xi;t,x)\dot\phi(\xi)\,d\xi=0$ follows  from Lemmas  \ref{psiexistencemonneau} and  \ref{ghortogonaltophidot}.

Let us show  \eqref{psiLinfinityasymptoticsfar} for $\psi$. Let $g$ be defined as in \eqref{gdef}.  We first notice  that by \eqref{aLinfinityestimate} and \eqref{abarestimate}, we have that, for all $ (t,x)\in  [t_0,t_0+h] \times\R^n$
\begin{equation}\label{psithmexiep12bound1}\|g(\cdot;t,x)\|_{C^{2,\beta}(\R)}\leq \frac{C}{\ep^\frac12|\ln\ep|}.\end{equation}
Suppose by contradiction that there is a sequence $\ep_k\to0$ as $k\to\infty$ such that if $\psi_k$ is solution to \eqref{eq:linearized wave} with $\ep=\ep_k$, then for some 
$(t_k,x_k)\in [t_0,t_0+h] \times\R^n$,  
\begin{equation}\label{psithmexiep12bound2}\|\psi_k(\cdot;t_k,x_k)\|_{L^\infty(\R)}\geq \frac{1}{\ep_k^{\frac12}|\ln\ep_k| b_k }\end{equation} with $0<b_k\to0$ as $k\to\infty$. 
Define the function $\tilde\psi_k(\xi):=\psi_k(\cdot;t_k,x_k)/\|\psi_k(\cdot;t_k,x_k)\|_{L^\infty(\R)}$.  Clearly, $\|\tilde\psi_k\|_{L^\infty(\R)}=1$, $\int_{\R}\tilde\psi_k(\xi)\dot\phi(\xi)\,d\xi=0$
and $\tilde\psi_k$ solves 
$$\mathcal{L}[\tilde\psi_k] 
	=  \tilde g_k \quad \text{in }\R,$$
with $$\tilde g_k(\xi)=\frac{g(\xi;t,x)}{\|\psi_k(\cdot;t_k,x_k)\|_{L^\infty(\R)}}.$$ Notice that by \eqref{psithmexiep12bound1} and \eqref{psithmexiep12bound2}, 
$\tilde g_k\to 0$ in $C^{2,\beta}(\R) $ as $k\to\infty$. 
By \eqref{h12normpsiLinfty} the sequence of functions $\tilde \psi_k$ is uniformly bounded in $C^{3,\beta}(\R)$.
Therefore, up to a subsequence, it  converges in $C^{3}(\R)$ to a function $\tilde\psi_\infty\in  H^\frac12(\R)\cap C^{3,\beta}(\R) $ which is solution to 
$$\mathcal{L}[\tilde\psi_\infty] =0 \quad \text{in }\R.$$
Moreover,  $\int_{\R}\tilde\psi_\infty(\xi)\dot\phi(\xi)\,d\xi=0$. Indeed, by the uniform convergence of $\tilde\psi_k$ to $\tilde\psi_\infty$ and using that 
 $\phi(-\infty)=0$ and $\phi(\infty)=1$, as $k\to\infty$,
\begin{equation*}\begin{split}
\left|\int_{\R}\tilde\psi_\infty(\xi)\dot\phi(\xi)\,d\xi\right|&=
\left|\int_{\R}\tilde\psi_k(\xi)\dot\phi(\xi)\,d\xi-\int_{\R}\tilde\psi_\infty(\xi)\dot\phi(\xi)\,d\xi\right|\\&\leq \|\tilde\psi_k-\psi_\infty\|_{L^\infty(\R)}\int_{\R}\dot\phi(\xi)\,d\xi
\\&=\|\tilde\psi_k-\psi_\infty\|_{L^\infty(\R)}\to0.\end{split}\end{equation*}
The uniqueness of the solution guaranteed by Lemma \ref{psiexistencemonneau}  implies that $\tilde\psi_\infty\equiv 0$ which is in contradiction with $\|\tilde\psi_\infty\|_{L^\infty(\R)}=1$. 
This proves \eqref{psiLinfinityasymptoticsfar} for $\psi$. 
The estimates for $\dot\psi$ and $\ddot\psi$ in \eqref{psiLinfinityasymptoticsfar}  then follow from \eqref{psithmexiep12bound1},   \eqref{h12normpsiLinfty} and
\eqref{psiLinfinityasymptoticsfar} for $\psi$, just proven. 

Next, we establish \eqref{psiLinfinityasymptoticsfar} for the time/space derivatives of $\psi$.  From \eqref{gcond}\, it follows that, for $i,j=1,\ldots,n$, 
$$\int_\R \partial_t g(\xi;t,x)\dot\phi(\xi)\,d\xi=0,\quad  \int_\R \partial_{x_i} g(\xi;t,x)\dot\phi(\xi)\,d\xi=0,\quad\int_\R \partial^2_{x_ix_j} g(\xi;t,x)\dot\phi(\xi)\,d\xi=0.$$
Moreover, from \eqref{eq:asymptotics for phi dot}, \eqref{aderivativesLinfinityestimate},  \eqref{aL2estimatefarxtder} and  \eqref{abarestimatextder}, the functions $\partial_t g$, $\partial_{x_i} g$ and  $ \partial^2_{x_ix_j}g $ belong to the space 
$H^\frac12_{\xi}(\R)\cap C_\xi^{2,\beta}(\R)$ uniformly in $(t,x)$. 
Thus, it is easy to see that $\partial_t \psi$,  $ \partial_{x_i}\psi$ and $ \partial^2_{x_ix_j}\psi$ exist and are  the unique solutions in $ H^\frac12_{\xi}(\R)$ such that 
$\int_\R \partial_t \psi(\xi;t,x)\dot\phi(\xi)\,d\xi=0,$ $  \int_\R \partial_{x_i} \psi (\xi;t,x)\dot\phi(\xi)\,d\xi=0$ and $\int_\R \partial^2_{x_ix_j} \psi(\xi;t,x)\dot\phi(\xi)\,d\xi=0$, to 
$$\mathcal{L}[\partial_t \psi] =\partial_t g,\quad \mathcal{L}[ \partial_{x_i}\psi] = \partial^2_{x_i}g,\quad \mathcal{L}[ \partial^2_{x_ix_j}\psi] = \partial^2_{x_ix_j}g\quad\text{in }\R,$$
respectively. 
Therefore,  as above, from \eqref{aderivativesLinfinityestimate},  the estimates  in \eqref{psiLinfinityasymptoticsfar} for $\partial_t \psi$, $\nabla_x \psi$, $\nabla_x \dot\psi$ and  $D^2_x\psi$ follow.

Now, we prove estimate \eqref{psiasymptoticsfar}. 
By \eqref{eq:asymptotics for phi}, \eqref{eq:asymptotics for phi dot}, \eqref{aL2estimatefar} and \eqref{abarestimate}, for $|\xi|\geq 1$, 
\begin{equation}\label{gL2estimate}|g(\xi;t,x)|\leq \frac{C}{\ep|\ln\ep||\xi|}.
\end{equation}
For $a>0$ let us denote $\phi_a(\xi):=\phi\left(\frac{\xi}{a}\right)$. Then, $\phi_a$ solves
$$C_n\mathcal{I}_1[\phi_a]=\frac{1}{a}W'(\phi_a)\quad\text{in }\R.$$
Therefore,  recalling \eqref{eq:c0-gamma} and that $W'(0)=0$, for $\xi\leq -1$, by \eqref{eq:asymptotics for phi} (note that $H(\xi)=0$),
 \begin{align*}\mathcal{L}[\phi_a](\xi)&=W''(\phi(\xi))\phi_a(\xi)-\frac{1}{a}W'(\phi_a(\xi))\\&
 =W''(0)\phi_a(\xi)-\frac{W''(0)}{a}\phi_a(\xi)+O(\phi\phi_a)+O(\phi_a^2)\\&
 =\alpha C_n\left(\phi_a(\xi)-\frac{1}{a}\phi_a(\xi)\right)+O\left(\frac{1}{\xi^2}\right)\\&
 =-C_n\frac{a-1}{\xi}+O\left(\frac{1}{\xi^2}\right).\end{align*}
Choose $a=3$. Then, there exists $R_0>0$ such that, 
 \begin{align}\label{Lphiaestimate}\mathcal{L}[\phi_a](\xi)&=-\frac{2C_n}{\xi}+O\left(\frac{1}{\xi^2}\right)=\frac{2C_n}{|\xi|}+O\left(\frac{1}{\xi^2}\right)\geq \frac{C_n}{|\xi|},\quad\text{for }\xi<-R_0.
 \end{align}
Choose $R_0$ so  large  that for $\xi<-R_0$, by  \eqref{eq:asymptotics for phi},
$$W''(\phi(\xi))\geq  W''(0)-C\phi(\xi)\geq  \alpha C_n-\frac{C}{|\xi|}\geq\frac{\alpha C_n}{2}>0.$$
Then, the operator $\mathcal{L}$ satisfies the maximum principle in $(-\infty,-R_0)$. By \eqref{psiLinfinityasymptoticsfar} and the monotonicity of $\phi$, 
 \begin{align}\label{Lphiaestimateboundary}\psi(\xi;t,x)\leq \frac{C\ep^\frac12}{\ep|\ln\ep|}\leq \frac{\tilde C}{\ep|\ln\ep|}\phi\(-\frac{R_0}{3}\)\leq  \frac{\tilde C\phi_a(\xi)}{\ep|\ln\ep|},\quad\text{for }\xi\geq -R_0.
  \end{align}
Choose $K\geq \tilde C$, with $\tilde C$ as in \eqref{Lphiaestimateboundary},  such that, if we denote $\hat\phi(\xi)=\frac{K\phi_a(\xi)}{\ep|\ln\ep|}$, then by \eqref{gL2estimate} and \eqref{Lphiaestimate},
$$\mathcal{L}[\hat\phi](\xi)\geq g(\xi;t,x)\quad\text{for }\xi<-R_0.$$
Since in addition by \eqref{Lphiaestimateboundary},
$$\psi(\xi;t,x)\leq \hat\phi(\xi)\quad\text{for }\xi\geq -R_0,$$  the maximum principle implies that
$$\psi(\xi;t,x)\leq\hat\phi(\xi)\leq \frac{C}{\ep|\ln\ep||\xi|}\quad \text{for } \xi<-R_0.$$ 
Comparing $\psi$ with $-\hat\phi$ we also get 
$$\psi(\xi;t,x)\geq -\frac{C}{\ep|\ln\ep||\xi|}\quad \text{for } \xi<-R_0.
$$ 
Similarly, choosing $a<0$ and using the periodicity of $W$, one can prove that 
$$|\psi(\xi;t,x)|\leq \frac{C}{\ep|\ln\ep||\xi|}\quad \text{for } \xi> R_0.$$ 
Estimate \eqref{psiasymptoticsfar} then follows. 

Next, we prove \eqref{psiLinfinityasymptoticnewclosefront}.
For $(t,x) \in Q_\rho = \{|d(t,x)|<\rho\}$, recall that $|\nabla d(t,x)| = 1$.  
By \eqref{eq:asymptotics for phi}, \eqref{eq:asymptotics for phi dot}, and \eqref{eq:aeestnonzerograd}, we have
\[
|g(\xi;t,x)| \leq C \quad \hbox{for all}~ \xi \in \R~\hbox{and}~(t,x) \in Q_\rho. 
\]
By the above arguments for the proof of \eqref{psiLinfinityasymptoticsfar}, we can show \eqref{psiLinfinityasymptoticnewclosefront} for $\psi$. The estimates for $\dot{\psi}$ and $\ddot{\psi}$ then follow. 
With \eqref{eq:asymptotics for phi}, \eqref{eq:asymptotics for phi dot}, and \eqref{aderivativesclosetothefront}, we find \eqref{psiLinfinityasymptoticnewclosefront} for $\partial_t\psi, \nabla_x \psi, \nabla_x \dot{\psi}$ in the same way as above for \eqref{psiLinfinityasymptoticsfar}.  

Lastly, we check \eqref{psiasymptoticsveryfarnew}. 
By \eqref{eq:asymptotics for phi}, \eqref{eq:asymptotics for phi dot}, \eqref{aeclosetofront-xi}, and \eqref{eq:aeestnonzerograd}, for $|\xi|\geq 1$ and $(t,x) \in Q_\rho$, 
\[
|g(\xi;t,x)|\leq \frac{C\gamma}{\ep|\ln\ep||\xi|}.
\]
Then, following the above arguments for the proof of  \eqref{psiasymptoticsfar}, we obtain \eqref{psiasymptoticsveryfarnew}.
\end{proof}

We conclude this section by stating the following estimate for the $n$- and $1$-dimensional fractional Laplacians of $\psi$. The proof is in Section \ref{lem:ae psi estimatesec}. 

\begin{lem}\label{lem:ae psi estimate}
There is   $C= C(\rho)>0$ such that 
\[
\abs{\ep \mathcal{I}_n\left[ \psi \( \frac{d(t,\cdot)}{\ep};t,\cdot\) \right](x) 
	- C_n \mathcal{I}_1[\psi\(\cdot;t,x\)]\(\frac{d(t,x)}{\ep}\)  } \leq  \frac{C}{ |\ln\ep|^{\frac{1}{2}}}
\]
for any $(t,x) \in  [t_0,t_0+h]  \times \R^n$. 
\end{lem}

\section{Heuristics}\label{sec:heuristics}

Here, we give two formal computations relating to Theorem \ref{thm:main} and its proof. 
We use the notation $\simeq$ to denote equality up to adding terms that vanish as $\ep \to 0$.

\subsection{Derivation of the mean curvature equation}\label{sec:heuristicspart1}

We believe it is helpful to view the heuristical derivation of the evolution of the fronts $\Gamma_t^i$ by mean curvature in Theorem \ref{thm:main}. For simplicity, we present the case $N=2$. 

For the following formal computations, assume that the signed distance function $d_i(t,x)$ associated to $\Gamma_t^i$ is smooth and that $\abs{\nabla d_i} = 1$. 
Moreover, we assume that there is a positive, uniform distance $\rho$ between $\Gamma_t^1$ and $\Gamma_t^2$. 

Consider the following  ansatz for the solution of \eqref{eq:pde}-\eqref{eq:initial cond}
\begin{equation} \label{eq:ansatz1}
u^{\ep}(t,x) \simeq \phi\( \frac{d_1(t,x)}{\ep}\)+ \phi\( \frac{d_2(t,x)}{\ep}\),
\end{equation}
with $\phi$ the solution of \eqref{eq:standing wave}. 
Plugging the ansatz into \eqref{eq:pde}, the left-hand side gives
\begin{equation}\label{eq:antatz1-left}
\ep\partial_t u^{\ep}
	\simeq  \dot{\phi}\(\frac{d_1}{\ep}\) \partial_td_1
		+\dot{\phi}\(\frac{d_2}{\ep}\)\partial_td_2.
\end{equation}
On the other hand, 
we use the equation for $\phi$  and estimates on $a_{\ep}$ (see Lemma \ref{lem:ae near front}) to write 
the fractional Laplacian of the ansatz as
\begin{equation}\label{eq:antatz1-right}
\begin{aligned}
\ep \mathcal{I}_n [u^{\ep}]  
	&\simeq \ep \mathcal{I}_n\bigg[\phi\(\frac{d_1}{\ep}\)\bigg] + \ep \mathcal{I}_n\bigg[\phi\(\frac{d_2}{\ep}\)\bigg]\\
	&= \(\ep \mathcal{I}_n\bigg[\phi\(\frac{d_1}{\ep}\)\bigg]  - C_n\mathcal{I}_1[\phi]\(\frac{d_1}{\ep}\)  \) + \(\ep \mathcal{I}_n\bigg[\phi\(\frac{d_2}{\ep}\)\bigg]   - C_n\mathcal{I}_1[\phi]\(\frac{d_2}{\ep}\)\)\\
	&\quad   +C_n\mathcal{I}_1[\phi]\(\frac{d_1}{\ep}\)  +   C_n\mathcal{I}_1[\phi]\(\frac{d_2}{\ep}\) \\
	&\simeq a_{\ep}\(\frac{d_1}{\ep}\) + a_{\ep}\(\frac{d_2}{\ep}\)
	   +W'\(\phi\( \frac{d_1}{\ep}\)\) + W'\(\phi\( \frac{d_2}{\ep}\)\).
\end{aligned}
\end{equation}
Freeze a point $(t,x)$ near the front $\Gamma_t^1$. 
Let $\xi = d_1(t,x)/\ep$. Since $d_1$ grows linearly away from $\Gamma_t^1$, we can assume, at least formally, separation of scales. That is, assume that $\xi$ and $(t,x)$ are unrelated. In this regard, let $\eta := d_2(t,x)$ with $|\eta| \geq \rho$.
 Since the ansatz $u^{\ep}$ is a solution to \eqref{eq:pde}, we can multiply the equation
 by $\dot{\phi}(\xi)$ and integrate over $\xi \in \R$ to write
\begin{equation}\label{eq:freeze}
\int_{\R} \ep \partial_t u^{\ep} \dot{\phi}(\xi) \, d \xi \simeq \frac{1}{\ep \abs{\ln \ep} } \int_{\R} \( \ep\mathcal{I}_n[ u^{\ep}]  - W'(u^\ep) \) \dot{\phi}(\xi) \, d \xi. 
\end{equation}
For convenience, we will consider the left and right-hand sides separately again.  

First, the left-hand side of \eqref{eq:freeze} with \eqref{eq:antatz1-left} gives
\begin{align*}
\int_{\R} \ep \partial_t u^{\ep}\, \dot{\phi}(\xi) \, d \xi
	&\simeq \partial_td_1(t,x) \int_{\R} [\dot{\phi}\(\xi\)]^2 \, d \xi+ \partial_td_2(t,x) \int_{\R} \dot{\phi}\(\frac{\eta}{\ep}\) \dot{\phi}(\xi) \, d \xi\\
	&\simeq c_0^{-1}\partial_td_1(t,x) +\partial_td_2(t,x) O\(\frac{ \ep^2}{\rho^2}\) \int_{\R}\dot{\phi}(\xi) \, d \xi\\
	&\simeq c_0^{-1}\partial_td_1(t,x),
\end{align*}
where we used \eqref{eq:c0-gamma}, \eqref{eq:standing wave}, and the asymptotics on $\dot{\phi}$ in \eqref{eq:asymptotics for phi dot}.

Then, we look at the right-hand side of \eqref{eq:freeze} with \eqref{eq:antatz1-right} and write
\begin{align*}
\frac{1}{\ep |\ln \ep| }& \int_{\R}  \left[\ep\mathcal{I}_n [u^{\ep}] - W'(u^{\ep})\right] \dot{\phi}(\xi) \, d \xi\\
	&\simeq \frac{1}{\ep |\ln \ep| } \int_{\R} \left[ a_{\ep}\(\xi\) + a_{\ep}\(\frac{\eta}{\ep}\)
	   +W'\(\phi\( \xi\)\) + W'\(\phi\(\frac{\eta}{\ep}\)\) \right]\dot{\phi}(\xi) \, d \xi\\
	  &\quad - \frac{1}{\ep |\ln \ep| } \int_{\R} W'\(\phi\( \xi\)+\phi\(\frac{\eta}{\ep}\)\) \dot{\phi}(\xi) \, d \xi.
\end{align*}
Using that $\phi(-\infty)=0$, $\phi(\infty)=1$  and that $W$ is periodic, we have
\begin{align*}
\frac{1}{\ep \abs{\ln \ep}} \int_{\R} W'\(\phi(\xi)\) \dot{\phi}(\xi) \, d \xi
	&= \frac{1}{\ep \abs{\ln \ep}}  \int_{\R} \frac{d}{d\xi}[W\(\phi\(\xi\)\)]\, d \xi
	= 0. 
\end{align*}
Next, we use again \eqref{eq:standing wave},  the asymptotics on $\phi$ in \eqref{eq:asymptotics for phi} (recall the definition of $\alpha$ in \eqref{eq:c0-gamma}), and that $W'(0)=0$ to  expand $W'$ around the origin and estimate
\begin{equation}\label{heuristic-potential}
\begin{split}
\frac{1}{\ep \abs{\ln \ep}} \int_{\R}W'\(\phi\(\frac{\eta}{\ep}\)\) \dot{\phi}(\xi) \, d \xi
	&= \frac{1}{\ep \abs{\ln \ep}}\int_{\R} W'\(\phi\(\frac{\eta}{\ep}\) - H\(\frac{\eta}{\ep}\)\) \dot{\phi}(\xi) \, d \xi\\
	&\simeq \frac{1}{\ep \abs{\ln\ep}} \int_{\R}\left[W'(0) + W''(0) \(\phi\(\frac{\eta}{\ep}\) - H\(\frac{\eta}{\ep}\)\)\right]\dot{\phi}(\xi) \, d \xi\\
	&\simeq   -\frac{1}{\ep \abs{\ln\ep}} \frac{\ep C_n}{\eta}\int_{\R} \dot{\phi}(\xi) \, d \xi\\
	&=-\frac{C_n}{ \abs{\ln\ep}\eta}= O\(\abs{\ln\ep}^{-1}\)\simeq 0.
\end{split}
\end{equation}
For the remaining $W'$ term, we  expand around $\phi(\xi)$ and use similar estimates to obtain
\begin{align*}
\frac{1}{\ep \abs{\ln \ep}} & \int_{\R} W'\(\phi\(\xi\)+\phi\(\frac{\eta}{\ep}\)\)\dot{\phi}(\xi) \, d \xi\\
 &=\frac{1}{\ep \abs{\ln \ep}} \int_{\R} W'\(\phi\(\xi\)+\phi\(\frac{\eta}{\ep}\)-  H\(\frac{\eta}{\ep}\)\)\dot{\phi}(\xi) \, d \xi\\
 &\simeq \frac{1}{\ep \abs{\ln \ep}} \int_{\R} 
 	\left[W'\(\phi\(\xi\)\)+W''\(\phi\(\xi\)\)\(\phi\(\frac{\eta}{\ep}\) - H\(\frac{\eta}{\ep}\)\)\right]\dot{\phi}(\xi) \, d \xi\\
&\simeq\frac{1}{\ep \abs{\ln \ep}} \int_{\R} 
 	W'\(\phi\(\xi\)\)\dot{\phi}(\xi) \, d \xi
		+ \frac{1}{\ep \abs{\ln \ep}} O\(\frac{\ep}{\rho}\)\int_{\R} W''\(\phi\(\xi\)\)\dot{\phi}(\xi) \, d \xi\\
&=0.
\end{align*}
Lastly, for the nonlocal terms, by \eqref{aL2estimatefar}, 
\begin{align*}
\frac{1}{\ep \abs{\ln \ep}} \int_{\R}  a_{\ep}\(\frac{\eta}{\ep}\) \dot{\phi}(\xi) \, d \xi
	& \simeq \frac{1}{\ep \abs{\ln \ep}} O\(\frac{\ep}{\rho}\)\int_{\R}  \dot{\phi}(\xi) \, d \xi\simeq 0
\end{align*}
and by Theorem \ref{lem:4},  
\begin{align*}
\frac{1}{\ep \abs{\ln \ep}} \int_{\R} a_{\ep}\(\xi \) \dot{\phi}(\xi) \, d \xi
	&= \bar{a}_{\ep}(t,x)\\
	&\simeq c_0^{-1} \mu \trace\( (I - \widehat{\nabla d_1(t,x)}\otimes\widehat{ \nabla d_1(t,x)}) D^2d_1(t,x)\).
\end{align*}
Combing all these pieces, \eqref{eq:freeze} for the ansatz gives 
\[
c_0^{-1} \partial_td_1(t,x)\simeq \mu c_0^{-1}\trace\( (I - \widehat{\nabla d_1(t,x)}\otimes\widehat{\nabla d_1(t,x)}) D^2d_1(t,x)\).
\]
The computation for $(t,x)$  near $\Gamma_t^2$ is similar. We conclude that the fronts move according to their mean curvature:
\[
\begin{cases}
\partial_td_1(t,x) \simeq \mu \trace\( (I - \widehat{\nabla d_1(t,x)}\otimes\widehat{\nabla d_1(t,x)}) D^2d_1(t,x)\) & \hbox{near}~\Gamma_t^1\\
\partial_td_2(t,x) \simeq \mu \trace\( (I - \widehat{\nabla d_2(t,x)}\otimes\widehat{\nabla d_2(t,x)}) D^2d_2(t,x)\) & \hbox{near}~\Gamma_t^2.
\end{cases}
\]

\begin{rem} \label{Potential-remark}
Computations \eqref{heuristic-potential} show that, for the case $N\geq 2$, the  velocity $v_i=-\partial_t d_i$ of  the front $\Gamma_t^i$ is also  affected by 
 the following interaction potential between the surfaces  
\begin{equation*}
\frac{c_0 C_n }{\abs{\ln\ep}}\sum_{j\neq i}\frac{1}{d_j}.
\end{equation*}
Since this is  a lower order term with respect to the mean curvature term in the equation for $d_i$, and  fronts remain at a positive distance  from each other,  the potential disappears when $\ep\to0$. 
\end{rem}

\subsection{ Derivation of \eqref{eq:linearized wave}}\label{Ansatzsection}

 It is actually  necessary to add a lower order correction to  \eqref{eq:ansatz1} for the ansatz to solve  the fractional Allen--Cahn equation \eqref{eq:pde}. This was already observed in the one-dimensional case in  \cite[Section 3.1]{GonzalezMonneau}. 
In order to showcase the equation for the corrector, for $\sigma\in\R$,  let $v^\ep$ be the solution to 
\begin{equation}\label{eq:pde-sub}
\ep \partial_t v^{\ep} = \frac{1}{\ep \abs{\ln \ep} }  (\ep\mathcal{I}_n v^{\ep}  - W'(v^\ep)) - \sigma,
\end{equation}
with initial condition \eqref{eq:initial cond}. 
Consider the simplest case in which $N=1$  and assume that $d(t,x) = d_1(t,x)$ is smooth with $\abs{\nabla d}=1$ and satisfies
\begin{equation}\label{eq:MC for d}
 \partial_td
 	= \mu \Delta d- c_0 \sigma
 \simeq  c_0\bar{a}_{\ep}(t,x) - c_0 \sigma,
\end{equation}
 in a neighborhood of $\Gamma_t^1$. 
Consider the new ansatz 
\[
v^{\ep}(t,x) \simeq \phi\( \frac{d(t,x)}{\ep}\) + \ep \abs{\ln\ep} \psi\( \frac{d(t,x)}{\ep}\) - \ep\abs{ \ln \ep} \tilde{\sigma},
\]
where the function $\psi$ and the constant  $\tilde{\sigma}$ are to  be determined.  
Assume that $\psi$ is smooth and bounded with bounded derivatives.  This in particular implies that the estimate in Lemma \ref{lem:ae psi estimate} holds true.

Plugging the ansatz into \eqref{eq:pde-sub}, the left-hand side gives
\begin{equation}\label{eq:ansats2-left}
\begin{aligned}
\ep \partial_t v^{\ep}
	&\simeq \dot{\phi}\( \frac{d}{\ep}\) \partial_td
		+  \ep \abs{\ln\ep}   \dot{\psi}\( \frac{d}{\ep}\)  \partial_td
	\simeq  \dot{\phi}\( \frac{d}{\ep}\) \partial_td,
\end{aligned}
\end{equation}
where we used that $\dot{\psi}$ and $\partial_td$ are bounded.
Next, we look at the right-hand side of \eqref{eq:pde-sub} for the ansatz. 
First, we use the equation for $\phi$ (see \eqref{eq:standing wave}) and  Lemmas \ref{lem:ae near front} and \ref{lem:ae psi estimate}  to find that
\begin{equation}\label{eq:ansats2-right1}
\begin{aligned}
\frac{\ep}{\ep \abs{\ln \ep}} \mathcal{I}_n[v^\ep]
	&\simeq\frac{\ep}{\ep \abs{\ln \ep}} \mathcal{I}_n\left[ \phi\( \frac{d}{\ep}\) \right]
		+ \ep \mathcal{I}_n\left[ \psi\( \frac{d}{\ep}\) \right]\\
	&=\frac{1}{\ep \abs{\ln \ep}}\(\ep \mathcal{I}_n\left[ \phi\( \frac{d}{\ep}\) \right]- C_n \mathcal{I}_1[\phi]\( \frac{d}{\ep}\) \)
	+\frac{1}{\ep \abs{\ln \ep}}C_n \mathcal{I}_1[\phi]\( \frac{d}{\ep}\)\\
	&\quad +  \(\ep \mathcal{I}_n\left[ \psi\( \frac{d}{\ep}\) \right]-C_n \mathcal{I}_1[\psi]\( \frac{d}{\ep}\) \)
	+C_n \mathcal{I}_1[\psi]\( \frac{d}{\ep}\) \\
	&\simeq \frac{1}{\ep \abs{\ln \ep}} a_{\ep}\(\frac{d}{\ep}\)
	 +  \frac{1}{\ep \abs{\ln \ep}}W'\(\phi\( \frac{d}{\ep}\)\)+C_n \mathcal{I}_1[\psi]\( \frac{d}{\ep}\).
\end{aligned}
\end{equation}
On the other hand, we do a Taylor expansion for $W'$ around $\phi(d/\ep)$ to estimate
\begin{equation}\label{eq:ansats2-right2}
\begin{aligned}
\frac{1}{\ep \abs{\ln \ep}}W'(v^{\ep})
	&\simeq \frac{1}{\ep \abs{\ln \ep}} \left[W'\(\phi\( \frac{d}{\ep}\)\) + W''\(\phi\( \frac{d}{\ep}\)\) \(v^{\ep} - \phi\( \frac{d}{\ep}\)\) \right]\\
	&\simeq \frac{1}{\ep \abs{\ln \ep}} \left[W'\(\phi\( \frac{d}{\ep}\)\)+ W''\(\phi\( \frac{d}{\ep}\)\) \(\ep \abs{\ln\ep}   \psi\( \frac{d}{\ep}\) - \ep \abs{\ln \ep} \tilde{\sigma}\)\right].
\end{aligned}
\end{equation}
Equating \eqref{eq:ansats2-left} with \eqref{eq:ansats2-right1} and \eqref{eq:ansats2-right2}, the equation for the ansatz gives
\begin{equation}\label{eq:first corrector eqn}
\begin{aligned}
 \dot{\phi}\( \frac{d}{\ep}\) \partial_td
 	&\simeq \frac{1}{\ep \abs{\ln \ep}} a_{\ep}\(\frac{d}{\ep}\) +C_n \mathcal{I}_1[\psi]\( \frac{d}{\ep}\)\\
	&\quad- W''\(\phi\( \frac{d}{\ep}\)\)   \psi\( \frac{d}{\ep}\)+\tilde{\sigma} W''\( \phi \(\frac{d}{\ep}\)\) -\sigma.
\end{aligned}
\end{equation}
Rearranging and using \eqref{eq:MC for d},  we  see that $\psi$ satisfies
\begin{align*}
-C_n \mathcal{I}_1[\psi]&\( \frac{d}{\ep}\)
+W''\(\phi\( \frac{d}{\ep}\)\)  \psi\( \frac{d}{\ep}\) \\
 &\simeq \frac{1}{\ep \abs{\ln \ep}} a_{\ep}\(\frac{d}{\ep}\)- \dot{\phi}\( \frac{d}{\ep}\) \partial_td
 	+ \tilde{\sigma} W''\( \phi \(\frac{d}{\ep}\)\)- \sigma\\
&\simeq  \frac{1}{\ep \abs{\ln \ep}} a_{\ep}\(\frac{d}{\ep}\)
		- \dot{\phi}\( \frac{d}{\ep}\)c_0 \bar{a}_{\ep} 
		+ c_0 \sigma  \dot{\phi}\(\frac{d}{\ep}\)+ \tilde{\sigma} W''\( \phi \(\frac{d}{\ep}\)\) - \sigma.
\end{align*}
Evaluating the equation when $|d(t,x)|>> \ep$, from \eqref{eq:asymptotics for phi}, \eqref{eq:asymptotics for phi dot},  \eqref{aL2estimatefar} and \eqref{abarestimate}, if $\psi(\pm\infty)=0$, we see that $\tilde\sigma$ must satisfy
$$ \tilde{\sigma} W''\( 0\) =\sigma.$$
This shows that the corrector $\psi=\psi(\xi;t,x)$ is solution to  \eqref{eq:linearized wave}. 

Note that  for equation \eqref{eq:linearized wave} to be solvable, the right hand-side $g=g(\xi;t,x)$ must satisfy the compatibility condition \eqref{gcond} guaranteed by Lemma \ref{ghortogonaltophidot}. 
Indeed,  multiply both sides of \eqref{eq:linearized wave} by $\dot{\phi}(\xi)$ and integrate over $\R$ to write
\begin{align}\label{ghortogonalphi'0}
\int_{\R} \mathcal{L}[\psi](\xi) \dot{\phi}(\xi)\, d\xi
	=\int_{\R} g(\xi;t,x) \dot{\phi}(\xi)\,d\xi. \end{align}
Since $\mathcal{I}_1$ is self-adjoint and $\phi$ satisfies \eqref{eq:standing wave}, the left-hand side of the equation gives
\begin{equation}\label{ghortogonalphi'}\begin{split}
\int_{\R} \mathcal{L}[\psi] \dot{\phi} \, d\xi
	&= \int_{\R}\(- C_n\mathcal{I}_1[\dot{\phi}]+ W''\(\phi\) \dot{\phi}\)\psi \, d\xi \\
	&= \int_{\R}\frac{d}{d\xi}\(- C_n\mathcal{I}_1[\phi]+ W'\(\phi\)\)\psi \, d\xi 
	=0. 
\end{split}
\end{equation} 

Finally, note that in view of \eqref{ghortogonalphi'0} and the above computations, the addition of the corrector $\psi$ in the ansatz for the solution of 
\eqref{eq:pde}-\eqref{eq:initial cond} does not give any additional contributions in the derivation of the mean curvature equation obtained  in Section \ref{sec:heuristicspart1}. 

\begin{rem} 
Notice that $\psi$ depends on the signed distance function $d(t,x)$ through $a_\ep$ and $\bar{a}_\ep$.
Hence, when $N>1$, we have a finite sequence of correctors, denoted by $\psi_1,\dots, \psi_N$, depending on the signed distance function $d_i(t,x)$ to the front $\Gamma_t^i$, $i=1,\dots,N$. 
\end{rem}

\section{Construction of barriers}\label{sec:barriers}

We now construct  local and global  strict subsolutions (supersolutions) to \eqref{eq:pde} needed for the proof of Theorem \ref{thm:main}. 
We will focus on the construction of subsolutions as the construction of supersolutions is analogous. 
We will start with the global ones.

 \subsection{Global subsolutions}
 
Fix   $t_0 \in [0,\infty)$ and   $h>0$. For $i=1,\ldots,N$  and $t\in[t_0,t_0+h]$, let  $\Omega^i_t$ be  bounded open sets with boundaries $\Gamma_t^i:= \partial \Omega_t^i$ such that 
\begin{equation}\label{eq:ordered phi sets}
\Omega^{i+1}_t\subset \subset  \Omega^i_t \quad \hbox{for}~i=1,\dots, N-1.
\end{equation}
Let $\tilde{d}_i(t,x)$ be the signed distance function associated to the set  $ \Omega^i_t $, then  $\Gamma_t^i =  \{ x \in \R^n : \tilde{d}_i(t,x) = 0\}$.
Assume  that there exists  a $\rho>0$ such that, for all $1 \leq i \leq N$, $\tilde{d}_i(t,x)$ is smooth in the set
\begin{equation}\label{eq:Q2rho}
Q_{2\rho}^i = \{ (t,x) \in[t_0,t_0+h]\times\R^n: |\tilde{d}_i(t,x)| \leq 2\rho\},
\end{equation}
and let $d_i$ be the smooth, bounded extension of $\tilde{d}_i$ outside of $Q_{\rho}^i$ as defined in Definition \ref{defn:extension}.
Assume in addition that there exists $\sigma>0$ such that
\begin{equation}\label{eq:mc for d}
\partial_td_i \leq 
  \mu\Delta  d_i - c_0 \sigma \quad \hbox{in}~Q_\rho^i.
\end{equation}
Let $\tilde{\sigma}>0$ be such that $\sigma = W''(0) \tilde{\sigma}$.
By \eqref{eq:ordered phi sets}, and perhaps making $\sigma $
smaller, we can assume that, for all $t\in[t_0,t_0+h]$, 
 \begin{equation}\label{Qisigmaemptyimters} \{x\in \R^n\,:\,|d_i(t,x)|<\tilde\sigma\}\cap \{x \in \R^n \,:\,|d_j(t,x)|<\tilde\sigma\}=\emptyset\quad i \not=j.\end{equation}
We define the smooth barrier $v^{\ep}(t,x)$ by
\begin{equation}\label{eq:barrier defn}
v^{\ep}(t,x)
	=\sum_{i=1}^N \phi\( \frac{d_i(t,x)- \tilde{\sigma} }{\ep}\) +\ep \abs{\ln \ep} \sum_{i=1}^N\psi_i\( \frac{d_i(t,x)-\tilde{\sigma}}{\ep};t,x\) - \tilde{\sigma} \ep \abs{\ln \ep}.
\end{equation}

Recall the parameter $\gamma \in (0,1)$ in \eqref{aepsilondef}.

\begin{lem}[Global subsolutions to \eqref{eq:pde}]  \label{lem:barrier}
Let $\rho$ be  
as in Definition \ref{defn:extension}. 
Assume \eqref{eq:mc for d}  with $\mu$ and $\sigma$ as in \eqref{eq:velocity-intro} and \eqref{Qisigmaemptyimters}, and  let $v^{\ep}$ be defined as in \eqref{eq:barrier defn}. Then there exists $\ep_0=\ep_0(\rho,\gamma,\sigma)>0$ such that for any $0<\ep<\ep_0$, 
$v^{\ep}$ is a solution to
\begin{equation}\label{eq:pde sub}
\ep \partial_t v^{\ep} - \frac{1}{\ep \abs{\ln \ep}} \( \ep \mathcal{I}_n[v^{\ep} ] - W'(v^{\ep} )\) \leq  -\frac{\sigma}{2}
\quad \hbox{in}~[t_0,t_0+h] \times \R^n.
\end{equation}
Moreover, 
there is a constant $C>0$ such that, for $\ep<\ep_0$,
\begin{equation}\label{eq:barrier estimate-N}
v^{\ep}(t,x) \geq N - 2 \tilde{\sigma} \ep |\ln \ep | 
\quad \hbox{in}~\left\{ (t,x)\in [t_0,t_0+h]\times\R^n :d_N(t,x) - \tilde{\sigma} \geq \frac{C}{\tilde{\sigma}|\ln \ep|}\right\}.
\end{equation}
\end{lem}

\begin{proof}
We will break the proof into four main steps. First, we estimate the equation for $v^{\ep}(t,x)$ for any $(t,x)$. Then, we will show that $v^{\ep}(t,x)$ satisfies \eqref{eq:pde sub} when $(t,x)$ is near a single front $\Gamma_t^{i_0}$ and then when $(t,x)$ is far from all fronts $\Gamma_t^{i}$, $i=1,\dots, N$. 
Lastly, we prove the estimate in \eqref{eq:barrier estimate-N}. 
For convenience, we use the following notation throughout the proof:
\begin{equation}\label{eq:notation}
\begin{aligned}
\phi_i &:= \phi \( \frac{d_i(t,x)- \tilde{\sigma} }{\ep}\)\\
 \psi_i &:= \psi_i \( \frac{d_i(t,x)- \tilde{\sigma} }{\ep};t,x\)\\
 \tilde{\phi}_i &:= \phi \( \frac{d_i(t,x)- \tilde{\sigma} }{\ep}\)- H \( \frac{d_i(t,x)- \tilde{\sigma} }{\ep}\)\\
a_{\ep}^i 
	&: = a_{\ep}\(\frac{d_i(t,x)- \tilde{\sigma} }{\ep};t,x\)\\
\bar{a}_{\ep}^i 
	&:= \bar{a}_{\ep}(t,x) \quad \hbox{corresponding to}~a_{\ep}^i.
\end{aligned}
\end{equation}
We note that it will be important for the reader to remember the dependence of $\psi_i$ and $a_{\ep}$ on the variables $t,x$ and $\xi = d_i(t,x)/\ep$ when taking derivatives in $t,x$. 


\medskip

\noindent
\underline{\bf Step 1}. Computation for $v^{\ep}(t,x)$ in \eqref{eq:pde} for an arbitrary $(t,x) \in [t_0,t_0+h] \times \R^n$. 

\medskip

First, the time derivative of $v^{\ep}$ at $(t,x)$ is given by
\begin{align*}
\ep \partial_t v^{\ep}(t,x)
	&=  \sum_{i=1}^N  \dot{\phi}_i \, \partial_td_i(t,x)
	 + \ep \abs{\ln \ep} \sum_{i=1}^N \left[\dot\psi_i \partial_t d_i(t,x)
	 +\ep \partial_t\psi_i\right].
\end{align*}
By \eqref{psiLinfinityasymptoticsfar} for $\dot{\psi}_i$ and $\partial_t\psi_i$, 
\begin{equation}\label{eq:ve-time}
\ep \partial_t v^{\ep}
	= \sum_{i=1}^N\dot{\phi}_i\, \partial_td_i(t,x) + O(\ep^{\frac12}) + O(\ep|\ln \ep|) 
	=  \sum_{i=1}^N\dot{\phi}_i\, \partial_td_i(t,x) + O(\ep^{\frac12}).
\end{equation}
Next, we consider the nonlocal term. For each $i = 1,\dots, N$, we use that $\phi$ satisfies \eqref{eq:standing wave} and apply Lemma \ref{lem:ae near front} to find
\begin{align*}
\ep\mathcal{I}_n [\phi_i](x)
	&= \ep\mathcal{I}_n [\phi_i](x)
	 	- C_n\mathcal{I}_1 [ \phi ]\( \frac{d_i(t,x)- \tilde{\sigma} }{\ep}\)
			+ W'\(\phi_i\)\\
	&= a_\ep^i + O(\ep)+ W'\(\phi_i\).
\end{align*}
Also, using that $\psi$ satisfies \eqref{eq:linearized wave} and  Lemma \ref{lem:ae psi estimate}, we find that
\begin{align*}
\ep \mathcal{I}_n [ \psi_i](x)
	&=\ep \mathcal{I}_n [ \psi_i](x)
	- C_n \mathcal{I}_1 [\psi_i] \( \frac{d_i(t,x)- \tilde{\sigma} }{\ep}\)
	  - \mathcal{L}[\psi]\( \frac{d_i(t,\cdot)- \tilde{\sigma} }{\ep}\) + W''\(\phi_i\)\psi_i\\
	&=  O(|\ln\ep|^{-\frac{1}{2}})
	 - \frac{ a_{\ep}^i }{\ep \abs{\ln \ep}}
	 + \dot{\phi}_i c_0 \(\bar{a}_{\ep}^i - \sigma\)
	 - \tilde{\sigma} \(W''\(\phi_i\) - W''(0)\)
	+W''\( \phi_i\)\psi_i.
\end{align*}
Therefore, the fractional Laplacian of $v^{\ep}$ can be written as
\begin{align*}
\ep \mathcal{I}_n[v^{\ep}](x)
	&= \sum_{i=1}^N \left[
	a_\ep^i + O(\ep)
			+ W'\(\phi_i\)
			\right]\\
	&\quad + \ep \abs{\ln \ep} \sum_{i=1}^N \bigg[ O(|\ln\ep|^{-\frac{1}{2}})
	 - \frac{ a_{\ep}^i }{\ep \abs{\ln \ep}} 
	+ \dot{\phi}_i c_0 \(\bar{a}_{\ep}^i - \sigma\) \\
	&\hspace{1in}  - \tilde{\sigma} \(W''\(\phi_i\) - W''(0)\)
	+W''\( \phi_i\)\psi_i\bigg].
\end{align*}
Recall the definition of $\tilde{\phi}_i$ introduced in \eqref{eq:notation}.
Since $W$ is periodic, we have that $W'(\phi_i) = W'(\tilde{\phi}_i)$ and similarly $W''(\phi_i) = W''(\tilde{\phi}_i)$.
Using this and rearranging, 
we equivalently write 
\begin{equation}\label{eq:ve-space}
\begin{aligned}
\ep \mathcal{I}_n[v^{\ep}](x)
	&= \sum_{i=1}^N
		W'(\tilde{\phi}_i)  + O(\ep) + O(\ep | \ln \ep|^{\frac{1}{2}})\\
	&\quad 
		+ \ep \abs{\ln \ep} \sum_{i=1}^N \left[
		W''( \tilde{\phi}_i)\psi_i
		+ \dot{\phi}_i c_0 \(\bar{a}_{\ep}^i - \sigma\)  
		- \tilde{\sigma} \(W''(\tilde{\phi}_i) - W''(0)\)\right].
\end{aligned}
\end{equation}
Then, with \eqref{eq:ve-time} and \eqref{eq:ve-space}, the 
left-hand side of \eqref{eq:pde sub}
at $(t,x)$ can be written as 
\begin{align*}
\mathcal{J}[v^\ep]
&:=\ep \partial_t v^{\ep}(t,x) - \frac{1}{\ep \abs{\ln \ep}} \( \ep \mathcal{I}_n[v^{\ep}(t, \cdot)](x) - W'(v^{\ep}(t,x) )\)\\
	&= 
	O(\ep^{\frac12}) + \sum_{i=1}^N\dot{\phi}_i \partial_td_i(t,x)\\
	&\quad - \frac{1}{\ep \abs{\ln \ep}} \Bigg\{
	 \sum_{i=1}^N
		W'(\tilde{\phi}_i)  + O(\ep) + O(\ep | \ln \ep|^{\frac{1}{2}})\\
	&\quad + \ep \abs{\ln \ep} \sum_{i=1}^N\left[ 
		W''( \tilde{\phi}_i)\psi_i
		+ \dot{\phi}_i c_0 \(\bar{a}_{\ep}^i - \sigma\)  
		- \tilde{\sigma} \(W''(\tilde{\phi}_i) - W''(0)\)\right]\\
		&\quad- W'\(\sum_{i=1}^N \tilde{\phi}_i + \ep \abs{\ln \ep}  \sum_{i=1}^N \psi_i -  \tilde{\sigma} \ep \abs{\ln \ep}\) 
		\Bigg\}.
\end{align*}
Grouping the error terms, the $\dot{\phi}_i$ terms, and the nonlinear terms together, we have
\begin{equation}\label{Eqn1}
\begin{aligned}
\mathcal{J}[v^\ep]
	&= 
	O(|\ln\ep|^{-\frac{1}{2}})
	+ \sum_{i=1}^N\dot{\phi}_i \left[\partial_td_i(t,x) - c_0 \(\bar{a}_{\ep}^i - \sigma\) \right]\\
	&\quad + \frac{1}{\ep \abs{\ln \ep}} \Bigg\{
	W'\(\sum_{i=1}^N \tilde{\phi}_i + \ep \abs{\ln \ep}  \sum_{i=1}^N \psi_i -  \tilde{\sigma} \ep \abs{\ln \ep}\)\\
	&\hspace{.75in}-\sum_{i=1}^N\( W'(\tilde{\phi}_i)+  \ep \abs{\ln \ep} 
		\left[W''( \tilde{\phi}_i)\psi_i
		- \tilde{\sigma} \(W''(\tilde{\phi}_i) - W''(0)\) \right]\)
		\Bigg\}.
\end{aligned}
\end{equation}
Fix an index $i_0 \in \{1,\dots,N\}$. For the remainder of Step 1, we will conveniently isolate every term indexed with $i_0$ to help with Step 2. 
First, we do a Taylor expansion for $W'$ around $\tilde{\phi}_{i_0}$ to obtain
\begin{equation}\label{eq:Wprimeforbarrier}
\begin{aligned}
W'\bigg(\sum_{i=1}^N \tilde{\phi}_i &+ \ep \abs{\ln \ep}  \sum_{i=1}^N \psi_i -  \tilde{\sigma} \ep \abs{\ln \ep}\bigg)\\
	&= W'(\tilde{\phi}_{i_0}) + W''(\tilde{\phi}_{i_0})\(\sum_{i\not=i_0} \tilde{\phi}_i + \ep \abs{\ln \ep}  \sum_{i=1}^N\psi_i -  \tilde{\sigma} \ep \abs{\ln \ep}\) \\
	&\quad+ O\(\(\sum_{i\not=i_0} \tilde{\phi}_i + \ep \abs{\ln \ep}  \sum_{i=1}^N\psi_i - \tilde{\sigma} \ep \abs{\ln \ep}\)^2\).
\end{aligned}
\end{equation}
By \eqref{psiLinfinityasymptoticsfar} for $\psi_i$, we have that
\begin{align*}
 \frac{1}{\ep \abs{\ln \ep}} O&\(\(\sum_{i\not=i_0} \tilde{\phi}_i + \ep \abs{\ln \ep}  \sum_{i=1}^N\psi_i -  \tilde{\sigma} \ep \abs{\ln \ep}\)^2\)\\
	&\hspace{.5in}=   \sum_{i\not= i_0} O\(\frac{(\tilde{\phi}_i)^2}{\ep \abs{\ln \ep}}\)
		 + O(\abs{\ln \ep}^{-1})+  O(\ep \abs{\ln \ep}).
\end{align*}
Hence, with \eqref{eq:Wprimeforbarrier} and grouping error terms, we can write \eqref{Eqn1} as
\begin{align*}
\mathcal{J}[v^\ep]
	&= 
	 O(\abs{\ln \ep}^{-\frac{1}{2}})
	+\sum_{i\not= i_0} O\(\frac{(\tilde{\phi}_i)^2}{\ep \abs{\ln \ep}}\)
	+ \sum_{i=1}^N\dot{\phi}_i \left[\partial_td_i(t,x) - c_0 \(\bar{a}_{\ep}^i - \sigma\) \right]\\
	&\quad+ \frac{1}{\ep \abs{\ln \ep}} \Bigg\{
	W'(\tilde{\phi}_{i_0}) + W''(\tilde{\phi}_{i_0})\(\sum_{i\not=i_0} \tilde{\phi}_i + \ep \abs{\ln \ep}  \sum_{i=1}^N\psi_i -  \tilde{\sigma} \ep \abs{\ln \ep}\)\\
	&\quad-\( W'(\tilde{\phi}_{i_0})+  \ep \abs{\ln \ep} 
		\left[W''( \tilde{\phi}_{i_0})\psi_{i_0}
		- \tilde{\sigma} \(W''(\tilde{\phi}_{i_0}) - W''(0)\) \right]\)\\
	&\quad
		-\sum_{i\not= i_0}\( W'(\tilde{\phi}_i)+  \ep \abs{\ln \ep} 
		\left[W''( \tilde{\phi}_i)\psi_i
		- \tilde{\sigma} \(W''(\tilde{\phi}_i) - W''(0)\) \right]\)
		\Bigg\}
\end{align*}
where in the last two lines we extracted the $i_0$ term. 
Cancelling the $W'(\tilde{\phi}_{i_0})$ and $W''(\tilde{\phi}_{i_0})\psi_{i_0}$ terms then distributing $1/(\ep \abs{\ln \ep})$, we simplify to
\begin{align*}
\mathcal{J}[v^\ep]
	&= 
	 O(\abs{\ln \ep}^{-\frac{1}{2}})
	+\sum_{i\not= i_0} O\(\frac{(\tilde{\phi}_i)^2}{\ep \abs{\ln \ep}}\)
	+ \sum_{i=1}^N\dot{\phi}_i \left[\partial_td_i(t,x) - c_0 \(\bar{a}_{\ep}^i - \sigma\) \right]\\
	&\quad + 
	 W''(\tilde{\phi}_{i_0})\(\sum_{i\not=i_0} \frac{\tilde{\phi}_i}{ \ep \abs{\ln \ep}} +  \sum_{i\not= i_0}\psi_i -  \tilde{\sigma} \)
	+\tilde{\sigma} \(W''(\tilde{\phi}_{i_0}) - W''(0)\)\\
	&\quad
		-\sum_{i\not= i_0}\left[ \frac{W'(\tilde{\phi}_i)}{ \ep \abs{\ln \ep}}+ 
		W''( \tilde{\phi}_i)\psi_i
		- \tilde{\sigma} \(W''(\tilde{\phi}_i) - W''(0)\) \right].
\end{align*}
Next, we do a Taylor expansion for $W'$ around $0$ and recall that $W'(0) = 0$ to write
\begin{align*}
W'(\tilde{\phi}_i)
	&= W'(0) + W''(0) \tilde{\phi}_i + O((\tilde{\phi}_i)^2) = W''(0) \tilde{\phi}_i + O((\tilde{\phi}_i)^2).
\end{align*}
With this, we now have that
\begin{align*}
\mathcal{J}[v^\ep]
	&= 
	O(\abs{\ln \ep}^{-\frac{1}{2}})
	+\sum_{i\not= i_0} O\(\frac{(\tilde{\phi}_i)^2}{\ep \abs{\ln \ep}}\)
	+ \sum_{i=1}^N\dot{\phi}_i \left[\partial_td_i(t,x) - c_0 \(\bar{a}_{\ep}^i - \sigma\) \right]\\
	&\quad + 
	 W''(\tilde{\phi}_{i_0})\(\sum_{i\not=i_0} \frac{\tilde{\phi}_i}{ \ep \abs{\ln \ep}} +  \sum_{i\not= i_0}\psi_i -  \tilde{\sigma} \)
	+\tilde{\sigma} \(W''(\tilde{\phi}_{i_0}) - W''(0)\)\\
	&\quad
		-\sum_{i\not= i_0}\left[ \frac{W''(0) \tilde{\phi}_i}{ \ep \abs{\ln \ep}}+ 
		W''( \tilde{\phi}_i)\psi_i
		- \tilde{\sigma} \(W''(\tilde{\phi}_i) - W''(0)\) \right].
\end{align*}
We rearrange to group the terms with $W''(\tilde{\phi}_{i_0}) - W''(0)$ together to obtain
\begin{equation}\label{Eqn2}
\begin{aligned}
\mathcal{J}[v^\ep]
	&= 
	O(\abs{\ln \ep}^{-\frac{1}{2}})
	+\sum_{i\not= i_0} O\(\frac{(\tilde{\phi}_i)^2}{\ep \abs{\ln \ep}}\)
	+ \sum_{i=1}^N\dot{\phi}_i \left[\partial_td_i(t,x) - c_0 \(\bar{a}_{\ep}^i - \sigma\) \right]\\
	&\quad + 
	\( W''(\tilde{\phi}_{i_0}) - W''(0)\)\sum_{i\not=i_0} \frac{\tilde{\phi}_i}{ \ep \abs{\ln \ep}} 
	-\tilde{\sigma}W''(0)\\
	&\quad
		+\sum_{i\not= i_0}\left[ \(W''(\tilde{\phi}_{i_0})- 
		W''( \tilde{\phi}_i)\)\psi_i
		+ \tilde{\sigma} \(W''(\tilde{\phi}_i) - W''(0)\) \right].
\end{aligned}
\end{equation}
Looking at the last two lines for $i \not= i_0$, 
note first that $(W''(\tilde{\phi}_{i_0})- W''(0)) \tilde{\phi}_i = O(\tilde{\phi}_i)$ and then
 Taylor expand $W''$ around $0$ to find
\begin{align*}
(W''(\tilde{\phi}_{i_0})-W''( \tilde{\phi}_i))\psi_i
				+\tilde{\sigma} \(W''(\tilde{\phi}_i) - W''(0)\)
			&= O(\psi_i) - \tilde{\sigma}\( W'''(0) \tilde{\phi}_i + O(\tilde{\phi}_i^2) \) \\
		&= O(\psi_i) + O(\tilde{\phi}_i). 
\end{align*}
Therefore, \eqref{Eqn2} can be written as
\begin{equation}\label{eq:end of step 1}
\begin{aligned}
\mathcal{J}[v^\ep]
	&= 
	O(\abs{\ln \ep}^{-\frac{1}{2}})
	+ \sum_{i\not= i_0} \left[	
		 O\left(\frac{\tilde{\phi}_{i}}{\ep \abs{\ln \ep}}\right)
		+O(\psi_i) \right]\\
	&\quad
	+\sum_{i=1}^N\dot{\phi}_i \left[\partial_td_i(t,x) - c_0 \(\bar{a}_{\ep}^i - \sigma\) \right]
  - \sigma.
\end{aligned}
\end{equation}
where we used that $\sigma = W''(0) \tilde{\sigma}$. 


\medskip

\noindent
\underline{\bf Step 2}. $v^{\ep}(t,x)$ satisfies \eqref{eq:pde sub} when $(t,x)$ is near the front $\Gamma_t^{i_0}$.

\medskip

Assume that $\abs{d_{i_0}(t,x) - \tilde\sigma} \leq \abs{\ln \ep}^{-1/2}$ for some index $1 \leq i_0 \leq N$.
Then, by \eqref{Qisigmaemptyimters}, for $\ep$ sufficiently small, 
\[
\abs{d_{i}(t,x) -  \tilde\sigma} \geq \abs{\ln \ep}^{-\frac12} \quad \hbox{for all}~i \not= i_0.
\]
We begin by estimating the error terms in \eqref{eq:end of step 1} for $i \not= i_0$. 
First, we use \eqref{eq:asymptotics for phi} to estimate
\begin{align*}
|\tilde{\phi}_{i}|
	&= \abs{\phi \(\frac{d_{i}(t,x) - \tilde \sigma}{\ep}\) - H\(\frac{d_{i}(t,x) -  \tilde\sigma}{\ep}\)}  
	\leq C  \frac{\ep}{\abs{d_{i}(t,x) -  \tilde\sigma}} 
	=  O(\ep\abs{\ln \ep}^{\frac12}),
\end{align*}
from which it follows that
\begin{equation*}\label{eq:errors step 1}
\begin{aligned}
 \frac{|\tilde{\phi}_i|}{\ep \abs{\ln \ep}}
 	=
	O\left(\abs{\ln \ep}^{-\frac12}\right).
\end{aligned}
\end{equation*}
Similarly, from \eqref{psiasymptoticsfar}, for $i\neq i_0$, 
\begin{equation*}
 \abs{\psi_i }\leq \frac{C}{\ep|\ln\ep|} \frac{\ep}{\abs{d_{i}(t,x) - \tilde \sigma}} = O\left(\abs{\ln \ep}^{-\frac12}\right).
\end{equation*}
Combining the above estimates in view of \eqref{eq:end of step 1}, we have
\[
\sum_{i\not= i_0} \left(O\(\frac{\tilde{\phi}_i}{\ep \abs{\ln \ep}}\) +O(\psi_i)
\right)
	= O\left(\abs{\ln \ep}^{-\frac12}\right).
\]

Next, we check the terms with the auxiliary functions $\bar{a}_{\ep}^i$ and $a_{\ep}^i$. 
For $i\not= i_0$, we use that $d_i$ is smooth, \eqref{eq:asymptotics for phi dot} and  \eqref{abarestimate} to obtain 
\begin{equation}\label{eq:mc-estimate-far}
\abs{ \sum_{i\not= i_0}\dot{\phi}_i \left[\partial_td_i(t,x) - c_0 \(\bar{a}_{\ep}^i - \sigma\) \right]}\leq 
  \sum_{i\not= i_0} \left(\frac{\ep}{d_{i}(t,x) - \tilde \sigma}\right)^2\frac{C}{\ep^\frac12|\ln\ep|}
	=  O\left(\ep^{\frac32}\right). 
\end{equation}
For $i= i_0$, we use that $\dot{\phi}_{i_0} \geq 0$, \eqref{eq:mc for d}, and Theorem \ref{lem:4} to estimate
\begin{align*}
\dot{\phi}_{i_0} [ \partial_td_{i_0}(t,x)- c_0( \bar{a}_{\ep}^{i_0}
		-  \sigma)]
	&=\dot{\phi}_{i_0}\([ \partial_td_{i_0}(t,x)- \mu \Delta d_{i_0}(t,x) + c_0 \sigma]+ [\mu \Delta d_{i_0}(t,x) - c_0 \bar{a}_{\ep}^{i_0}]\)\\
	&\leq\dot{\phi}_{i_0}(0+ o_\ep(1)) =  o_\ep(1).
\end{align*}
Consequently, in \eqref{eq:end of step 1}, we have that
\begin{align*}
\mathcal{J}[v^\ep]
	\leq   o_\ep(1)
	-\sigma.
\end{align*}
Taking $\ep$ sufficiently small, \eqref{eq:pde sub} holds. 


\medskip

\noindent
\underline{\bf Step 3}. $v^{\ep}(t,x)$ satisfies \eqref{eq:pde sub} when $(t,x)$ is away from all  fronts $\Gamma_t^{i}$.

\medskip

Assume that for all $i=1,\dots, N$,
 \[
\abs{d_{i}(t,x) -  \tilde\sigma} \geq \abs{\ln \ep}^{-\frac12}.
\]
Then, we estimate exactly as in Step 2 but we include $i= i_0$ in  \eqref{eq:mc-estimate-far}
to obtain 
\begin{align*}
\mathcal{J}[v^\ep]
	\leq	  o_\ep(1)-\sigma.
\end{align*}
Taking $\ep$ sufficiently small, \eqref{eq:pde sub} holds. 


\medskip

\noindent
\underline{\bf Step 4}. $v^{\ep}(t,x)$ satisfies \eqref{eq:barrier estimate-N}. 

\medskip

Let $\tilde{\rho} := \frac{C}{\tilde{\sigma}|\ln \ep|}$ for $C>0$ to be determined and $\ep>0$ sufficiently small. 
Fix $(t,x)$ such that $d_N(t,x) - \tilde{\sigma} \geq \tilde{\rho}$. Then,
\begin{equation*}\label{d_i-tildesigmaineq}
d_i(t,x) - \tilde{\sigma} \geq \tilde{\rho} \quad \hbox{for all}~1 \leq i \leq N,
\end{equation*}
and by \eqref{eq:asymptotics for phi} and \eqref{psiasymptoticsfar},
\begin{align*}
v^{\ep}(t,x)
	&\geq N  + \sum_{i=1}^N \left( \phi \left( \frac{d_i(t,x) - \tilde{\sigma}}{\ep}\right) - H \left( \frac{d_i(t,x) - \tilde{\sigma}}{\ep}\right) \right)
		- C \sum_{i=1}^N\frac{\ep}{d_i(t,x) - \tilde{\sigma}} - \tilde{\sigma}\ep |\ln \ep| \\
	&\geq N -  \sum_{i=1}^N \left( \frac{\ep}{\alpha(d_i(t,x) - \tilde{\sigma})} + \frac{C\ep^2}{(d_i(t,x) - \tilde{\sigma})^2} \right)
		- C \sum_{i=1}^N\frac{\ep}{d_i(t,x) - \tilde{\sigma}} - \tilde{\sigma}\ep |\ln \ep| \\
	&\geq N - \frac{C\varepsilon}{\rho} - \frac{C\varepsilon^2}{\rho^2} - \tilde{\sigma} \varepsilon|\ln \varepsilon| \\
	&\geq N - 2 \tilde{\sigma} \varepsilon|\ln \varepsilon| 
\end{align*}
for $C$ sufficiently large in the choice of $\tilde{\rho}$. This proves \eqref{eq:barrier estimate-N}.
\end{proof}

 \subsection{Local subsolutions}

We now construct local subsolutions to \eqref{eq:pde}. For $i=1,\ldots,N$  and $t\in[t_0,t_0+h]$, let  $\Omega^i_t$ be  bounded open sets
satisfying \eqref{eq:ordered phi sets}. Assume that there exists $R,\,\rho$ and $x_0\in\R^n$  such that the signed distance function $\tilde{d}_i(t,x)$ associated to the set  $ \Omega^i_t $ is smooth in the set 
 $Q_{2\rho}^i\cap ([t_0,t_0+h]\times B(x_0,R))$ (recall \eqref{eq:Q2rho}).
   Let $d_i$ be the  bounded extension of $\tilde{d}_i$ outside of $Q_{\rho}^i$  as defined in Definition \ref{defn:extension},
then $d_i$ is smooth in  $[t_0,t_0+h]\times B(x_0,R)$. Assume that 
there exist $\sigma>0$ and  $0<r<R$ such that,
\begin{equation}\label{eq:mc for d-local}
\partial_td_i \leq \mu\Delta  d_i - c_0 \sigma \quad \hbox{in}~Q_\rho\cap ([t_0,t_0+h]\times  B(x_0,r)).
\end{equation}
As before, let $\tilde{\sigma}>0$ be such that $\sigma = W''(0) \tilde{\sigma}$, then, by eventually making $\sigma $
smaller, we can assume that, for all $t\in[t_0,t_0+h]$, \eqref{Qisigmaemptyimters}  holds true. 

The following lemma is the local version of Lemma \ref{lem:barrier}. 

\begin{lem}[Local subsolutions to \eqref{eq:pde}]  \label{lem:barrier-barrier}
For $0 <r <R$, 
let $\rho$ as in Definition \ref{defn:extension} and $\gamma$ as in \eqref{aepsilondef} be such that  $0<\rho,\,\gamma<R-r$.
Assume that $d_i$ is smooth in  $[t_0,t_0+h]\times B(x_0,R)$ and that \eqref{eq:mc for d-local} with $\mu$ as in \eqref{eq:velocity-intro}  and \eqref{Qisigmaemptyimters} hold. 
Let $v^{\ep}$ be defined as in \eqref{eq:barrier defn}.
 Then there exists $\ep_0=\ep_0(\rho,\gamma,\sigma)>0$ such that for any $0<\ep<\ep_0$, 
$v^{\ep}$ is a solution to
\[
\ep \partial_t v^{\ep} - \frac{1}{\ep \abs{\ln \ep}} \( \ep \mathcal{I}_n[v^{\ep} ] - W'(v^{\ep} )\) \leq  -\frac{\sigma}{2}
\quad \hbox{in}~[t_0,t_0+h] \times B(x_0,r). 
\]
\end{lem}
\begin{proof}
The proof of the lemma is exactly as the proof of Lemma \ref{lem:barrier} except that we now only consider points $(t,x) \in [t_0,t_0+h]\times B(x_0,r)$. 
Note that if  $0 < \gamma< R-r$ and $x\in  B(x_0,r)$, then $d(t,x+\ep(\cdot))$ is smooth at all points in the domain of integration of $a_\ep$ and all estimates in 
Lemma \ref{lem:ae bound} and Theorem \ref{lem:psi-reg} hold true. 
In addition, the smoothness of $d(t,\cdot)$ in $B(x_0,R)$ and its boundedness in the whole $\R^n$ are sufficient to guarantee
 that for $0 < \rho< R-r$ and $x\in  B(x_0,r)$, Lemmas  \ref{lem:ae near front} and  \ref{lem:ae psi estimate} hold true, see  Sections \ref{sec:estimatesbis} and
  \ref{lem:ae psi estimatesec}. 
\end{proof}

\section{Proof of Theorem \ref{thm:main}}\label{sec:proof of thm}

\begin{proof}
We apply an adaptation of the abstract method in \cite{Barles-Souganidis,Barles-DaLio} as described in Section \ref{sec:flows}.

Begin by defining the families of open sets $(D^i)_{i=1}^N$ and $(E^i)_{i=1}^N$ by 
\begin{align*}
D^i &= \operatorname{Int}\left\{ (t,x) \in (0,\infty) \times\R^n : \liminf_{\ep \to 0}{_*} \frac{u^{\ep}(t,x) - i}{\ep \abs{\ln \ep}} \geq 0 \right\} \subset (0,\infty)\times\R^n\\[.5em]
E^i &= \operatorname{Int}\left\{ (t,x) \in (0,\infty) \times\R^n : \limsup_{\ep \to 0}{^*}\frac{u^{\ep}(t,x)  - (i-1)}{\ep \abs{\ln \ep}} \leq 0 \right\} \subset (0,\infty)\times\R^n.
\end{align*}
To define the traces of $D^i$ and $E^i$, we first define the functions $\underline{\chi}^i, \overline{\chi}^i:(0,\infty) \times \R^n \to \{-1,1\}$, respectively, by
\[
\underline{\chi}^i = \one_{D^i} - \one_{(D^i)^c} \quad \hbox{and} \quad 
\overline{\chi}^i = \one_{(E^i)^c} - \one_{E^i}.
\]
Since $D^i$ is open, $\underline{\chi}^i$ is lower semicontinuous, and since $(E_i)^c$ is closed, $\overline{\chi}^i$ is upper semicontinuous. To ensure that $\overline{\chi}^i$ and $\underline{\chi}^i$ remain lower and upper semicontinuous, respectively, at $t=0$, we set
\[
\underline{\chi}^i(0,x) = \liminf_{t\to 0,~y\to x} \underline{\chi}^i(t,y) \quad \hbox{and} \quad
\overline{\chi}^i(0,x) = \limsup_{t\to 0,~y\to x} \overline{\chi}^i(t,y).
\]
Define the traces $D_0^i$ and $E^i_0$ by
\begin{align*}
D_0^i = \{x \in \R^n : \underline{\chi}^i(0,x) = 1 \} \quad \hbox{and} \quad
E_0^i = \{x \in \R^n : \overline{\chi}^i(0,x) = -1 \}.
\end{align*}
Note that $D_0^i $ and $E_0^i $ are open sets. 
For $t > 0$, define the sets $D_t^i$ and $E_t^i$ by
\[
D_t^i = \{x \in \R^n : (t,x) \in D^i\}  
\quad \hbox{and} \quad
E_t^i = \{x \in \R^n : (t,x) \in E^i\}. 
\]
We need the following propositions for the abstract method. Their proofs are delayed until the end of the section. 

\begin{prop}[Initialization]\label{prop:initialization}
For each $i=1,\dots, N$,
\[
\Omega_0^i \subset D_0^i
\quad \hbox{and} \quad
(\overline{\Omega}_0^i)^c \subset E_0^i.
\]
\end{prop}

\begin{prop}[Propagation]\label{prop:sub/sup flows}
For each $i=1,\dots, N$,
$(D_t^i)_{t>0}$ is a generalized super-flow, 
and  $((E^i_t)^c)_{t>0}$ is a generalized sub-flow, according to Definition \ref{defn:flow}.
\end{prop}

By \cite[Theorem 1.1, Theorem 1.2]{Barles-DaLio},
it follows from Propositions \ref{prop:initialization} and \ref{prop:sub/sup flows} that
\[
{^+}\Omega_t^i \subset D_t^i \subset {^+}\Omega_t^i \cup \Gamma_t^i
\quad \hbox{and} \quad
{^-}\Omega_t^i \subset E_t^i \subset {^-}\Omega_t^i \cup \Gamma_t^i.
\]
The conclusion readily follows; we write the details for completeness.

First, since ${^+}\Omega_t^i \subset D_t^i$, we use the definition of $D_t^i$ to see that
\begin{equation}\label{eq:liminf from below}
\liminf_{\ep \to 0}{_*} u^{\ep}(t,x) \geq i \quad \hbox{for}~x \in {^+}\Omega_t^i.
\end{equation}
Using that $ {^-}\Omega_t^{i+1} \subset E_t^{i+1}$, we similarly get
\begin{equation} \label{eq:limsup from above}
\limsup_{\ep \to 0}{^*} u^{\ep}(t,x) \leq (i+1)-1 = i \quad \hbox{for}~x \in {^-}\Omega_t^{i+1}.
\end{equation}
Therefore, for $i=1,\dots, N-1$, 
\[
\lim_{\ep \to 0} u^{\ep}(t,x) = i \quad \hbox{in}~ {^+}\Omega_t^i \cap {^-}\Omega_t^{i+1}.
\]
Now, since constant functions solve  \eqref{eq:pde} and $0 \leq u_0^{\ep} \leq N$, by the comparison principle, we know $0 \leq u^{\ep} \leq N$. In particular,
\[
0 \leq \liminf_{\ep \to 0}{_*} u^{\ep} \quad \hbox{and} \quad 
\limsup_{\ep \to 0}{^*} u^{\ep} \leq N.
\]
Together with \eqref{eq:limsup from above} and respectively \eqref{eq:liminf from below} we have
\[
\lim_{\ep \to 0} u^{\ep}(t,x) = 0 \quad \hbox{in}~ {^-}\Omega_t^1
\quad \hbox{and} \quad
\lim_{\ep \to 0} u^{\ep}(t,x) = N \quad \hbox{in}~ {^+}\Omega_t^N.
\]
\end{proof}

It remains to prove Propositions \ref{prop:initialization} and \ref{prop:sub/sup flows}. 

\subsection{Proof of Proposition \ref{prop:initialization} }

\begin{proof}
We will prove that $\Omega_0^{i_0} \subset D_0^{i_0}$ for all $1 \leq i_0 \leq N$. 
The proof of $(\overline{\Omega}_0^{i_0})^c \subset E_0^{i_0}$ is similar. 

Fix $i_0$ and a point $x_0 \in \Omega_0^{i_0}$.  
To prove that $x_0 \in D_0^{i_0}$, it is enough to show that, for all $(t,x)$ in a neighborhood of $(0,x_0)$ in $[0,\infty)\times\R^n$, 
\[
\liminf_{\ep \to 0}{_*} \frac{u^{\ep}(t,x) - i_0}{\ep \abs{\ln \ep}} \geq 0.
\]
For this, we will use \eqref{eq:barrier defn} to construct a suitable (global in space) subsolution $v^{\ep} \leq u^{\ep}$, depending on a small constant $\tilde{\sigma}>0$. 

For  $0<\tilde\sigma<d_i^0(x_0)$, where $d_i^0$ is defined  in \eqref{eq:initial d_i},  let $r_i>0$ be given by
\begin{equation}\label{riballinsideomegai}
r_i = d_i^0(x_0) - \tilde{\sigma}, \quad i=1,\dots, i_0.
\end{equation}
Note that $B(x_0, r_i) \subset \subset \Omega_0^i$,  and since the sets $\Gamma_0^i$ are separated,  for $\tilde\sigma$ small enough, 
\begin{equation}\label{ri0-ri0+1}
r_i - r_{i+1} = d_i^0(x_0) - d_{i+1}^0(x_0) \geq d(\Gamma_0^i,\Gamma_0^{i+1})\geq 4\tilde \sigma.
\end{equation}
 For a constant $C>0$, to be determined and for    $t\leq r_{i_0}/(2C)\leq r_i/(2C)$,  let  $\tilde{d}_{i}(t,x)$ be the signed distance function associated to the ball
$B(x_0,{r_i-Ct})$, namely
\begin{equation*}\label{eq:di for initialization}
\tilde{d}_{i}(t,x) = r_{i}-Ct - \abs{x-x_0},\quad i=1,\ldots,i_0. 
\end{equation*}
Then, by \eqref{riballinsideomegai} and \eqref{ri0-ri0+1}, 
\begin{equation}\label{didi+1proppinitial} d_i^0(x)-\tilde\sigma \geq \tilde d_{i}(0,x)\geq \tilde d_{i+1}(0,x)+4\tilde \sigma.\end{equation}
For $0<\rho<r_{i_0}/4$, let  $d_i$ be the  smooth, bounded extension of $\tilde{d}_i$  outside of
$$Q^i_\rho= \left\{(t,x)\in \left[0,\frac{r_{i_0}}{2C}\right]\times\R^n\,:\,|\tilde d_i(t,x)|\leq\rho\right\},$$ as in Definition \ref{defn:extension}. 
Then, for $(t,x)\in Q^i_\rho$, we have that $d_i=\tilde d_i$, $|x-x_0|\geq r_{i_0}/4$ and 
\[
\partial_t d_i(t,x) =  -C,\quad D^2   d_i(t,x)=-\frac{1}{|x-x_0|}\left(I -\frac{(x-x_0)\otimes (x-x_0)}{|x-x_0|^2}\right)
\] 
from which we find that
\begin{equation}\label{varphi:ineLeminit}\begin{split}
\partial_t  d_i(t,x) - \mu\Delta  d_i (t,x)
	= -C+ \mu\frac{n - 1}{|x-x_0|}
	\leq -C + 4\mu\frac{n - 1}{r_{i_0}} 
	&\leq - c_0 \sigma
\end{split}\end{equation}
for $C>0$ sufficiently large, with $\sigma=W''(0)\tilde\sigma$ and $\mu$ as  in \eqref{eq:velocity-intro}. 
Moreover, we can assume, by eventually making $\tilde\sigma$ smaller,  that $\rho>2\tilde\sigma$. 

Let $v^{\ep} = v^{\ep}(t,x)$ be given by
\[
v^{\ep}(t,x)
	=\sum_{i=1}^{i_0} \phi\( \frac{d_i(t,x)- \tilde{\sigma} }{\ep}\) +\ep \abs{\ln \ep} \sum_{i=1}^{i_0}\psi_i\( \frac{d_i(t,x)-\tilde{\sigma}}{\ep};t,x\) - \tilde{\sigma} \ep \abs{\ln \ep}.
\]
Note that $v^\ep$ is the same as \eqref{eq:barrier defn} except we only sum over $1 \leq i \leq i_0$. 
By \eqref{didi+1proppinitial}  and the fact that the balls $ B(x_0,{r_i-Ct})$ shrink with the same constant velocity, condition \eqref{Qisigmaemptyimters} is satisfied in the time interval $[0, r_{i_0}/(2C)]$. Since in addition  \eqref{varphi:ineLeminit} holds true, we are in position to apply  Lemma \ref{lem:barrier} (with $N=i_0$) to conclude that  $v^{\ep}$ is a solution to \eqref{eq:pde sub} in $[0, r_{i_0}/(2C)] \times \R^n$ for $\ep$ small enough.
We claim that choosing $\gamma$ in \eqref{aepsilondef} sufficiently small (but independent of $\ep$), and  $\ep$ sufficiently small, 
\begin{equation}\label{Prop7.1initialcond}
v^{\ep}(0,x)\leq u^{\ep}(0,x)=u_0^\ep(x)\quad\text{for all }x\in\R^n,
\end{equation}
with $u_0^\ep$ as in \eqref{eq:initial cond}. We break into two cases:  we first consider the case when $x$ is near the boundary of one of the balls, say $\partial B(x_0,r_j)$ for some $j$,  and then when $x$ is far from the boundaries of all balls $\partial B(x_0,r_i)$, $i=1,\ldots, i_0$. 

\medskip

\noindent
{\bf Case 1}:~\emph{There is a $ j \in \{1,\ldots, i_0\}$ such that  $|\tilde d_{j}(0,x)-\tilde\sigma|\leq \frac{C_0}{|\ln\ep|}$ for $C_0>0$ to be determined.}

For $\ep$ small enough, $\frac{C_0}{|\ln\ep|}+\tilde\sigma\leq 2\tilde\sigma<\rho$, so that $0\leq \tilde d_{j}(0,x)\leq2\tilde\sigma<\rho$. Hence, recalling the Definition \ref{defn:extension} of $d_i$ and by the second inequality in \eqref{didi+1proppinitial}, we have that 
$$d_j(0,x)=\tilde d_j(0,x),$$ 
$$d_i(0,x)-\tilde\sigma\geq \tilde\sigma\quad\text{for }i=1,\ldots, j-1,$$
$$d_i(0,x)-\tilde\sigma\leq -\tilde\sigma\quad\text{for }i=j+1,\ldots, i_0.$$
Therefore, by  \eqref{eq:asymptotics for phi} and 
\eqref{psiasymptoticsveryfarnew},  
$$\phi\( \frac{d_i(0,x)- \tilde{\sigma} }{\ep}\)+\ep \abs{\ln \ep} \psi_i\( \frac{d_i(0,x)-\tilde{\sigma}}{\ep};0,x\) \leq 
\begin{cases} 
\displaystyle 1+\frac{C\ep}{\tilde\sigma} & \text{for }i=1,\ldots,j-1\\
\displaystyle \frac{C\ep}{\tilde\sigma}& \text{for }i=j+1,\ldots,i_0,
\end{cases}
$$
and by the monotonicity of $\phi$, \eqref{eq:asymptotics for phi} and \eqref{psiLinfinityasymptoticnewclosefront}, 
\begin{align*}\phi\( \frac{d_j(0,x)- \tilde{\sigma} }{\ep}\)+\ep \abs{\ln \ep} \psi_{j}\( \frac{d_{j}(0,x)-\tilde{\sigma}}{\ep};0,x\)&\leq  \phi\( \frac{C_0}{\ep|\ln\ep|}\)+C\ep \abs{\ln \ep}
\\& \leq 1-\frac{C\ep|\ln\ep|}{C_0}+
C\ep \abs{\ln \ep}.\end{align*}
From the last two estimates, we deduce that, for $\ep$ small enough, 
$$v^{\ep}(0,x)\leq j-\frac{C\ep|\ln\ep|}{C_0}+
C\ep \abs{\ln \ep}-\tilde{\sigma} \ep \abs{\ln \ep}\leq j -C\ep \abs{\ln \ep}, $$
if $C_0$ is sufficiently small. 

Similarly, by the first inequality in \eqref{didi+1proppinitial}, \eqref{eq:asymptotics for phi}, and using that $\phi>0$,
$$\phi\( \frac{d_i^0(x)}{\ep}\)\geq 
\begin{cases}
1-\frac{C\ep}{\tilde\sigma} & \text{for }i=1,\ldots,j-1\\
0 & \text{for }i=j+1,\dots i_0,
\end{cases}
$$ and 
since $d_j^0(x)\geq \tilde d_j(0,x)+\tilde\sigma\geq\tilde \sigma$, 
$$\phi\( \frac{d_j^0(x)}{\ep}\)\geq 1-\frac{C\ep}{\tilde\sigma}.$$
Putting it all together, we infer that, for $\ep$ small enough,  
\begin{align*} v^{\ep}(0,x)\leq  j -C\ep \abs{\ln \ep}\leq j-\frac{C\ep}{\tilde\sigma}\leq \sum_{i=1}^N
\phi\( \frac{d_i^0(x)}{\ep}\)=u_0(x), 
\end{align*}
which proves  \eqref{Prop7.1initialcond} in Case 1. 

\medskip

\noindent
{\bf Case 2}: \emph{  $|\tilde d_{i}(0,x)-\tilde\sigma|\geq \frac{C_0}{|\ln\ep|}$ for all $i=1,\ldots, i_0$.}

Recalling the Definition \ref{defn:extension} of $d_i$, 
if $|\tilde d_i(0,x)|<\rho$, then 
 $ d_i(0,x)=\tilde d_i(0,x)$. 
 In this case, by the  first inequality in \eqref{didi+1proppinitial} and the monotonicity of 
$\phi$, we have that 
$$\phi\( \frac{d_i(0,x)- \tilde{\sigma} }{\ep}\)\leq \phi\( \frac{d_i^0(x)}{\ep}\),$$
and by \eqref{psiasymptoticsveryfarnew},
$$\ep \abs{\ln \ep}\left| \psi_i\( \frac{d_i(0,x)-\tilde{\sigma}}{\ep};0,x\)\right|
\leq C\gamma\ep|\ln\ep|.$$
If instead, $\tilde d_i(0,x)\geq \rho $, then $ d_i(0,x)\geq \rho$ and $d_i^0(x)\geq \tilde d_i(0,x)\geq \rho$, and by 
\eqref{eq:asymptotics for phi},
$$\phi\( \frac{d_i(0,x)- \tilde{\sigma} }{\ep}\)\leq 1\leq \phi\( \frac{d_i^0(x)}{\ep}\) +\frac{C\ep}{\rho}.$$
Finally, if $\tilde  d_i(0,x)\leq -\rho $, then  $ d_i(0,x)\leq -\rho$ and 
$$\phi\( \frac{d_i(0,x)- \tilde{\sigma} }{\ep}\)\leq \frac{C\ep}{\rho}\leq \phi\( \frac{d_i^0(x)}{\ep}\) +\frac{C\ep}{\rho}.$$
For $|\tilde  d_i(0,x)|\geq \rho $,  by  \eqref{psiasymptoticsfar} we also have
$$\ep \abs{\ln \ep}\left| \psi_i\( \frac{d_i(0,x)-\tilde{\sigma}}{\ep};0,x\)\right|\leq \frac{C\ep}{\rho}.$$
Putting it all together, we infer that, for $\ep$ small enough, 
\begin{align*} v^{\ep}(0,x)\leq u_0(x)+C\gamma\ep|\ln\ep|-\tilde\sigma\ep|\ln\ep|\leq  u_0(x),\end{align*}
choosing $\gamma\leq \tilde\sigma/C.$ 
This proves  \eqref{Prop7.1initialcond} in Case 2. 

\medskip

By \eqref{Prop7.1initialcond} and the comparison principle, 
$$u^\ep(t,x)\geq v^\ep(t,x)\quad \text{for all }(t,x)\in  \left[0, \frac{r_{i_0}}{2C}\right] \times \R^n. $$
Since $2\sigma<\rho$, note that $\{x :d_{i_0}(t,x) \geq 2\tilde{\sigma} \}= \{x :\tilde{d}_{i_0}(t,x) \geq 2\tilde{\sigma} \}$. Consequently, by  \eqref{eq:barrier estimate-N} (with $N=i_0$) and the inequality above, for $t\in [0, r_{i_0}/(2C)]$ and 
$$x\in  \{x:\tilde{d}_{i_0}(t,x) \geq 2\tilde{\sigma} \}=\{|x-x_0|\leq r_{i_0}-Ct-2\tilde \sigma\}\supset\left \{|x-x_0|\leq \frac{r_{i_0}}{2}-2\tilde\sigma\right\},$$
we have 
\begin{align*}
\liminf_{\ep \to 0}{_*} \frac{u^{\ep}(t,x) - i_0}{\ep \abs{\ln \ep}} 
&\geq \liminf_{\ep \to 0}{_*} \frac{v^{\ep}(t,x) - i_0}{\ep \abs{\ln \ep}} 
\geq -2 \tilde{\sigma}.
\end{align*} 
Letting $\tilde{\sigma} \to 0$, the result follows. 
\end{proof}

\subsection{Proof of Proposition \ref{prop:sub/sup flows} }

\begin{proof}
Fix $1 \leq i_0 \leq N$. We will show that $(D_t^{i_0})_{t>0}$ is a generalized super-flow. 
The proof that $((E_t^{i_0})^c)_{t>0}$ is a generalized sub-flow is similar. 

Let $(t_0,x_0) \in (0,\infty) \times \R^n$, $h, \, r>0$, and $\varphi_{i_0}:(0,\infty)\times \R^n \to \R$ be a smooth function satisfying \ref{item:i}-\ref{item:v} in Definition \ref{defn:flow}
in $[t_0,t_0+h]$ with $F$ given in \eqref{Fdefintro-flow}   and $\mu$ in \eqref{eq:velocity-intro}.
Then, there exists $h'>h$ such that $\varphi_{i_0}$ satisfies  \ref{item:i}-\ref{item:v} in $[t_0,t_0+h']$, with an eventually smaller $\delta$ in \ref{item:ii}. 
For $t\in [t_0,t_0+h']$, let us denote 
$$\Omega^{i_0}_t=\{x\in\R^n\,:\,\varphi_{i_0}(t,x)>0\}\quad\text{and}\quad\Gamma^{i_0}_t=\partial \Omega^{i_0}_t.$$
By  \ref{item:iii}  there exists $R'>r$ such that $\nabla \varphi_{i_0}\neq 0$ on  $\Gamma^{i_0}_t\cap  \overline{B}(x_0,R')$ which is therefore a smooth (and by \ref{item:i} a non-empty) set. 
Let $\tilde{d}_{i_0}(t,x)$ be the signed distance function associated to $\Omega^{i_0}_t$, and let
$Q_\rho^{i_0} =\{(t,x)\in \left[t_0,t_0+h'\right]\times\R^n\,:\,|\tilde{d}_{i_0}(t,x)|\leq\rho\}$ for $\rho>0$. 
Then, there exist $\rho_0>0$ and $r'>r$ such that $\tilde d_{i_0}$  is smooth 
in $Q_{\rho_0}^{i_0}\cap  ([t_0,t_0+h']\times B(x_0,R'))$
 and by \ref{item:ii} (recall \eqref{meancurvature_intro_sigma}),  
\begin{equation}\label{eq:mc for d-locali_0prop72}
\partial_t\tilde d_{i_0} \leq \mu\Delta  \tilde d_{i_0} - \frac{\delta}{2} \quad \hbox{in}~Q_{\rho_0}^{i_0}\cap  ([t_0,t_0+h']\times B(x_0,r')).
\end{equation}
Next, for $i=1,\ldots, i_0-1$, and $\tilde\sigma>0$, define 
\begin{equation*}\label{omegat:prop7.2}\Omega^{i}_t=\{x\in\R^n : \tilde{d}_{i_0}(t,x)\geq -4 (i_0-i)\tilde\sigma\}
 \quad\hbox{and} \quad \Gamma_t^i = \partial \Omega^{i}_t.
\end{equation*}
Clearly, $\Omega^{i+1}_t\subset \Omega^{i}_t$, and if $\tilde{d}_{i}(t,x)$ is the signed distance function associated to $\Omega^{i}_t$, 
 then 
\begin{equation}\label{didi+1proppflow}
 \tilde d_{i}(t,x)\geq \tilde d_{i+1}(t,x)+4\tilde \sigma \quad \hbox{for all}~i=1,\dots, i_0-1.
\end{equation} 
Moreover, by \eqref{eq:mc for d-locali_0prop72} there exist $R,\,\rho,\,\tilde\sigma>0$ 
 such that if   $d_{i}(t,x)$ is the  bounded extension of $\tilde{d}_{i}(t,x)$ outside of 
$ Q_\rho^{i} $ as in Definition \ref{defn:extension}, then  $d_i$ is smooth 
in $[t_0,t_0+h']\times B(x_0,R)$ and 
\begin{equation*}\label{eq:mc for d-locali_0prop72-i}
\partial_td_{i} \leq \mu\Delta  d_{i} - c_0 \sigma \quad \hbox{in}~Q_{\rho}^{i}\cap  ([t_0,t_0+h']\times B(x_0,r)) \quad \hbox{for all}~i=1,\dots, i_0,
\end{equation*}
where $\sigma=W''(0)\tilde\sigma$. 
Then, 
after taking $\rho$ perhaps smaller and  $\gamma$ in \eqref{aepsilondef} such that $0<\rho,\,\gamma<R-r$, 
Lemma \ref{lem:barrier-barrier} (with $N= i_0$) implies that the function
\[
v^{\ep}(t,x)
	=\sum_{i=1}^{i_0} \phi\( \frac{d_i(t,x)- \tilde{\sigma} }{\ep}\) +\ep \abs{\ln \ep} \sum_{i=1}^{i_0}\psi_i\( \frac{d_i(t,x)-\tilde{\sigma}}{\ep};t,x\) - \tilde{\sigma} \ep \abs{\ln \ep}
\]
 is a solution to \eqref{eq:pde sub} in $[t_0,t_0+h'] \times B(x_0,r)$.
We now show that for $\tilde\sigma<\rho/2$ and $\ep,\,\gamma$ small enough,
\begin{equation}\label{initialpropo2}
v^{\ep}(t_0,x)\leq u^{\ep}(t_0,x)\quad\text{for all }x\in  \R^n
\end{equation}
and
\begin{equation}\label{initialpropo2-boundary}
v^{\ep}(t,x)\leq u^{\ep}(t,x)\quad\text{for all }(t,x) \in [t_0,t_0+h'] \times( \R^n \setminus B(x_0,r)). 
\end{equation}
We start with \eqref{initialpropo2}. 
Since $\varphi_{i_0}$ satisfies \ref{item:i} and \ref{item:iv}   in Definition \ref{defn:flow}, 
 for $\tilde\sigma$ small enough we have that $\Omega_{t_0}^{i} \subset \subset D_{t_0}^{i_0}$ for all $i=1,\ldots,i_0$, 
 and there is 
a compact set $K\subset D_{t_0}^{i_0}$ such that 
\begin{equation}\label{didi+1proppflow-K}
d_{K}(x)-2\tilde\sigma \geq \tilde d_{i}(t_0,x) \quad \hbox{for all}~i = 1,\dots, i_0
\end{equation}
where $d_K$ denotes the signed distance function from $K$. 

The proof of \eqref{initialpropo2} is broken into three cases:  we first consider the case when $x$ is close to one of the sets $\Gamma_{t_0}^j$, 
then  when $x$ is in $K$ but far from all sets $\Gamma_{t_0}^i$, and finally when $x$ is not in $K$. 

\medskip

\noindent
{\bf Case 1}:~\emph{There is a $ j \in \{1,\ldots, i_0\}$ such that  $|\tilde d_{j}(t_0,x)-\tilde\sigma|\leq \frac{C_0}{|\ln\ep|}$ for $C_0>0$ to be determined.}

Like in Case 1 in the proof of Proposition  \ref{prop:initialization}, we can show, for $\ep, C_0>0$ small enough
\[
v^\ep(t_0,x) \leq j - C\ep |\ln \ep|.
\]
Since $\tilde{d}_j(t_0,x) \geq 0$  for $\ep$ small enough, and by \eqref{didi+1proppflow-K},  we know that $x \in K\subset D_{t_0}^{i_0}$. 
 Since $K$ is compact and by the definition of  $D_{t_0}^{i_0}$,  given $\delta_0>0$, for $\ep$ small enough and $y\in K$,  
\begin{equation}\label{eq:ue-delta}
\frac{u^\ep(t_0,y) - i_0}{\ep |\ln \ep|}  \geq - \delta_0.  
\end{equation}
In particular, for $\delta_0 \leq C$,
\[
u^\ep(t_0,x) \geq i_0 -\delta_0 \ep |\ln \ep| \geq j - C \ep | \ln \ep| \geq v^\ep(t_0,x). 
\]
Therefore, \eqref{initialpropo2} holds for Case 1 for sufficiently small $\ep >0$. 

\medskip

\noindent
{\bf Case 2}: \emph{$x \in K$ and  $|\tilde d_{i}(t_0,x)-\tilde\sigma|\geq \frac{C_0}{|\ln\ep|}$ for all $i=1,\ldots, i_0$.}

We note that 
\[
\phi \left(\frac{d_i(t_0,x) - \tilde{\sigma}}{\ep}\right) \leq 1 
\]
and then estimate $\psi_i$ as in Case 2 in the proof of Proposition  \ref{prop:initialization}, to find
\[
v^\ep(t_0,x) \leq i_0 + C\gamma\ep|\ln \ep| - \tilde{\sigma} \ep | \ln \ep|. 
\]
On the other hand, since $x \in K$, we know that \eqref{eq:ue-delta} holds for given $\delta_0>0$ and $\ep$ small enough. 
Choosing $\gamma \leq \tilde{\sigma}/(2C) $ and $\delta_0=\tilde{\sigma}/2$, 
\[
u^\ep(t_0,x) \geq  i_0 - \delta_0 \ep | \ln \ep|  \geq v^\ep(t_0,x) 
\]
for $\ep >0$ small. We now have \eqref{initialpropo2} in Case 2. 

\medskip

\noindent
{\bf Case 3}: \emph{$x \notin K $.}

Since $d_{K}(t_0,x) \leq 0$, by \eqref{didi+1proppflow-K}, we have that $\tilde{d}_i(t_0,x) \leq - 2\tilde{\sigma}$,  for all $i=1,\ldots, i_0$. 
If $-\rho < \tilde{d}_i(t_0,x) < - 2\tilde{\sigma}$, then $\tilde{d}_i(t_0,x) = d_i(t_0,x)$. By \eqref{eq:asymptotics for phi} and
 \eqref{psiasymptoticsfar}, 
\begin{align*}
\phi \left(\frac{d_i(t_0,x)-\tilde{\sigma}}{\ep}\right)
	&+\ep |\ln \ep| \psi \left(\frac{d_i(t_0,x)-\tilde{\sigma}}{\ep};t_0,x\right)
	- \tilde{\sigma}\ep|\ln \ep|
\leq \frac{C\ep}{\tilde{\sigma}}+\frac{C\ep}{\tilde{\sigma}}-\tilde{\sigma}\ep|\ln \ep|  \leq 0
\end{align*}
for $\ep>0$ sufficiently small. 
If $\tilde{d}_i(t_0,x)< -\rho$, then $d_i(t_0,x) < -\rho$ and we similarly find
\begin{align*}
\phi \left(\frac{d_i(t_0,x)-\tilde{\sigma}}{\ep}\right)
	&+\ep |\ln \ep| \psi \left(\frac{d_i(t_0,x)-\tilde{\sigma}}{\ep};t_0,x\right)
	- \tilde{\sigma}\ep|\ln \ep|
\leq \frac{C\ep}{\rho}+\frac{C\ep}{\rho}-\tilde{\sigma}\ep|\ln \ep|  \leq 0
\end{align*}
for small $\ep>0$. In particular, we have shown $v^\ep(t_0,x) \leq 0$. 
Now, since the zero function is a solution to \eqref{eq:pde} and $u_0^{\ep} \geq 0$, the comparison principle implies $u^\ep(t_0,x) \geq 0$. Therefore, 
\[
u^\ep(t_0,x) \geq 0
	\geq v^\ep(t_0,x),
\]
and \eqref{initialpropo2} holds for Case 3. 

\medskip
This proves \eqref{initialpropo2}. Inequality \eqref{initialpropo2-boundary} follows with a similar argument using that 
 $\varphi_{i_0}$ satisfies \ref{item:i} and \ref{item:v}  in Definition \ref{defn:flow}.
With \eqref{initialpropo2} and \eqref{initialpropo2-boundary}, the comparison principle then implies
\begin{equation}\label{eq:comp-claim}
u^{\ep} (t,x)\geq v^{\ep}(t,x) \quad \hbox{for all}~ (t,x)\in [t_0,t_0+h'] \times \R^n. 
\end{equation}
By  \eqref{eq:barrier estimate-N} (with $N=i_0$),  we have that, for all $t\in[t_0,t_0+h']$, 
\begin{align*}
 \frac{u^{\ep}(t,x) - i_0}{\ep \abs{\ln \ep}} 
 	&\geq  \frac{v^{\ep}(t,x) - i_0}{\ep \abs{\ln \ep}} 
	\geq -2\tilde{\sigma} \quad \hbox{in}~\left\{x \in \R^n: d_{i_0}(t,x) - \tilde{\sigma} \geq \frac{C}{\tilde\sigma|\ln\ep|}\right\}.
\end{align*}
Letting $\ep \to 0$, it follows that, for all $t\in[t_0,t_0+h']$, 
\begin{equation}\label{lastinclusion_prop72}
\{x \in \R^n :d_{i_0}(t,x) - \tilde{\sigma} \geq 0\} \subset \left\{ x \in \R^n : \liminf_{\ep \to 0}{_*} \frac{u^{\ep}(t,x) - i_0}{\ep \abs{\ln \ep}} \geq -2 \tilde{\sigma} \right\}.
\end{equation}
Now, let $x_1 \in \{ x \in B(x_0,r) : \varphi_{i_0}(t_0+h,x) >0\}$, so that  $d_{i_0}(t_0+h,x_1)>0$.
Then, there exist $r_1>0$ and $0<\tau<h'-h$ such that for $|t-(t_0+h)|<\tau$, it holds that $B(x_1,r_1)\subset \{x \in B(x_0,r) :d_{i_0}(t,x) > 0 \}$ 
and by \eqref{lastinclusion_prop72}, for  $\tilde\sigma<r_1/2$, 
$$[t_0+h-\tau,t_0+h+\tau]\times B\left(x_1,\frac{r_1}{2}\right)\subset\left\{ (t,x): \liminf_{\ep \to 0}{_*} \frac{u^{\ep}(t,x) - i_0}{\ep \abs{\ln \ep}} \geq -2 \tilde{\sigma} \right\}.$$
Taking $\tilde{\sigma} \to 0$, we see that $(t_0+h,x_1)$ is an interior point of the set 
$$\left\{ (t,x) \in (0,\infty) \times\R^n : \liminf_{\ep \to 0}{_*} \frac{u^{\ep}(t,x) - i_0}{\ep \abs{\ln \ep}} \geq 0 \right\} ,$$
namely, it belongs to $D^{i_0}$. This proves the desired inclusion
\[
\{x \in B(x_0,r) :\varphi_{i_0}(t_0+h,x) > 0 \} 
	= \{x \in B(x_0,r) :d_{i_0}(t_0+h,x) >0 \} \subset  D^{i_0}_{t_0+h}. 
\]
\end{proof}

\section{Proof of Lemma \ref{lem:ae bound}}
\label{sec:estimates}

First, note that by the regularity of $\phi$ and $d$, there is some $\tau\in(0,1)$ and $C>0$ such that 
\begin{equation}\label{Taylorexpansionsecondphi0}\begin{split}&\abs{ \phi\(\xi +  \frac{d(t,x+\ep z) - d(t,x)}{\ep}\) - \phi\(\xi+ \nabla d(t,x) \cdot z\)}
\\&\leq \dot \phi\(  \xi +  (1-\tau) \nabla d(t,x) \cdot z+\tau \frac{d(t,x+\ep z) - d(t,x)}{\ep}\)C\ep|z|^2.\end{split}
\end{equation}

\subsection{Proof of \eqref{aLinfinityestimate} and \eqref{abarestimate}}

To prove \eqref{aLinfinityestimate}, we will show that for all $(\xi,t,x)\in\R\times [t_0,t_0+h]\times\R^n$,
\begin{equation}\label{aLinfinityestimate0}
|a_\ep(\xi;t,x)|\leq C\ep^\frac12.
\end{equation}
The same estimate for $\dot a_\ep$,  $\ddot a_\ep$ and $\dddot a_\ep$ will follow similarly by replacing $\phi$ in expression \eqref{aepsilondef} by
$\dot\phi$, $\ddot\phi$ and  $\dddot\phi$, respectively.

Begin by writing
\begin{align*}
a_{\ep}(\xi;t,x)
	&=\int_{\abs{z}<\ep^{-\frac12}} \( \phi\(\xi +  \frac{d(t,x+\ep z) - d(t,x)}{\ep}\) - \phi\( \xi + \nabla d(t,x) \cdot z\) \) \frac{dz}{\abs{z}^{n+1}}\\
	&\quad + \int_{\ep^{-\frac12}< |z|<\gamma \ep^{-1}} \( \phi\(\xi +  \frac{d(t,x+\ep z) - d(t,x)}{\ep}\) - \phi\( \xi + \nabla d(t,x) \cdot z\) \) \frac{dz}{\abs{z}^{n+1}}\\
	&=: I + II.
\end{align*}
For the long-range interactions,  
\begin{align*}
\abs{II}
	&\leq 2  \int_{\abs{z}>\ep^{-\frac12}}  \frac{dz}{\abs{z}^{n+1}}
	= C \ep^{\frac12}.
\end{align*}
For the short range interactions, we use \eqref{Taylorexpansionsecondphi0}  to estimate
\begin{align*}
\abs{I}
	&\leq C \int_{\abs{z}<\ep^{-\frac12}}   \ep \abs{z}^2 \, \frac{dz}{\abs{z}^{n+1}}
	= C \ep^{\frac12}.
\end{align*}
Estimate \eqref{aLinfinityestimate0} follows. 

Consequently, using the behavior of $\phi$ at $\pm\infty$ in \eqref{eq:standing wave}, we estimate
\begin{equation}\label{eq:abar-eps-proof}
\begin{aligned}
|\bar{a}_{\ep}(t,x)|
	&\leq \frac{1}{\ep \abs{\ln \ep}} \int_{\R} \abs{a_{\ep}(\xi;t,x)} \dot{\phi}(\xi) \, d \xi
	\leq \frac{C \ep^{\frac12}}{\ep \abs{\ln \ep}} \int_{\R} \dot{\phi}(\xi) \, d \xi
	= \frac{C}{\ep^{\frac12} \abs{\ln \ep}},
\end{aligned}
\end{equation}
which gives  \eqref{abarestimate}.

\subsection{Proof of \eqref{aderivativesLinfinityestimate} and \eqref{abarestimatextder}}

We first show the estimate for $\partial_t {a}_\ep$ in \eqref{aderivativesLinfinityestimate}. For this, we write
\begin{align*}
\partial_t a_\ep(\xi;t,x)
	&= \int_{|z|<\frac{\gamma}{\ep}} \left\{ \dot{\phi}\left( \xi + \frac{d(t,x+\ep z) - d(t,x)}{\ep}\right) 		\frac{d_t(t,x+\ep z) - d_t(t,x)}{\ep} \right.\\
	&\quad\qquad  - \dot{\phi}\left( \xi + \nabla d(t,x)\cdot z \right) \nabla d_t(t,x)\cdot z \bigg\} \frac{dz}{|z|^{n+1}}\\
	&=  \int_{|z|<\gamma} (\ldots) + \int_{\gamma< |z|<\frac{\gamma}{\ep}} (\ldots)
	=: I + II. 
\end{align*}
Beginning with $I$, we write
\begin{align*}
I
	&= \int_{|z|<\gamma} \left\{ \dot{\phi}\left( \xi + \frac{d(t,x+\ep z) - d(t,x)}{\ep}\right) 		\frac{d_t(t,x+\ep z) - d_t(t,x)-\nabla d_t(t,x)\cdot (\ep z) }{\ep} \right.\\
	&\quad+\left( \dot{\phi}\left( \xi + \frac{d(t,x+\ep z) - d(t,x)}{\ep}\right) - \dot{\phi}\left( \xi + \nabla d(t,x)\cdot z \right)\right) \nabla d_t(t,x)\cdot z \bigg\} \frac{dz}{|z|^{n+1}}.
\end{align*}
With the regularity of $\phi$ and $d$, we have
\begin{align*}
|I|
	&\leq  C\int_{|z|<\gamma} \left\{
	\frac{|d_t(t,x+\ep z) - d_t(t,x)-\nabla d_t(t,x)\cdot (\ep z) |}{\ep} \right.\\
	&\quad+\frac{|d_t(t,x+\ep z) - d_t(t,x)-\nabla d_t(t,x)\cdot (\ep z) |}{\ep}| \nabla d_t(t,x)\cdot z| \bigg\} \frac{dz}{|z|^{n+1}}\\
	&\leq C \int_{|z|<\gamma} \left( \ep |z|^2 + \ep |z|^3\right) \frac{dz}{|z|^{n+1}} \leq C \gamma \ep\leq C |\ln \ep|. 
\end{align*}
Next, we estimate $II$. By the regularity of $\phi$ and $d$, we have
\begin{align*}
|II|
&\leq C \int_{\gamma < |z|< \frac{\gamma}{\ep}} |z| \frac{dz}{|z|^{n+1}} 
= C |\ln \ep|.
\end{align*}
Combining the estimates for $I$ and $II$, we have \eqref{aderivativesLinfinityestimate} for $\partial_ta_\ep$. 
The estimate for $\nabla_x a_\ep$ in \eqref{aderivativesLinfinityestimate} follows along the same lines. 

We now check \eqref{aderivativesLinfinityestimate} for $D^2_x a_\ep$. 
For ease, we drop the notation in $t$ and write
\begin{align*}
&D^2_x a_\ep(\xi;x)\\
	&=  \int_{|z|<\frac{\gamma}{\ep}} \bigg\{ \ddot{\phi}\left( \xi + \frac{d(x+\ep z) - d(x)}{\ep}\right) \frac{(\nabla d(x+\ep z) - \nabla d(x))\otimes (\nabla d(x+\ep z) - \nabla d(x))}{\ep^2}  \\
	&\quad\qquad - \ddot{\phi}\left( \xi + \nabla d(x)\cdot z \right) D^2d(x)z \otimes D^2d(x)z \\
	&\quad\qquad + \dot{\phi}\left( \xi + \frac{d(x+\ep z) - d(x)}{\ep}\right)  \frac{D^2d(x+\ep z) - D^2d(x)}{\ep}\\
	&\quad\qquad - \dot{\phi}\left( \xi + \nabla d(x)\cdot z \right) D^3d(x)z.
	\bigg\} \frac{dz}{|z|^{n+1}}.
\end{align*}
Using the regularity of $\phi$ and $d$, we estimate to find
\begin{equation}\label{D2axsplit}\begin{split}
&|D^2_x a_\ep(\xi;x)|\\
	&\leq  \int_{|z|<\frac{\gamma}{\ep}} \bigg\{ \bigg|\ddot{\phi}\left( \xi + \frac{d(x+\ep z) - d(x)}{\ep}\right) \frac{(\nabla d(x+\ep z) - \nabla d(x))\otimes (\nabla d(x+\ep z) - \nabla d(x))}{\ep^2} \bigg| \\
	&\quad\qquad + \bigg| \ddot{\phi}\left( \xi + \nabla d(x)\cdot z \right) D^2d(x)z \otimes D^2d(x)z\bigg| \\
	&\quad\qquad + \bigg|\left(\dot{\phi}\left( \xi + \frac{d(x+\ep z) - d(x)}{\ep}\right)  - \dot{\phi}(\xi) \right) \frac{D^2d(x+\ep z) - D^2d(x)}{\ep}\bigg|\\
	&\quad\qquad +\bigg| \left(\dot{\phi}\left( \xi + \nabla d(x)\cdot z \right) - \dot{\phi}(\xi) \right)D^3d(x)z \bigg|\\
	&\quad\qquad +\bigg| \dot{\phi}(\xi) \left( \frac{D^2d(x+\ep z) - D^2d(x)}{\ep} - D^3d(x)z\right)	\bigg|
	\bigg\} \frac{dz}{|z|^{n+1}}\\
	&\leq C \int_{|z|<\frac{\gamma}{\ep}}\bigg\{ \frac{|\nabla d(x+\ep z) - \nabla d(x)|^2}{\ep^2}  
	 +|D^2d(x)z \otimes D^2d(x)z| \\
	 &\quad\qquad+  \frac{|d(x+\ep z) - d(x)|}{\ep} \frac{|D^2d(x+\ep z) - D^2d(x)|}{\ep}
	  + | \nabla d(x)\cdot z||D^3d(x)z|\\
	&\quad\qquad +\left| \frac{D^2d(x+\ep z) - D^2d(x)}{\ep} - D^3d(x)z	\right|
	\bigg\} \frac{dz}{|z|^{n+1}}\\
	&\leq C\int_{|z|< \frac{\gamma}{\ep}}  |z|^2 \frac{dz}{|z|^{n+1}}
	 \leq \frac{C}{\ep}. 
\end{split}
\end{equation}
With \eqref{aderivativesLinfinityestimate}, we estimate as in \eqref{eq:abar-eps-proof} to obtain
\[
|\partial_t \bar{a}_\ep (t,x)|
	\leq \frac{1}{\ep |\ln \ep|} \int_{\R} |\partial_t a_\ep(\xi;t,x)| \dot{\phi}(\xi) \, d \xi
	\leq \frac{C |\ln \ep|}{\ep |\ln \ep|} \int_{\R}  \dot{\phi}(\xi) \, d \xi
	=  \frac{C}{\ep}
\]
and similarly for $\nabla_x a_\ep$ and $D^2_x a_\ep$. This gives  \eqref{abarestimatextder}.

\medskip

\subsection{Proof of \eqref{aL2estimatefar}}
By  \eqref{aLinfinityestimate}, we can assume that $|\xi|>1$. 

Choose $\kappa>0$ such that 
\begin{equation}\label{kappaaepest}\|\nabla d\|_\infty \kappa<\frac14.\end{equation}
Then, for $|z|<\kappa |\xi|$,  we have that $| \nabla d(t,x) \cdot z|\leq |\xi|/4$ and $ \frac{|d(t,x+\ep z) - d(t,x)|}{\ep}\leq \|\nabla d\|_\infty |z|\leq |\xi|/4$. 
In particular, by \eqref{Taylorexpansionsecondphi0} and the  estimate  \eqref{eq:asymptotics for phi dot} for $\dot\phi$, 

\begin{equation}\label{estimatecase1_0far}\left|\phi\(\xi +  \frac{d(t,x+\ep z) - d(t,x)}{\ep}\) - \phi\( \xi + \nabla d(t,x) \cdot z\)\right|\leq \frac{C\ep|z|^2}{|\xi|^2}.
\end{equation}
Now, we write
\begin{align*}
a_{\ep}(\xi;t,x)&=\int_ {\{|z|<\frac{\gamma}{\ep}\}} \( \phi\(\xi +  \frac{d(t,x+\ep z) - d(t,x)}{\ep}\) - \phi\( \xi + \nabla d(t,x) \cdot z\) \) \frac{dz}{\abs{z}^{n+1}}\\
	&=\int_{\{|z|<\kappa |\xi|\}\cap \{|z|<\frac{\gamma}{\ep}\}} (\ldots)
		+ \int_{\{|z|>\kappa |\xi|\}\cap \{|z|<\frac{\gamma}{\ep}\}}  (\ldots)\\&
=: I + II.
\end{align*}
By  \eqref{estimatecase1_0far}, \begin{equation*}\label{Icase1prooffarfromthefrontaep}|I|\leq \frac{C\ep}{|\xi|^2} \int_{\{|z|<\kappa |\xi|\}}\frac{dz}{\abs{z}^{n-1}}=\frac{C\ep}{|\xi|}. 
\end{equation*}
Next, we have
$$|II|\leq 2 \int_{\{|z|>\kappa |\xi|\}}\frac{dz}{\abs{z}^{n+1}}\leq \frac{C}{|\xi|}.$$
Estimate \eqref{aL2estimatefar} for $a_\ep$ then follows. 

Estimate \eqref{aL2estimatefar} for $\dot{a}_{\ep}$ is proven in the same way, just replacing $\phi$ by $\dot{\phi}$.

\medskip

\subsection{Proof of   \eqref{aL2estimatefarxtder}}
By \eqref{aderivativesLinfinityestimate}, we can  assume that $|\xi|>1$. 

Choose $\kappa>0$ satisfying \eqref{kappaaepest}. 
We split
\begin{align*}
\partial_t a_\ep(\xi;t,x)
	&= \int_{\{|z|<\frac{\gamma}{\ep}\}} \left\{ \dot{\phi}\left( \xi + \frac{d(t,x+\ep z) - d(t,x)}{\ep}\right) 		\frac{d_t(t,x+\ep z) - d_t(t,x)}{\ep} \right.\\
	&\quad\qquad  - \dot{\phi}\left( \xi + \nabla d(t,x)\cdot z \right) \nabla d_t(t,x)\cdot z \bigg\} \frac{dz}{|z|^{n+1}}\\
	&=  \int_{\{|z|<\kappa |\xi|\}\cap \{|z|<\frac{\gamma}{\ep}\}} (\ldots) + \int_{\{|z|>\kappa |\xi|\}\cap \{|z|<\frac{\gamma}{\ep}\}}  (\ldots)
	=: I + II. 
\end{align*}
Beginning with $I$, we write
\begin{align*}
I&= \int_{\{|z|<\kappa |\xi|\}\cap \{|z|<\frac{\gamma}{\ep}\}} \left\{ \dot{\phi}\left( \xi + \frac{d(t,x+\ep z) - d(t,x)}{\ep}\right) 		\frac{d_t(t,x+\ep z) - d_t(t,x)-\nabla d_t(t,x)\cdot (\ep z) }{\ep} \right.\\
	&\quad+\left( \dot{\phi}\left( \xi + \frac{d(t,x+\ep z) - d(t,x)}{\ep}\right) - \dot{\phi}\left( \xi + \nabla d(t,x)\cdot z \right)\right) \nabla d_t(t,x)\cdot z \bigg\} \frac{dz}{|z|^{n+1}}.
\end{align*}
As in the proof of  \eqref{aL2estimatefar},  by  the  estimates  \eqref{eq:asymptotics for phi dot} for $\dot\phi$ and 
 $\ddot\phi$, if $|z|<\kappa |\xi|$ and $|z|<\frac{\gamma}{\ep}$, then 

$$0<\dot{\phi}\left( \xi + \frac{d(t,x+\ep z) - d(t,x)}{\ep}\right)\leq \frac{C}{|\xi|^2}$$ 
and 
$$ \left|\dot{\phi}\left( \xi + \frac{d(t,x+\ep z) - d(t,x)}{\ep}\right) - \dot{\phi}\left( \xi + \nabla d(t,x)\cdot z \right)\right|\leq  \frac{C\ep }{|\xi|^2}|z|^2\leq \frac{C }{|\xi|^2}|z| .$$
Therefore,
$$|I|\leq  \frac{C}{|\xi|^2}\int_{\{|z|<\kappa |\xi|\}}(\ep |z|^2+|z|^2)\frac{dz}{|z|^{n+1}}\leq \frac{C}{|\xi|}.$$
Regarding $II$, we have
\begin{align*}
|II|&\leq C\int_{\{|z|>\kappa |\xi|\}\cap \{|z|<\frac{\gamma}{\ep}\}}\left (\frac{1}{\ep}+|z|\right)\frac{dz}{|z|^{n+1}}\leq \frac{C}{\ep}\int_{\{|z|>\kappa |\xi|\}}\frac{dz}{|z|^{n+1}}\leq \frac{C}{\ep|\xi|}.
\end{align*}
From the estimates of $I$ and $II$, \eqref{aL2estimatefarxtder} for $\partial_t a_\ep$ follows. 

The estimate for  $\nabla_x a_\ep$ can be proven in the same way. 

We next  check \eqref{aL2estimatefarxtder}   for $D^2_x a_\ep$. 
As before, if $\kappa$ satisfies  \eqref{kappaaepest} and if $|z|<\kappa |\xi|$, 
then,  by   \eqref{eq:asymptotics for phi dot} for $\dot\phi$ and 
 $\ddot\phi$, 
$$\left |\ddot{\phi}\left( \xi + \frac{d(t,x+\ep z) - d(t,x)}{\ep}\right)\right|,\, |\ddot{\phi}\left( \xi + \nabla d(t,x)\cdot z \right)|,\,  \dot{\phi}(\xi) \leq \frac{C}{|\xi|^2},$$
$$\left|\dot{\phi}\left( \xi + \nabla d(t,x)\cdot z \right) - \dot{\phi}(\xi) \right|,\,\left|\dot{\phi}\left( \xi + \frac{d(t,x+\ep z) - d(t,x)}{\ep}\right)  - \dot{\phi}(\xi) \right|\leq \frac{C|z|}{|\xi|^2}.$$
Therefore, recalling \eqref{D2axsplit},
\begin{align*}
|D^2_x a_\ep(\xi;t,x)|
	&\leq  \frac{C}{|\xi|^2} \int_{\{|z|<\kappa |\xi|\}\cap \{|z|<\frac{\gamma}{\ep}\}}|z|^2\frac{dz}{|z|^{n+1}}+C\int_{\{|z|>\kappa |\xi|\}\cap \{|z|<\frac{\gamma}{\ep}\}}|z|^2\frac{dz}{|z|^{n+1}}\\&
	\leq \frac{C}{|\xi|^2} \int_{\{|z|<\kappa |\xi|\}}\frac{dz}{|z|^{n-1}}+\frac{C}{\ep^2} \int_{\{|z|>\kappa|\xi|\}}\frac{dz}{|z|^{n+1}}\leq   \frac{C}{\ep^2|\xi|}.
\end{align*}
\medskip

The proof of the second line in \eqref{aL2estimatefarxtder} is proven the same way, just replacing $\phi$ with $\dot{\phi}$. 

\medskip

The proof of  Lemma \ref{lem:ae bound} is then complete. \qed

\section{Proof of Lemma \ref{lem:ae-updates}} \label{sec:ae-updates}

\subsection{Proof of  \eqref{eq:aeestnonzerograd}} 
For ease, we drop the notation in $t$.  Note that 
\[
\frac{d(x+\ep z) - d(x)}{\ep} =  \nabla d(x) \cdot z
	+ \ep O(|z|^2).
\] 
If $|\nabla d(x)|\leq C_0\ep$, with $C_0>0$ to be determined, then 
$$|a_\ep(\xi;x)|\leq C\ep  \int_{\{|z|<\frac{\gamma}{\ep}\}}\frac{dz}{|z|^{n-1}}\leq C\leq \frac{C \ep |\ln \ep|}{|\nabla d(x)|},$$ 
as desired.

Next, let us assume 
\begin{equation}\label{eqaesplitnabla}c_1:=|\nabla d(x)|> C_0\ep.\end{equation}
Let $T$ be an orthogonal matrix such that if $z = Ty$, then 
\begin{equation}\label{eq:change-of-variablesT}
\nabla d(x) \cdot z = \nabla d(x) \cdot Ty = c_1 y_n . 
\end{equation}
With the change of variables $\ep z = Ty$ and using the monotonicity of $\phi$,  we obtain
\begin{align*}
a_\ep(\xi;x)
	&= \ep \int_{\{|y|<\gamma\}}  
	 \left\{ \phi \left( \xi + \frac{c_1 y_n}{\ep} + \frac{O( |y|^2)}{\ep}\right)
		 - \phi\left( \xi + \frac{c_1y_n}{\ep} \right) \right\} \frac{dy}{|y|^{n+1}}\\
		& \leq \ep \int_{\R^n}
	 \left\{ \phi \left( \xi + \frac{c_1 y_n}{\ep} + \frac{C |y|^2}{\ep}\right)
		 - \phi\left( \xi + \frac{c_1y_n}{\ep} \right) \right\} \frac{dy}{|y|^{n+1}}.
\end{align*}
Fix $r>\ep  $, to be chosen. For $y = (y', y_n)$, we write
\begin{align*}
\frac{a_\ep(\xi;x)}{\ep}
	&\le  \int_{ \{|y'|>r\}\cup\{ |y_n|>r\} }(\ldots)
		+ \int_{ \{|y'|, |y_n|<\ep \}}(\ldots)\\
		&\quad+ \underbrace{\int_{\{|y_n|<r\} \cap \{\ep < |y'|<r\}}(\ldots)}_{:= I}
		+  \underbrace{\int_{\{|y'|<\ep\} \cap \{\ep < |y_n|<r\}}(\ldots)}_{:= II}. 
\end{align*}
Since
\begin{align*}
&\int_{\{|y'|>r\}\cup\{ |y_n|>r\} }(\ldots)  \frac{dy}{|y|^{n+1}}
	 \leq 2 \int_{\{|y|>r\}} \frac{dy}{|y|^{n+1}} 
	 \leq \frac{C}{r}
\end{align*}
and 
\begin{align*}
&\int_{ \{|y'|, |y_n|<\ep \}}(\ldots) \frac{dy}{|y|^{n+1}}
	 \leq \frac{C}{\ep} \int_{ \{|y|< \sqrt{2}\ep \}}|y|^2 \frac{dy}{|y|^{n+1}} 
	 \leq C,
\end{align*}
we have
\begin{equation}\label{eq:aesplit}
\frac{a_\ep(\xi;x)}{\ep}
	\leq  O(1) + O\left(\frac{1}{r}\right) + I + II.
\end{equation}
Next, we will show that
\begin{equation}\label{eq:aesplitI}
I\leq \frac{C }{c_1}|\ln \ep|.
\end{equation}
For this, we will make the change of variables 
\begin{equation}\label{eq:change-of-variables}
t = \frac{y_n}{|y'|} \quad \hbox{and} \quad
y' = \theta p, \quad \theta \in S^{n-2}.
\end{equation}
Indeed,  changing variables first in $t$ for a fixed $y'$, we have
\begin{align*}
I
&= \int_{\{\ep < |y'|<r\}} dy' \int_{\{|y_n|<r\}} \frac{dy_n}{(|y'|^2 + |y_n|^2)^{\frac{n+1}{2}}}\\
&\qquad \left\{ \phi \left( \xi + \frac{c_1 y_n}{\ep} + \frac{C}{\ep}  ( |y'|^2 + |y_n|^2)\right)
		 - \phi\left( \xi + \frac{c_1 y_n}{\ep} \right) \right\} \\
&= \int_{\{\ep < |y'|<r\}}\frac{dy'}{|y'|^{n}}\int_{\{|t| < \frac{r}{|y'|} \}} \, \frac{dt}{\left(1 +t^2\right)^{\frac{n+1}{2}}}\\
&\qquad \left\{ \phi \left( \xi + \frac{|y'|c_1}{\ep}  t + C \frac{|y'|^2}{\ep}(1 + t^2)\right)
		 - \phi\left( \xi + \frac{|y'|c_1}{\ep}  t\right) \right\}.
\end{align*}
Then, we change to polar coordinates in $y'$ to write
\begin{align*}
I &=  \int_{S^{n-2}} d \theta \int_{\ep}^r \frac{dp}{p^2}\int_{\{|t| < \frac{r}{p} \}}\frac{dt}{ (1 +t^2)^{\frac{n+1}{2}}}\\
&\qquad \bigg\{ \phi \left( \xi + \frac{p c_1}{\ep}t  +C\frac{p^2}{\ep}(1 + t^2)\right)
		 - \phi\left( \xi + \frac{pc_1 }{\ep}  t\right) \bigg\}.
\end{align*}
We split
\[
I = \int_{S^{n-2}} d \theta \int_{\ep}^r \frac{dp}{p^2}\int_{\{|t| <1 \}} \, dt (\dots) 
+ \int_{S^{n-2}} d \theta \int_{\ep}^r \frac{dp}{p^2}\int_{\{1<|t| <\frac{r}{p} \}} \, dt (\dots)  =: I_1 + I_2. 
\]
If $|t|<1$, then
$1+t^2\leq C$. With this and the monotonicity of $\phi$, 
\begin{align*}
 I_1
	&\leq C\int_{\ep}^r \frac{dp}{p^2}\int_{\{|t| < 1 \}}
	 \bigg\{ \phi \left( \xi + \frac{p c_1}{\ep} t +C\frac{p^2}{\ep}\right)
		 - \phi\left( \xi + \frac{p c_1 }{\ep}  t\right) \bigg\} \, dt\\
	&= C\int_{\ep}^r \frac{dp}{p^2}\int_{\{|t| < 1\}} dt\int_0^1
	\dot\phi \left( \xi + \frac{p c_1}{\ep} t +C\frac{ p^2 \tau}{\ep} \right) \frac{p^2}{\ep} d \tau \\
	&= \frac{C}{\ep}\int_{\ep}^r d p  \int_0^1 d\tau \int_{\{|t| < 1\}}
	\partial_t \left[\phi \left( \xi + \frac{p c_1}{\ep} t +C\frac{ p^2 \tau}{\ep} \right) \right] \frac{\ep}{pc_1} d t\\
	&=\frac{C}{c_1}\int_{\ep}^r \frac{d p}{p}  \int_0^1 d\tau 
	\left\{\phi \left( \xi + \frac{p c_1}{\ep} +C\frac{ p^2 \tau}{\ep} \right) 
	-\phi \left( \xi - \frac{p c_1}{\ep} +C\frac{ p^2 \tau}{\ep}\right) \right\} \\
	&\leq \frac{ C|\ln\ep|}{c_1}. 
\end{align*}
This shows that 
\begin{equation}\label{eq:I1estimate}
I_1\leq\frac{ C|\ln\ep|}{c_1}. 
\end{equation}
Next, estimating $I_2$, we use that the integrand function is positive to find
\begin{align*}
I_2 &
	 \leq C
	  \int_{\ep}^{r} \frac{dp}{p^2}\int_{\{|t| >1 \}} \frac{dt}{(1+t^2)^{\frac{n+1}{2}}}\quad \bigg\{ \phi \left( \xi + \frac{p c_1}{\ep} t + \frac{Cp ^2}{\ep} (1+t^2)\right)
		 - \phi\left( \xi + \frac{p c_1 }{\ep}  t\right) \bigg\}. 
	\end{align*}
Let $r$ satisfy
\begin{equation}\label{eq:c1r}
c_1 \geq r,
\end{equation}
so that $c_1/p>1$ for $0<p<r$. Then we can split 
\begin{align*}
I_2 &\leq 
C
\int_{\ep}^{r} \frac{dp}{p^2}\int_{\{|t|>\sqrt{\frac{c_1}{p}} \}} \, dt (\dots)
+
C 
\int_{\ep}^{r} \frac{dp}{p^2}\int_{\{1<|t| <\sqrt{\frac{c_1}{p}} \}} \, dt (\dots)\\
  &=: J_1 + J_2.
\end{align*}
We will show that
\begin{equation}\label{eq:J1J2est}
 J_1,J_2 \leq\frac{ C|\ln\ep|}{c_1}.
\end{equation}
For $J_1$, we estimate
\begin{align*}
J_1
	&\leq C  \int_{\ep}^{r} \frac{dp}{p^2}\int_{\{|t| >\sqrt{\frac{c_1}{p}} \}} \, \frac{dt}{|t|^{n+1}}
	= \frac{C}{c_1^{\frac{n}{2}}} \int_{\ep}^{r} p^{\frac{n}{2}-2} dp 
	\leq  \frac{C}{c_1} \int_{\ep}^{r} \frac{dp}{p}
	\leq \frac{C}{c_1}|\ln\ep|.
\end{align*}
Estimating $J_2$, we find
\begin{align*}
J_2
	&=C \int_{\ep}^{r} \frac{dp}{p^2}\int_{\{1<|t| <\sqrt{\frac{c_1}{p}} \}} \frac{dt}{(1+t^2)^{\frac{n+1}{2}}} \int_0^1 
	  \dot{\phi} \left( \xi + \frac{p c_1}{\ep}  t +\frac{Cp^2 \tau}{\ep} (1+t^2)\right) \frac{p^2}{\ep} (1+t^2)d \tau\\
	&= C \int_{\ep}^{r} \frac{dp}{p} \int_0^1 d \tau \int_{\{1<|t| <\sqrt{\frac{c_1}{p}} \}} \frac{dt}{(1+t^2)^{\frac{n-1}{2}}} 
	 \partial_t \left[ \phi\left( \xi + \frac{p c_1}{\ep}  t +\frac{Cp^2 \tau}{\ep} (1+t^2)\right) \right] \frac{1}{c_1 + 2Cp \tau t}.
\end{align*}
Now, if
\[
\ep < p < r, \quad 0 < \tau < 1 \quad \hbox{and} \quad 1 \leq |t| \leq \sqrt{\frac{c_1}{p}},
\]
then
\[
|2C p\tau t| \leq 2C p \sqrt{\frac{c_1}{p}} = 2C \sqrt{c_1 p}\leq 2C \sqrt{c_1r},
\]
where we can assume $C>1$ in the last inequality above. Choose
\begin{equation}\label{eq:c1r-update}
r=\frac{c_1}{16C^2},
\end{equation}
then \eqref{eq:c1r} is satisfied, and choosing $C_0=16 C^2$ in \eqref{eqaesplitnabla} we also have that $r>\ep$. Moreover, 
$$c_1 + 2Cp\tau t\geq c_1- 2C \sqrt{c_1r}\geq \frac{c_1}{2}. $$
Then, integrating by parts, we get
\begin{align*}
J_2
	&\leq \frac{C}{c_1}  \int_{\ep}^{r} \frac{dp}{p} \int_0^1 d \tau \int_{\{1<|t| <\sqrt{\frac{c_1}{p}} \}} 
	  \partial_t \left[ \phi\left( \xi + \frac{p c_1}{\ep}  t +\frac{Cp^2 \tau}{\ep} (1+t^2)\right) \right]\frac{dt}{|t|^{n-1}} \\
	  &= \frac{C}{c_1}  \int_{\ep}^{r} \frac{dp}{p} \int_0^1 d \tau 
	  \Bigg[ (n-1)\int_{\{1<|t| <\sqrt{\frac{c_1}{p}} \}} 
	    \phi\left( \xi + \frac{p c_1 }{\ep}  t +\frac{Cp^2 \tau}{\ep} (1+t^2)\right) \frac{t}{|t|} \frac{dt}{|t|^{n}} \\
	    &\quad +  \phi\left( \xi + \frac{pc_1}{\ep}  t +\frac{Cp^2 \tau}{\ep} (1+t^2)\right) \frac{1}{|t|^{n-1}}\bigg|_{-\sqrt{\frac{c_1}{p}}}^{-1}\\
	    &\quad +  \phi\left( \xi + \frac{pc_1}{\ep}  t +\frac{Cp^2 \tau}{\ep} (1+t^2)\right) \frac{1}{|t|^{n-1}}\bigg|^{\sqrt{\frac{c_1}{p}}}_{1}\Bigg]\\
	    &\leq \frac{C}{c_1}  \int_{\ep}^{r} \frac{dp}{p} \int_0^1 d \tau  \left [ \int_{|t|>1}\frac{dt}{|t|^n}+1\right]\\
	    &= \frac{ C|\ln\ep|}{c_1}.
\end{align*}
  Hence, we obtain \eqref{eq:J1J2est}. The desired estimate \eqref{eq:aesplitI} follows from \eqref{eq:I1estimate} and \eqref{eq:J1J2est}.

It remains to estimate $II$. First, we make the same change of variables as in $I$ to write
\begin{align*}
II
	&=  \int_{S^{n-2}} d \theta \int_{0}^\ep \frac{dp}{p^2}\int_{\{\frac{\ep}{p} < |t|< \frac{r}{p}\}}\frac{dt}{ (1 +t^2)^{\frac{n+1}{2}}}\\
&\qquad \bigg\{ \phi \left( \xi + \frac{pc_1}{\ep} t + C\frac{p^2}{\ep}(1 + t^2)\right)
		 - \phi\left( \xi + \frac{p c_1}{\ep}  t\right) \bigg\}.
\end{align*}
With the regularity of $\phi$ and $d$, we have
\begin{align*}
II
	&\leq C 
	 \int_{0}^\ep \frac{dp}{p^2}\int_{\{\frac{\ep}{p} < |t|< \frac{r}{p}\}}
	\frac{p^2}{\ep} 
	\frac{dt}{ (1 +t^2)^{\frac{n-1}{2}}}\\
	&\leq  \frac{C}{\ep} \int_{0}^\ep dp \int_{\{1 < |t|< \frac{r}{p}\}}\frac{dt}{|t|}\\
	&\leq \frac{C}{\ep} \int_{0}^\ep (\ln r-\ln p) \, dp
	 \leq C |\ln \ep|. 
\end{align*}
Combining the previous estimate for $II$ with \eqref{eq:aesplit},   \eqref{eq:aesplitI} and \eqref{eq:c1r-update}   we get
\[
\frac{a_\ep(\xi;x)}{\ep}
	\leq \frac{C}{r} + \frac{ C|\ln\ep|}{c_1}\leq \frac{ C|\ln\ep|}{c_1}.
\] Similarly, one can show
\[
\frac{a_\ep(\xi;x)}{\ep}
	\geq - \frac{ C|\ln\ep|}{c_1}.
\] 
This proves \eqref{eq:aeestnonzerograd} for $a_\ep$. 
The estimates for $\dot{a}_\ep$ and $\ddot{a}_\ep$ are proved in the same way, just replacing $\phi$ by $\dot{\phi}$ and $\ddot{\phi}$, respectively. 
Lastly, \eqref{eq:aeestnonzerograd} for $\bar{a}_\ep$ follows from the estimate for $a_\ep$. 

\subsection{Proof of  \eqref{aderivativesclosetothefront}}
We will only prove the estimate for $\nabla_x a_\ep$, the  estimate for  $\partial_t a_\ep$ and 
$\nabla_x \dot a_\ep$ will follow similarly. Moreover,   \eqref{aderivativesclosetothefront} for the derivatives of $\bar{a}_\ep$ follows from the estimates for the derivatives of $a_\ep$.
For ease, we continue to drop the notation in $t$. 

We perform the same Taylor expansion  as above but we make the error term explicit, 
$$\frac{d(x+\ep z) - d(x)}{\ep} =  \nabla d(x) \cdot z
	 +\ep \int_0^1D^2d(x+\tau\ep z)(1-\tau)d\tau\, z\cdot z.
	$$
Recall  that  by Definition  \ref{defn:extension}, $d$ is  the smooth signed distance function from the front  $\Gamma_t$ in the set 
$\{|d|<\rho\}$ . Therefore, 
$|\nabla d(x)|=1$ for all $x$ in that set. As in \eqref{eq:change-of-variablesT} (with $c_1=1$), let $T$ be an orthogonal matrix such that if $z = Ty$, then 
\begin{equation}\label{eq:change-of-variablesT-2}
\nabla d(x) \cdot z = \nabla d(x) \cdot Ty =  y_n.
\end{equation}
With the change of variables $ z = Ty$, we obtain, for all $x\in \{|d|<\rho\}$, 
\begin{align*}
a_\ep(\xi;x)
	&=  \int_{\{|y|<\frac{\gamma}{\ep}\}}   \frac{dy}{|y|^{n+1}}
	 \left\{ \phi \left( \xi +  y_n +\ep \int_0^1D^2d(x+\tau\ep Ty)(1-\tau)d\tau\, Ty\cdot Ty\right)
		 - \phi\left( \xi + y_n \right) \right\}.
\end{align*}
Therefore, for $i=1,\ldots, n$, 
\begin{align*}
\partial_{x_i} a_\ep(\xi;x)&=  \int_{\{|y|<\frac{\gamma}{\ep}\}}   \frac{dy}{|y|^{n+1}}
	\dot  \phi \left( \xi +  y_n +\ep \int_0^1D^2d(x+\tau\ep Ty)(1-\tau)d\tau\, Ty\cdot Ty\right)\\
	&\qquad \cdot \ep \int_0^1D^2d_{x_i}(x+\tau\ep Ty)(1-\tau)d\tau\, Ty\cdot Ty,
	\end{align*}
	from which we deduce that 
	$$|\partial_{x_i} a_\ep(\xi;x)|\leq \ep  \int_{\{|y|<\frac{\gamma}{\ep}\}}   \frac{dy}{|y|^{n-1}}\leq C,$$
	as desired. 

\subsection{Proof of \eqref{aeclosetofront-xi}}
For ease, we continue to  drop the notation in $t$. 

If $x\in Q_\rho$, then, recalling Definition \ref{defn:extension},  $|\nabla d(x)|=1$. 
Estimate  \eqref{aeclosetofront-xi} is then a consequence of estimate \eqref{eq:aeestnonzerograd} for $a_\ep$ when $|\xi|\leq \gamma^{-2}$ for $\ep$ small enough. Therefore, we  assume 
$|\xi|>\gamma^{-2}$.

We will show that 
\begin{equation}\label{aeclosetofront-xi-oneside}
a_\ep(\xi;x)\leq\frac{C\gamma}{|\xi|}.
\end{equation}
Performing first the change of variable $\ep z=Tz$ with $T$ as  in \eqref{eq:change-of-variablesT-2}  
and then the change of variable in \eqref{eq:change-of-variables}, we get 
\begin{align*}
a_\ep(\xi;x) &=  \ep \int_{S^{n-2}} d \theta \int_{0}^\gamma \frac{dp}{p^2}\int_{\left\{|t| < \sqrt{\frac{\gamma ^2}{p^2}-1}\right \}}\frac{dt}{ (1 +t^2)^{\frac{n+1}{2}}}\\
&\qquad \bigg\{ \phi \left( \xi + \frac{p }{\ep}t  +\frac{p^2}{\ep}O(1 + t^2)\right)
		 - \phi\left( \xi + \frac{p }{\ep}  t\right) \bigg\}.
\end{align*}
By the monotonicity of $\phi$, for some $C_0>0$, 
\begin{align*}
&a_\ep(\xi;x) \leq C\ep 
 \int_{0}^\gamma \frac{dp}{p^2}\int_{\left\{|t| < \frac{\gamma }{p}\right \}}\frac{dt}{ (1 +t^2)^{\frac{n+1}{2}}}
 \bigg\{ \phi \left( \xi + \frac{p }{\ep}t  +\frac{C_0p^2}{\ep} (1+t^2)\right)
		 - \phi\left( \xi + \frac{p }{\ep}  t\right) \bigg\}.
\end{align*}
		
If $|\xi|\geq 2\gamma\ep^{-1}(1+2C_0)$, then for $0<p<\gamma<1$ and $|t| < \frac{\gamma}{p}$, 
$$\left| \xi + \frac{p }{\ep}t  +\frac{C_0p^2}{\ep}(1+t^2) \right|
	\geq |\xi| - \frac{|pt|}{\ep} - \frac{C_0p^2}{\ep} - \frac{C_0|pt|}{\ep}\geq |\xi| -\frac{\gamma}{\ep}(1+2C_0)
	\geq \frac{|\xi|}{2}$$
and by \eqref{eq:asymptotics for phi dot}, for some $\tau\in(0,1)$, 
\begin{align*} 
\phi \left( \xi + \frac{p }{\ep}t  +\frac{C_0p^2}{\ep} (1+t^2)\right)
		 - \phi\left( \xi + \frac{p }{\ep}  t\right) 
&=\dot\phi\(\xi+ \frac{p }{\ep}  t+\tau \frac{C_0p^2}{\ep} (1+t^2)\)\frac{C_0p^2}{\ep} (1+t^2)\\&
\leq \frac{C}{\left| \xi + \frac{p }{\ep}t  +\frac{C_0p^2}{\ep}(1+t^2) \right|^2} \frac{C_0p^2}{\ep} (1+t^2)\\&
		 \leq \frac{C p^2}{\ep |\xi|^2} (1+t^2).
\end{align*}
Therefore, 
\begin{align*}
a_\ep(\xi;x)&\leq 	 \frac{C}{ |\xi|^2} 	 \int_{0}^\gamma dp\int_{\left\{|t| < \frac{\gamma}{p}\right \}}\frac{dt}{ (1 +t^2)^{\frac{n-1}{2}}}
\leq \frac{C}{ |\xi|^2}\(\gamma+ \int_{0}^\gamma dp\int_{\left\{1<|t| < \frac{\gamma}{p}\right \}}\frac{dt}{| t|}\)\\&
\leq  \frac{C}{ |\xi|^2}\(\gamma  +\int_{0}^\gamma (\ln\gamma-\ln p)dp\)
\leq 	 \frac{C\gamma }{ |\xi|^2} \leq 	 \frac{C\ep}{|\xi|}, 
\end{align*}		
where we used that $|\xi|\geq 2\gamma\ep^{-1}(1+2C_0)$ in the last inequality. This implies \eqref{aeclosetofront-xi-oneside}. 

Next, let us assume $ \gamma^{-2}<|\xi|<2\gamma\ep^{-1}(1+2C_0)$. We split, 
\begin{equation}\label{aepsplit:aeclosetofront-xi}\begin{split}
&\frac{a_\ep(\xi;x)}{\ep} \leq C
 \int_{0}^\gamma \frac{dp}{p^2}\int_{\left\{|t| < \frac{\gamma}{p}\right \}}\frac{dt}{ (1 +t^2)^{\frac{n+1}{2}}}
 \bigg\{ \phi \left( \xi + \frac{p }{\ep}t  +\frac{C_0p^2}{\ep} (1+t^2)\right)
		 - \phi\left( \xi + \frac{p }{\ep}  t\right) \bigg\}\\
		 &=C\int_{0}^\gamma \frac{dp}{p^2}\int_{\left\{|t| <1\right \}}(\ldots)+C\int_{0}^\gamma \frac{dp}{p^2}\int_{\left\{1<|t| < \frac{\gamma}{p}\right \}}(\ldots)\\&
		 =:I+II. 
\end{split}
\end{equation}
As in the proof of \eqref{eq:aeestnonzerograd}, for $C_1=2C_0$, 	
\begin{align}\label{Ifirstest:aeclosetofront-xi} I\leq C
	 \int_{0}^\gamma \frac{dp}{p} \int_0^1 \(\phi\( \xi + \frac{p }{\ep}  +\frac{C_1p^2\tau}{\ep}\right)-\phi\( \xi - \frac{p }{\ep}  +\frac{C_1p^2\tau}{\ep}\right)\)d\tau.
\end{align}
We next observe that
\begin{equation}\label{eq:p<p1} \frac{p }{\ep}  +\frac{C_1p^2\tau}{\ep}< \frac{|\xi|}{2}\quad \text{if}\quad 0< p< p_1:=\frac{\ep|\xi|}{1+\sqrt{1+2C_1\tau\ep|\xi|}},\end{equation}
and since 
 $|\xi|<2\gamma\ep^{-1}(1+2C_0)$,  
\begin{equation}\label{p1p2}c \ep|\xi|\leq p_1\leq C\ep|\xi|.\end{equation}
Since the integrand function in \eqref{Ifirstest:aeclosetofront-xi}  is positive,  we can write
\begin{align*}
I&\leq C\int_0^1d\tau\int_0^{p_1}dp\,(\ldots)+C\int_0^1d\tau \int_{p_1}^{\max\{\gamma,p_1\}} dp\,(\ldots)
=:J_1+J_2. 
\end{align*}
By \eqref{eq:p<p1}  and \eqref{eq:asymptotics for phi dot}, if $0<p<p_1$, 
\begin{align*}\phi\( \xi + \frac{p }{\ep}  +\frac{C_1p^2\tau}{\ep}\right)-\phi\( \xi - \frac{p }{\ep}  +\frac{C_1p^2\tau}{\ep}\right)\leq 
\frac{C}{\(|\xi|-\frac{p }{\ep}-\frac{C_1p^2\tau}{\ep}\)^2}\frac{p }{\ep} \leq \frac{Cp}{\ep|\xi|^{2}}.
\end{align*}
Therefore, using \eqref{p1p2} and that  $|\xi|<2\gamma\ep^{-1}(1+2C_0)$, 
$$J_1\leq \frac{C}{\ep|\xi|^{2}}\int_0^{p_1}p\,dp\leq  \frac{Cp_1^2}{\ep|\xi|^{2}}\leq C\ep \leq  \frac{C\gamma}{|\xi|}.$$
Next, again by \eqref{p1p2},
\begin{align*}J_2&\leq C  \int_{p_1}^{\max\{\gamma,p_1\}}  \frac{dp}{p}
\leq C\frac{\max\{\gamma,p_1\}-p_1}{p_1}\leq \frac{C\gamma}{\ep|\xi|}.\end{align*}
Putting it all together, we obtain
\begin{equation}\label{Iestimate:aeclosetofront-xi}\ep I\leq \frac{C\gamma}{|\xi|}.
\end{equation}
Let us now estimate $II$ in \eqref{aepsplit:aeclosetofront-xi}. By the monotonicity of $\phi$, 
\begin{align*}\label{IIfirstest:aeclosetofront-xi} 
 II&\leq \int_{0}^\gamma \frac{dp}{p^2}\int_{\left\{1<|t| < \frac{\gamma }{p}\right \}}\frac{dt}{ |t|^{n+1}}
 \bigg\{ \phi \left( \xi + \frac{p t}{\ep}  +\frac{C_1p^2t^2}{\ep} \right) - \phi\left( \xi + \frac{p t}{\ep}  \right) \bigg\}\\
 &= \int_{0}^\gamma  \frac{dp}{p^2}\int_{\left\{1<t < \frac{\gamma }{p}\right \}}(\ldots)+ \int_{0}^\gamma  \frac{dp}{p^2}\int_{\left\{- \frac{\gamma }{p}<t<-1\right \}}(\ldots)\\&
 =:II_1+II_2.
\end{align*}
We are only going to  estimate $II_1$ since the estimate for $II_2$  follows with a similar argument. Thus, we assume below that $t>1$. 
Observe that 
\begin{equation}\label{eqII:p<p1} \frac{p t}{\ep} +\frac{C_1p^2t^2}{\ep}< |\xi|-\frac{\gamma}{2} |\xi|\quad \text{if}\quad0< tp< p_1:=\frac{2\ep\(|\xi|-\frac{\gamma}{2} |\xi|\)}{1+\sqrt{1+4C_1\ep \(|\xi|-\frac{\gamma}{2} |\xi|\)}}\end{equation}
and 
\begin{equation}\label{eqII:p<p2} \frac{p t}{\ep}  > |\xi|+\gamma|\xi|\quad 
\text{if}\quad tp> p_2:=\ep\(|\xi|+\gamma |\xi|\). \end{equation}
Notice that for    $ |\xi|<2\gamma\ep^{-1}(1+2C_0)$ and $0<\gamma<1$, 
\begin{equation}\label{p1p2bis}c \ep|\xi|\leq p_1,\,p_2\leq C\ep|\xi|\end{equation}
and 
\begin{equation}\label{p1-p2} 0\leq p_2-p_1\leq \ep|\xi|\(1-\frac{2}{1+\sqrt{1+O(\gamma)}}\)+C\ep\gamma |\xi|\leq C\gamma\ep|\xi| .\end{equation}
Therefore, if $0<tp<p_1$, by  \eqref{eq:asymptotics for phi dot} and \eqref{eqII:p<p1}, 
\begin{align}\label{phiest1:aeclosetofront-xi}
 \phi \left( \xi + \frac{pt }{\ep}  +\frac{C_1p^2t^2}{\ep} \right) - \phi\left( \xi + \frac{pt }{\ep}  \right) \leq \frac{C}{\(|\xi|- \frac{p t}{\ep} -\frac{C_1p^2t^2}{\ep}\)^2}\frac{p^2t^2}{\ep} 
 \leq \frac{Cp^2 t^2}{\ep\gamma^2|\xi|^2},
\end{align} 
and if  $tp> p_2$, 
by    \eqref{eq:asymptotics for phi} and \eqref{eqII:p<p2}, 
\begin{align}\label{phiest2:aeclosetofront-xibis}
 \phi \left( \xi + \frac{pt }{\ep}  +\frac{C_1p^2t^2}{\ep} \right) - \phi\left( \xi + \frac{pt }{\ep}  \right)
 \leq 1-\(1- \frac{C}{ \xi + \frac{pt }{\ep}  }\)\leq  \frac{C}{\gamma|\xi|}.
\end{align}
Assume $\gamma>p_2$. The cases $\gamma<p_1$ and $p_1\leq \gamma\leq p_2$ can be treated similarly.  We split
\begin{align*}
II_1&\leq  \int_{0}^{p_1} (\ldots)+\int_{p_1}^{p_2} (\ldots)
+\int_{p_2}^\gamma (\ldots)
=:J_1+J_2+J_3.
\end{align*}
We further split
\begin{align*}
J_1=\int_{0}^{p_1}\frac{dp}{p^2} \int_1^{\frac{p_1}{p}}dt\, (\ldots)+\int_{0}^{p_1}\frac{dp}{p^2}  \int_{\frac{p_1}{p}}^{\frac{p_2}{p}}dt\, (\ldots)+\int_{0}^{p_1} \frac{dp}{p^2} \int_{\frac{p_2}{p}}^{\frac{\gamma}{p}}dt \,(\ldots)=:J_1^1+J_1^2+J_1^3.
\end{align*}
By \eqref{p1p2bis}, \eqref{phiest1:aeclosetofront-xi} and  for $\ep<\gamma^3$,  
\begin{align*}J_1^1&\leq \frac{C}{\ep \gamma^2|\xi|^2 } \int_{0}^{p_1}dp \int_1^{\frac{p_1}{p}}\frac{dt}{t^{n-1}}\leq \frac{C}{\ep   \gamma^2|\xi|^2 } \int_{0}^{p_1}dp \int_1^{\frac{p_1}{p}}\frac{dt}{t}
=\frac{C}{\ep  \gamma^2|\xi|^2 } \int_{0}^{p_1}(\ln p_1-\ln p)dp\\&= \frac{Cp_1}{\ep  \gamma^2|\xi|^2  }\leq \frac{C}{\gamma^2 |\xi|}\leq \frac{C\gamma}{\ep |\xi|}. \end{align*}
By \eqref{p1p2bis} and  \eqref{p1-p2}, 
\begin{align*}
J_1^2&\leq 2 \int_{0}^{p_1}\frac{dp}{p^2} \int_{\frac{p_1}{p}}^{\frac{p_2}{p}}\frac{dt}{t^{n+1}}\leq  2 \int_{0}^{p_1}\frac{dp}{p^2} \int_{\frac{p_1}{p}}^{\frac{p_2}{p}}\frac{dt}{t^{3}}
= \int_{0}^{p_1}\(\frac{1}{p_1^2}-\frac{1}{p_2^2}\)\, dp=\frac{(p_2-p_1)(p_1+p_2)}{p_1p_2^2}\\&\leq \frac{C\gamma}{\ep|\xi|}.
\end{align*}
By \eqref{p1p2bis},  \eqref{phiest2:aeclosetofront-xibis} and since $ |\xi|>\gamma^{-2}$,
\begin{align*}
J_1^3&\leq \frac{C}{\gamma|\xi|} \int_{0}^{p_1}\frac{dp}{p^2} \int_{\frac{p_2}{p}}^{\frac{\gamma}{p}}\frac{dt}{t^{n+1}}\leq  \frac{C}{\gamma|\xi|} \int_{0}^{p_1}\frac{dp}{p^2} \int_{\frac{p_2}{p}}^{\infty}\frac{dt}{t^{3}}=\frac{C}{\gamma|\xi|}\int_{0}^{p_1}\frac{dp}{p_2^2} = \frac{Cp_1}{\gamma |\xi|p_2^2 } \leq
 \frac{C }{\ep\gamma  |\xi|^2}\leq \frac{C\gamma }{\ep  |\xi|}.
 \end{align*}
Therefore,
\begin{equation*}\label{J1est:aeclosetofront-xi}J_1\leq  \frac{C\gamma }{\ep  |\xi|}.
\end{equation*}
Next, by \eqref{p1p2bis} and \eqref{p1-p2}, 
\begin{align*}\label{J2est:aeclosetofront-xi}
J_2&\leq 2\int_{p_1}^{p_2}\frac{dp}{p^2}\int_1^{\infty}\frac{dt}{t^{n+1}}\leq C\int_{p_1}^{p_2}\frac{dp}{p^2}=C\frac{p_2-p_1}{p_1p_2}\leq  \frac{C\gamma}{\ep|\xi|}.
\end{align*}
Finally, if $p_2<p<\gamma$ and $t>1$, then $tp> p_2$. Hence, by  \eqref{p1p2bis}, \eqref{phiest2:aeclosetofront-xibis} and since $ |\xi|>\gamma^{-2}$,
\begin{align*}
J_3\leq \frac{C}{\gamma|\xi|}\int_{p_2}^{\gamma}\frac{dp}{p^2}\int_1^{\infty}\frac{dt}{t^{n+1}}\leq  \frac{C}{\gamma|\xi|}\int_{p_2}^{\gamma}\frac{dp}{p^2} 
\leq\frac{C}{\gamma|\xi|p_2}\leq \frac{C}{\ep\gamma |\xi|^2}\leq \frac{C\gamma }{\ep  |\xi|}.
\end{align*}
From the last three estimates, we infer that
\begin{align*}
\ep II_1\leq  \frac{C\gamma }{ |\xi|}.
\end{align*}
Similarly, one gets the same upper bound for 
$\ep II_2$ and thus
$$\ep II\leq  \frac{C\gamma }{ |\xi|},$$ which together with \eqref{aepsplit:aeclosetofront-xi} and \eqref{Iestimate:aeclosetofront-xi} gives \eqref{aeclosetofront-xi-oneside}. 
The lower bound for $ a_\ep(\xi;x) $ is obtained with a similar argument. Estimate \eqref{aeclosetofront-xi} follows. 

\medskip 
This completes the proof of Lemma \ref{lem:ae-updates}. 
\qed

\section{Proof of Lemma \ref{lem:ae near front}}\label{sec:estimatesbis}

 Lemma \ref{lem:ae near front} follows from the next  four lemmas. 
 
\begin{lem}\label{lem:ae-near}
Let $d$ be as in Definition \ref{defn:extension}.
If $\abs{d(t,x)}<\rho$, then there is $C >0$ such that
\[
\abs{a_{\ep}\(\frac{d(t,x)}{\ep};t,x\)- \left(\ep \mathcal{I}_n\left[\phi\(\frac{d(t,\cdot)}{\ep}\) \right](x) - C_n \mathcal{I}_1[\phi]\(\frac{d(t,x)}{\ep}\) \right)} 
	\leq  \frac{C\ep}{\gamma}.
\]
\end{lem}

\begin{lem}\label{lem:ae-away}
Let $d$ be as in Definition \ref{defn:extension}.
If $\abs{d(t,x)}>\rho$, then there is $C>0$ such that
\[
\abs{a_{\ep}\(\frac{d(t,x)}{\ep};t,x\) } 
	\leq  \frac{C\ep}{\rho}.
\]
\end{lem}

\begin{lem}\label{lem:I_n-away}
Let $d$ be as in Definition \ref{defn:extension}.
If $\abs{d(t,x)}>\rho$, then there is $C >0$ such that
\[
\abs{ \mathcal{I}_n\left[ \phi \( \frac{d(t,\cdot)}{\ep}\) \right](x) } 
	\leq  \frac{C}{\rho}.
\]
\end{lem}

\begin{lem}\label{lem:I_1-away}
Let $d$ be as in Definition \ref{defn:extension}.
If $\abs{d(t,x)}>\rho$, then there is $C >0$ such that
\[
\abs{ \mathcal{I}_1[\phi]\(\frac{d(t,x)}{\ep}\) } 
	\leq  \frac{C\ep}{\rho}.
\]
\end{lem}

\subsection{Proof of Lemma \ref{lem:ae-near}}

First, we write $a_{\ep} = a_{\ep}\(d(t,x)/\ep ; t,x\)$ as
\begin{equation}\label{eq:split-for-ae-laplac}
\begin{aligned}
a_{\ep}
	&=\operatorname{P.V.} \int_{\R^n} \( \phi\(\frac{d(t,x+\ep z)}{\ep}\) - \phi\(\frac{d(t,x)}{\ep}\) \)\, \frac{dz}{\abs{z}^{n+1}} \\
	&\quad -\operatorname{P.V.}\int_{\R^n} \(\phi\( \frac{d(t,x)}{\ep} + \nabla d(t,x) \cdot z\) - \phi\(\frac{d(t,x)}{\ep}\)\) \, \frac{dz}{\abs{z}^{n+1}}\\
	&\quad- \int_{|z|>\frac{\gamma}{\ep}} \( \phi\(\frac{d(t,x+\ep z)}{\ep}\)  -\phi\( \frac{d(t,x)}{\ep} + \nabla d(t,x) \cdot z\) \) \, \frac{dz}{\abs{z}^{n+1}}.
\end{aligned}
\end{equation}
Recalling that $0 < \phi < 1$,  we first estimate
\begin{align*}
\int_{|z|>\frac{\gamma}{\ep}} \left| \phi\(\frac{d(t,x+\ep z)}{\ep}\)  -\phi\( \frac{d(t,x)}{\ep} + \nabla d(t,x) \cdot z\)\right| \, \frac{dz}{\abs{z}^{n+1}}
	\leq 2 \int_{|z|>\frac{\gamma}{\ep}}  \frac{dz}{\abs{z}^{n+1}} =  \frac{C\ep}{\gamma}. 
\end{align*}
Since $e = \nabla d(t,x)$ is a unit vector when $\abs{d(t,x)} \leq \rho$, by applying Lemma \ref{lem:one to n} to the second integral in the right-hand side of  \eqref{eq:split-for-ae-laplac} and a change of variable in the first integral, we obtain
\begin{align*}
a_{\ep}
	&= \ep \operatorname{P.V.} \int_{\R^n} \( \phi\(\frac{d(t,x+z)}{\ep}\) - \phi\(\frac{d(t,x)}{\ep}\) \)\, \frac{dz}{\abs{z}^{n+1}} 
	 -C_n \mathcal{I}_1[\phi]\(\frac{d(t,x)}{\ep}\) + O\left( \frac{\ep}{\gamma}\right)\\
	&= \ep \mathcal{I}_n\left[\phi\(\frac{d(t,\cdot)}{\ep}\) \right](x) - C_n \mathcal{I}_1[\phi]\(\frac{d(t,x)}{\ep}\)
	+ O\left( \frac{\ep}{\gamma}\right).
\end{align*}

\subsection{Proof of Lemma \ref{lem:ae-away}}
The lemma is an immediate consequence of \eqref{aL2estimatefar}. 

\subsection{Proof of Lemma \ref{lem:I_n-away}}
We drop the notation in $t$. For $c=\|\nabla d\|_\infty$, we write
\begin{align*}
\mathcal{I}_n\left[ \phi \( \frac{d(\cdot)}{\ep}\) \right](x)
	&=\operatorname{P.V.} \int_{\R^n} \(\phi \( \frac{d(x+z)}{\ep}\) - \phi \( \frac{d(x)}{\ep}\) \) \frac{dz}{\abs{z}^{n+1}}\\
	&= \int_{\abs{z}< \frac{\rho}{2c}} \(\phi \( \frac{d(x+z)}{\ep}\) - \phi \( \frac{d(x)}{\ep}\) - \dot{\phi} \( \frac{d(x)}{\ep}\) \frac{\nabla d(x)}{\ep} \cdot z  \) \frac{dz}{\abs{z}^{n+1}}\\
	&\quad + \int_{\abs{z}> \frac{\rho}{2c}} \(\phi \( \frac{d(x+z)}{\ep}\) - \phi \( \frac{d(x)}{\ep}\) \) \frac{dz}{\abs{z}^{n+1}}\\
	&=: I + II. 
\end{align*}
Looking first at the long-range interactions, we use $0 < \phi< 1$ to estimate
\[
\abs{II} \leq  2 \int_{\abs{z}> \frac{\rho}{2c}} \frac{dz}{\abs{z}^{n+1}} = \frac{C}{\rho}.
\]
For the short-range interactions, fix $z$ such that $\abs{z}<\rho/(2c)$. Then, by Taylor's theorem, 
\begin{align*}
&\phi \( \frac{d(x+z)}{\ep}\) - \phi \( \frac{d(x)}{\ep}\) - \dot{\phi} \( \frac{d(x)}{\ep}\) \frac{\nabla d(x)}{\ep} \cdot z \\
	&=\frac12 \( \ddot{\phi} \( \frac{d(x+\tau z)}{\ep}\) \frac{\nabla d(x + \tau z) \otimes \nabla d(x+\tau z)}{\ep^2} +  \dot{\phi} \( \frac{d(x+\tau z)}{\ep}\) \frac{D^2d(x+\tau z)}{\ep} \) z \cdot z
\end{align*}
for some $0 \leq \tau \leq 1$. 
Since
\[
\abs{\frac{d(x+\tau z)}{\ep}} 
	\geq \frac{\abs{d(x)}}{\ep} - \frac{c\abs{z}}{\ep}
	> \frac{\rho}{\ep} - \frac{\rho}{2\ep} = \frac{\rho}{2\ep},
\]
we can apply \eqref{eq:asymptotics for phi dot} to estimate
\[
\abs{\dot{\phi} \( \frac{d(x+\tau z)}{\ep}\)} \leq \frac{C \ep^2}{\rho^2}
\quad \hbox{and} \quad
\abs{\ddot{\phi} \( \frac{d(x+\tau z)}{\ep}\)} \leq \frac{C \ep^2}{\rho^2}.
\]
Therefore, with the regularity of $d$, 
\[
\abs{\phi \( \frac{d(x+z)}{\ep}\) - \phi \( \frac{d(x)}{\ep}\) - \dot{\phi} \( \frac{d(x)}{\ep}\) \frac{\nabla d(x)}{\ep} \cdot z}
	\leq  \frac{C}{\rho^2} \abs{z}^2. 
\]
Thus,  we have that
\[
\abs{I} \leq  \frac{C}{\rho^2} \int_{\abs{z}< \frac{\rho}{2c}} \abs{z}^2 \frac{dz}{\abs{z}^{n+1}} = \frac{C}{\rho}.
\]
The conclusion follows by combining the estimates for $I$ and $II$. 

\subsection{Proof of Lemma \ref{lem:I_1-away}}
From \eqref{eq:standing wave}, the periodicity of $W$  and $W'(0)=0$, we see that 
$$ \mathcal{I}_1[\phi]\(\frac{d(t,x)}{\ep}\)=\frac{1}{C_n}W'\left(\phi\(\frac{d(t,x)}{\ep}\)\right)=\frac{1}{C_n}W'\left(\tilde\phi\(\frac{d(t,x)}{\ep}\)\right)=O\left(\tilde\phi\(\frac{d(t,x)}{\ep}\)\right),$$
where, as usual, $\tilde\phi(\xi)=\phi(\xi)-H(\xi)$ and $H$ is the Heaviside function. Therefore, by \eqref{eq:asymptotics for phi},
$$\left|\mathcal{I}_1[\phi]\(\frac{d(t,x)}{\ep}\)\right|\leq \frac{C\ep}{|d(t,x)|}\leq \frac{C\ep}{\rho}.$$

\medskip
The proof of  Lemma \ref{lem:ae near front} is then completed.  \qed

\section{Proof of Lemma \ref{lem:ae psi estimate} }\label{lem:ae psi estimatesec}

Lemma \ref{lem:ae psi estimate} is a consequence of the following three lemmas. 

\begin{lem}\label{lem:I_1-psi}
Let $d$ be as in Definition \ref{defn:extension}.
If $\abs{d(t,x)}< \rho/2$, then there is $C >0$ such that
\[
\abs{\ep \mathcal{I}_n\left[ \psi \( \frac{d(t,\cdot)}{\ep};t,\cdot\) \right](x) 
	- C_n \mathcal{I}_1[\psi\(\cdot;t,x\)]\(\frac{d(t,x)}{\ep}\)  } \leq   \frac{C}{|\ln\ep|^\frac12}.
\]
\end{lem}

\begin{lem}\label{lem:I_n-awaypsi}
Let $d$ be as in Definition \ref{defn:extension}.
If $\abs{d(t,x)}>\rho/2$, then there is $C >0$ such that
\[
\abs{ \mathcal{I}_1[\psi\(\cdot;t,x\)]\(\frac{d(t,x)}{\ep}\) } 
	\leq  \frac{C}{|\ln\ep|\rho}.
\]
\end{lem}

\begin{lem}\label{lem:ae-awaypsi}
Let $d$ be as in Definition \ref{defn:extension}.
If $\abs{d(t,x)}>\rho/2$, then there is $C >0$ such that
\[
\abs{\ep \mathcal{I}_n\left[ \psi \( \frac{d(t,\cdot)}{\ep};t,\cdot\) \right](x) 
	  } \leq \frac{C}{|\ln\ep|^\frac12\rho}.
\]
\end{lem}

\subsection{Proof of Lemma \ref{lem:I_1-psi}}
For simplicity in notation, we drop the dependence on  $t$. 
By Lemma \ref{lem:one to n} and the fact that $|\nabla d(x)|=1$ if $\abs{d(x)}<\rho$, we have 
\begin{align*}
\ep& \mathcal{I}_n\left[ \psi \( \frac{d(\cdot)}{\ep};\cdot\) \right](x) 
	- C_n \mathcal{I}_1[\psi\(\cdot;x\)]\(\frac{d(x)}{\ep}\) \\&
	=\operatorname{P.V.} \int_{\R^n}\left(\psi\(\frac{d(x+\ep z)}{\ep};x+\ep z\)-\psi\(\frac{d(x)}{\ep}+\nabla d(x)\cdot z;x\)\right)\frac{dz}{|z|^{n+1}}\\&
	=\operatorname{P.V.} \int_{\R^n}\left(\psi\(\frac{d(x+\ep z)}{\ep};x+\ep z\)-\psi\(\frac{d(x)}{\ep}+\nabla d(x)\cdot z;x+\ep z \)\right)\frac{dz}{|z|^{n+1}}\\&
	\quad+\operatorname{P.V.} \int_{\R^n}\left(\psi\(\frac{d(x)}{\ep}+\nabla d(x)\cdot z;x+\ep z \)-\psi\(\frac{d(x)}{\ep}+\nabla d(x)\cdot z;x\)\right)\frac{dz}{|z|^{n+1}}\\&
=:I+II.
	\end{align*}
First, let us estimate  $I$.  We will show that 
\begin{equation}\label{lemmaI_1-psiI}|I|\leq \frac{C}{|\ln\ep|}.
\end{equation}
Notice that 
\begin{equation}\label{I_1-psilemsecondorder}\abs{\psi\(\frac{d(x+\ep z)}{\ep};x+\ep z\)-\psi\(\frac{d(x)}{\ep}+\nabla d(x)\cdot z;x+\ep z\)}\leq C\|\dot\psi\|_\infty\ep|z|^2,
\end{equation}
so we can actually write 
$$I= \int_{\R^n}\left(\psi\(\frac{d(x+\ep z)}{\ep};x\)-\psi\(\frac{d(x)}{\ep}+\nabla d(x)\cdot z;x\)\right)\frac{dz}{|z|^{n+1}}.$$
From here, we split
\begin{align*}I&=\int _{\{|z|<\ep^{-\frac12}\}}(\ldots)
	+\int _{\{{|z|>\ep^{-\frac12}}\}}(\ldots)
	=:I_1+I_2. 
	\end{align*}
By \eqref{I_1-psilemsecondorder} and \eqref{psiLinfinityasymptoticsfar} for $\dot\psi$
\begin{align*}|I_1|\leq  C\ep\|\dot\psi\|_\infty\int_{\{|z|< \ep^{-\frac12}\}}\frac{dz}{|z|^{n-1}}\leq \frac{C\ep}{\ep^\frac12|\ln\ep|}\ep^{-\frac12}=\frac{C}{|\ln\ep|},
\end{align*}
	and  from  \eqref{psiLinfinityasymptoticsfar} for $\psi$,
\begin{equation*}\label{I_5case1I_1psilem}
	|I_2|\leq 2\|\psi\|_\infty \int_{\{|z|>\ep^{-\frac12}\}}\frac{dz}{|z|^{n+1}}\leq C\frac{\ep^{\frac12}}{\ep^\frac12|\ln\ep|}=\frac{C}{|\ln\ep|}.
\end{equation*}
This proves \eqref{lemmaI_1-psiI}. 

Next, using that
 $ \operatorname{P.V.}\int_{\{\abs{z}<|\ln\ep|^{\frac12}\}} \nabla_x\psi\(\frac{d(x)}{\ep}; x\) \cdot z \frac{dz}{\abs{z}^{n+1}}=0$,   we write 
\begin{align*}
II&=
	\int_{\{\abs{z}<|\ln\ep|^{\frac12}\}}\left\{\psi\(\frac{d(x)}{\ep}+\nabla d(x)\cdot z;x+
	\ep z \)-\psi\(\frac{d(x)}{\ep}+\nabla d(x)\cdot z;x \)\right.
	\\& \quad \left.-\nabla_x\psi\(\frac{d(x)}{\ep}+\nabla d(x)\cdot z; x\)\cdot (\ep z)\right\}\frac{dz}{|z|^{n+1}}\\&\quad
+ \int_{\{\abs{z}<|\ln\ep|^{\frac12}\}}\left\{ \nabla_x\psi\(\frac{d(x)}{\ep}+\nabla d(x)\cdot z; x\) - \nabla_x\psi\(\frac{d(x)}{\ep}; x\) \right\}\cdot (\ep z)\frac{dz}{|z|^{n+1}}\\&\quad
+ \int_{\{\ln\ep|^{\frac12}<\abs{z}<\frac{\rho}{2\ep}\}}\left\{\psi\(\frac{d(x)}{\ep}+\nabla d(x)\cdot z;x+\ep z\)-\psi\(\frac{d(x)}{\ep}+\nabla d(x)\cdot z;x\)\right\}\frac{dz}{|z|^{n+1}}\\&
\quad + \int_{\{\abs{z}>\frac{\rho}{2\ep}\}}\left\{\psi\(\frac{d(x)}{\ep}+\nabla d(x)\cdot z;x+\ep z\)-\psi\(\frac{d(x)}{\ep}+\nabla d(x)\cdot z;x\)\right\}\frac{dz}{|z|^{n+1}}\\&
=: II_1+II_2+II_3 +II_4. 
\end{align*}
By   \eqref{psiLinfinityasymptoticsfar} for $D^2_x\psi$, 
\begin{align*}
|II_1|\leq C\|D^2_x\psi\|_\infty\ep^2\int_{\{\abs{z}<|\ln\ep|^{\frac12}\}}\frac{dz}{|z|^{n-1}}\leq\frac{C}{\ep^2|\ln\ep|}\ep^2 |\ln\ep|^{\frac12}=\frac{C}{|\ln\ep|^\frac12}.
\end{align*}
By   \eqref{psiLinfinityasymptoticnewclosefront} for $\nabla_x\dot\psi$, 
\begin{align*}
|II_2|\leq  \|\nabla_x\dot\psi(\cdot;x)\|_\infty\ep\int_{\{\abs{z}<|\ln\ep|^{\frac12}\}}\frac{dz}{|z|^{n-1}}\leq\frac{C}{\ep|\ln\ep|}\ep|\ln\ep|^{\frac12}=\frac{C}{|\ln\ep|^\frac12}.
\end{align*}
Next, since $|d(x)|<\rho/2$, for $\abs{z}<\rho/(2\ep)$ we have that $|d(x+\ep z)|\leq \rho$, so that from  \eqref{psiLinfinityasymptoticnewclosefront} for $\psi$, 
\begin{align*}
|II_3|\leq 2\sup_{|d(y)|\leq\rho} \|\psi(\cdot,y)\|_\infty \int_{\{|z|>|\ln\ep|^\frac12\}} \frac{dz}{|z|^{n+1}}\leq\frac{C}{|\ln\ep|^\frac12}.
\end{align*}
Finally, by \eqref{psiLinfinityasymptoticsfar} for $\psi$, 
\begin{align*}
|II_4|\leq 2 \|\psi\|_\infty \int_{\{|z|>\frac{\rho}{2\ep}\}} \frac{dz}{|z|^{n+1}}\leq\frac{C}{\ep^\frac12|\ln\ep|}\frac{\ep}{\rho}=\frac{C\ep^\frac12}{\rho|\ln\ep|}.
\end{align*}

\noindent This proves that 
\begin{equation}\label{lemmaI_1-psiII}|II|\leq \frac{C}{|\ln\ep|^\frac12}.
\end{equation}
From \eqref{lemmaI_1-psiI} and  \eqref{lemmaI_1-psiII}, we have Lemma \ref{lem:I_1-psi}. 

\subsection{Proof of Lemma \ref{lem:I_n-awaypsi}}

From \eqref{eq:linearized wave}, \eqref{eq:asymptotics for phi}, \eqref{eq:asymptotics for phi dot},  \eqref{aL2estimatefar}, \eqref{abarestimate} and \eqref{psiasymptoticsfar},
for $|\xi|\geq 1$,  
\begin{align*}
\abs{ \mathcal{I}_1[\psi\(\cdot;t,x\)]\(\xi\) }&\leq C\left|\psi (\xi)\right|+ \frac{|a_{\ep}\(\xi;t,x\)|}{\ep \abs{\ln \ep}}+C|1+\bar{a}_{\ep}(t,x)|\dot\phi(\xi)+C|\phi(\xi)|\\&
\leq \frac{C}{\ep|\ln\ep||\xi|}.
\end{align*}
In particular, 
\[
\abs{ \mathcal{I}_1[\psi\(\cdot;t,x\)]\(\frac{d(t,x)}{\ep}\) } 
	\leq  \frac{C}{|\ln\ep||d(t,x)|}\leq  \frac{C}{|\ln\ep|\rho}.
\]

\subsection{Proof of Lemma \ref{lem:ae-awaypsi}}
For convenience, we drop the dependence on  $t$ and begin by writing 
\begin{equation}\label{eq:ae-awaypsi-split}
\begin{aligned}
\ep \mathcal{I}_n\left[ \psi \( \frac{d(\cdot)}{\ep};\cdot\) \right](x)&=
\ep \mathcal{I}_n\left[ \psi \( \frac{d(\cdot)}{\ep};\cdot\) \right](x)-|\nabla d(x)|C_n \mathcal{I}_1[\psi\(\cdot;x\)]\(\frac{d(x)}{\ep}\) \\&\quad+|\nabla d(x)| C_n \mathcal{I}_1[\psi\(\cdot;x\)]\(\frac{d(x)}{\ep}\) . 
\end{aligned}
\end{equation}
 With Lemma  \ref{lem:one to n} and 
 as in the proof of  Lemma \ref{lem:I_1-psi}, we obtain
 \begin{align*}
\ep& \mathcal{I}_n\left[ \psi \( \frac{d(\cdot)}{\ep};\cdot\) \right](x) 
	 -|\nabla d(x)|C_n \mathcal{I}_1[\psi\(\cdot;x\)]\(\frac{d(x)}{\ep}\) \\&
	=\operatorname{P.V.} \int_{\R^n}\left(\psi\(\frac{d(x+\ep z)}{\ep};x+\ep z\)-\psi\(\frac{d(x)}{\ep}+\nabla d(x)\cdot z;x+\ep z \)\right)\frac{dz}{|z|^{n+1}}\\&
	\quad+\operatorname{P.V.} \int_{\R^n}\left(\psi\(\frac{d(x)}{\ep}+\nabla d(x)\cdot z;x+\ep z \)-\psi\(\frac{d(x)}{\ep}+\nabla d(x)\cdot z;x\)\right)\frac{dz}{|z|^{n+1}}\\&
=:I+II,
\end{align*}
with 
$$|I|\leq \frac{C}{|\ln\ep|}.$$
Next, using that
 $ \operatorname{P.V.}\int_{\{\abs{z}<|\ln\ep|^{-\frac12}\}} \nabla_x\psi\(\frac{d(x)}{\ep}; x\) \cdot z \frac{dz}{\abs{z}^{n+1}}=0$,   we write, for $c=\|\nabla d\|_\infty$, 
\begin{align*}
II&=\int_{\{\abs{z}<|\ln\ep|^{-\frac12}\}}\left\{\psi\(\frac{d(x)}{\ep}+\nabla d(x)\cdot z;x+
	\ep z \)-\psi\(\frac{d(x)}{\ep}+\nabla d(x)\cdot z;x \)\right.
	\\& \quad \left.-\nabla_x\psi\(\frac{d(x)}{\ep}+\nabla d(x)\cdot z; x\)\cdot (\ep z)\right\}\frac{dz}{|z|^{n+1}}\\&\quad
+ \int_{\{\abs{z}<|\ln\ep|^{-\frac12}\}}\left\{ \nabla_x\psi\(\frac{d(x)}{\ep}+\nabla d(x)\cdot z; x\) - \nabla_x\psi\(\frac{d(x)}{\ep}; x\) \right\}\cdot (\ep z)\frac{dz}{|z|^{n+1}}\\&\quad
+ \int_{\{|\ln\ep|^{-\frac12}<\abs{z}<\frac{\rho}{4c\ep}\}}\left\{\psi\(\frac{d(x)}{\ep}+\nabla d(x)\cdot z;x+\ep z\)-\psi\(\frac{d(x)}{\ep}+\nabla d(x)\cdot z;x\)\right\}\frac{dz}{|z|^{n+1}}\\&\quad
+ \int_{\{\abs{z}>\frac{\rho}{4c\ep}\}}\left\{\psi\(\frac{d(x)}{\ep}+\nabla d(x)\cdot z;x+\ep z\)-\psi\(\frac{d(x)}{\ep}+\nabla d(x)\cdot z;x\)\right\}\frac{dz}{|z|^{n+1}}\\&
=: II_1+II_2+II_3+II_4. 
\end{align*}
By   \eqref{psiLinfinityasymptoticsfar} for $D^2_x\psi$, 
\begin{align*}
|II_1|\leq C\|D^2_x\psi\|_\infty\ep^2\int_{\{\abs{z}<|\ln\ep|^{-\frac12}\}}\frac{dz}{|z|^{n-1}}\leq\frac{C}{\ep^2|\ln\ep|}\ep^2|\ln\ep|^{-\frac12}=\frac{C}{|\ln\ep|^\frac32}.
\end{align*}
By  \eqref{psiLinfinityasymptoticsfar} for $\nabla_x\dot\psi$, 
\begin{align*}
|II_2|\leq  \|\nabla_x\dot\psi(\cdot;x)\|_\infty\ep\int_{\{\abs{z}<|\ln\ep|^{-\frac12}\}}\frac{dz}{|z|^{n-1}}\leq\frac{C}{\ep}\ep |\ln\ep|^{-\frac12}= \frac{C}{|\ln\ep|^\frac12}.
\end{align*}
Next,  note that if $|d(x)|>\rho/2$ and $|z|<\frac{\rho}{4c\ep}$, then $|\frac{d(x)}{\ep}+\nabla d(x)\cdot z|>\frac{\rho}{4\ep}$. Therefore, by \eqref{psiasymptoticsfar},
\begin{align*}
|II_3|\leq \frac{C}{\ep|\ln\ep|}\frac{\ep}{\rho} \int_{\{|z|>|\ln\ep|^{-\frac12}\}} \frac{dz}{|z|^{n+1}}=\frac{C}{|\ln\ep|^\frac12 \rho}.
\end{align*}
Finally, from \eqref{psiLinfinityasymptoticsfar} for $\psi$, 
\begin{align*}
|II_4|\leq 2\|\psi\|_\infty\int_{\{\abs{z}>\frac{\rho}{4c\ep}\}}\frac{dz}{|z|^{n+1}}\leq \frac{C\ep^\frac12}{|\ln\ep|\rho}. 
\end{align*}
We conclude that 
 $$\abs{ \mathcal{I}_n\left[ \psi \( \frac{d(\cdot)}{\ep};\cdot\) \right](x)-|\nabla d(x)|C_n \mathcal{I}_1[\psi\(\cdot;x\)]}\leq\frac{C}{|\ln\ep|^\frac12\rho}. $$
 Moreover, by Lemma \ref{lem:I_n-awaypsi}, 
 $$|\nabla d(x)| \abs{C_n \mathcal{I}_1[\psi\(\cdot;x\)]}\leq\frac{C}{|\ln\ep|\rho}. $$
 Recalling recalling \eqref{eq:ae-awaypsi-split}, the lemma then follows. 

\medskip
The proof of Lemma \ref{lem:ae psi estimate} is then completed.  \qed
 
 \section*{Acknowledgements}
 
 The first author has been supported by the NSF Grant DMS-2155156 ``Nonlinear PDE methods in the study of interphases.'' 
The second author has been supported by the Australian Laureate Fellowship FL190100081 ``Minimal surfaces, free boundaries and partial differential equations.''
Both authors acknowledge the support of NSF Grant DMS RTG 18403.
 


\end{document}